\documentclass{amsart}
\usepackage{graphicx} 
\usepackage{tikz}
\usepackage{tikz-cd} 
\usepackage{amsmath}
\usepackage{xcolor}
\usepackage{amsfonts}
\usepackage{amssymb}
\usepackage{amsthm}
\usepackage{mathrsfs}
 \usepackage{enumerate}
 \usepackage{enumitem}
\usepackage[margin = 1.25 in]{geometry} 
\usepackage{bbm}

\usepackage[citecolor=teal,linkcolor=blue,urlcolor=cyan]{hyperref} 
\hypersetup{colorlinks}

\numberwithin{equation}{section}
\usepackage{bm}

\usepackage{textcomp}
\renewcommand{\div}{\mathrm{div}}
\newcommand{\spann}{\mathrm{span}}
\newcommand{\Harm}{\mathrm{Harm}}

\newcommand{\Sym}{\mathrm{Sym}}
\newcommand{\Hom}{\mathrm{Hom}}
\newcommand{\diag}{\mathrm{diag}}

\theoremstyle{plain}

\newtheorem{theorem}{Theorem}[section]
\newtheorem{proposition}[theorem]{Proposition}
\newtheorem{lemma}[theorem]{Lemma}
\newtheorem{corollary}[theorem]{Corollary}

\newtheorem{example}{Example}[section]

\newtheorem{definition}[theorem]{Definition}
 
\newtheorem{remark}[theorem]{Remark}

\newcommand{\PD}{\mathcal{PD}}
\newcommand{\Ein}{\mathbf{Ein}^{1,2}}
\newcommand{\Einh}{\widehat{\mathbf{Ein}}{}^{1,2}}
\renewcommand{\H}{\mathbf{H}^{2,2}}
\newcommand{\Hh}{\widehat{\mathbf{H}}^{2,2}}
\newcommand{\R}{\mathbb{R}}
\newcommand{\E}{\mathcal{E}} 
\newcommand{\C}{\mathbb{C}}
\newcommand{\Z}{\mathbb{Z}}
\renewcommand{\P}{\mathbf{P}}
\newcommand{\RP}{\mathbb{R}\P}
\newcommand{\CP}{\mathbb{C}\P}
\newcommand{\K}{\mathcal{K}}
\newcommand{\X}{\mathbb{X}} 
\newcommand{\Aff}{\mathbf{Aff}}

\newcommand{\T}{\mathbf{T}}
\renewcommand{\L}{\mathbf{L}}
\newcommand{\I}{\mathrm{I}}
\newcommand{\N}{\mathbf{N}}
\renewcommand{\S}{\mathbf{S}}
\newcommand{\D}{\mathbf{D}}
\newcommand{\Max}{\mathcal{C}_{\text{Max}}(\H)}

\newcommand{\der}[1]{\frac{\partial}{\partial #1}}
\newcommand{\id}{\mathrm{id}}
\newcommand{\zbar}{\overline{z}}
\newcommand{\del}{\partial}

\newcommand{\II}{\mathrm{II}}
\renewcommand{\r}{\mathbf{r}}

\newcommand{\Q}{\mathbf{Q}}
\newcommand{\dz}{\mathrm{d}z}
\newcommand{\dzb}{\mathrm{d}\overline{z}}

\newcommand{\End}{\mathsf{End}}
\newcommand{\g}{\mathfrak{g}}
\newcommand{\so}{\mathfrak{so}}
\newcommand{\spl}{\mathfrak{sp}}
\newcommand{\GL}{\mathsf{GL}}
\newcommand{\U}{\mathsf{U}}
\newcommand{\SL}{\mathsf{SL}}
\newcommand{\Sp}{\mathsf{Sp}}
\newcommand{\SO}{\mathsf{SO}}
\renewcommand{\O}{\mathsf{O}}
\renewcommand{\leq}{\leqslant}
\renewcommand{\geq}{\geqslant}
\newcommand{\gl}{\mathfrak{gl}}
\newcommand{\pr}{\mathsf{pr}}

\newcommand{\G}{\mathsf{G}}

\title[Polynomial Maximal Surfaces in \texorpdfstring{$\H$}{H22} and spherical polygons]{Polynomial Maximal Surfaces in \texorpdfstring{$\H$}{H22}\\and spherical polygons}
\author{Parker Evans and Alex Moriani}
\address{Parker Evans: Department of Mathematics, Washington University in St. Louis, St. Louis, MO}
\email{evans.p@wustl.edu}
\address{Alex Moriani: Dipartimento di Matematica \emph{Giuseppe Peano}, Università degli Studi di Torino, Turin, Italy}
\email{alex.moriani@unito.it}
\date{}
\begin{document}

\begin{abstract}
We study the moduli space of complete maximal surfaces in pseudo-hyperbolic space $\H$ with polygonal asymptotic boundary a lightlike $(n+4)$-gon in the Einstein universe $\Ein$, up to the action of $\SO_0(2,3)$. Importantly, this space $\mathcal{M}_{n+4}(\H)$ is disconnected for $n\geq 2$. We introduce a \emph{turning invariant} $\Theta$ that classifies the components in terms of the geometry of the boundary polygons in the Einstein universe $\Ein$, which is related to \emph{spherical polygons} in $\S^2$. We provide a geometric parametrization of the moduli space $\mathcal{M}_{n+4}(\H)$ via holomorphic data, namely a pair of polynomials $\alpha$, $\beta$. These polynomials appear in relation to certain components of the complexified second fundamental form of the corresponding  surface. 
We then relate the holomorphic data $\alpha$, $\beta$ to the turning invariant $\Theta$. This leads to a characterization of the \emph{Hitchin component} of polygonal maximal surfaces studied by Tamburelli-Wolf. 
\end{abstract}

\maketitle

\vspace{-3ex}
\tableofcontents 

\newpage 

\setcounter{section}{-1}

\section{Introduction}

\subsection{Background}
Minimal surfaces in 4-dimensional space forms $\S^4$, $\mathbf{E}^4$, $\mathbf{H}^4$ have become a well-investigated class of objects in differential geometry. 
In this work, we study surfaces in the \emph{pseudo-Riemannian} space form $\H$, pseudo-hyperbolic space of signature $(2,2)$. Concretely, $\H$ can be seen as the space of negative-definite lines in pseudo-Euclidean space $\R^{2,3}$. We consider immersed \emph{spacelike} surfaces $\Sigma \subset \H$, those with positive-definite tangent planes, which in addition are critical points of the area functional under compact deformations. Such surfaces are called \emph{maximal}, as the normal bundle is timelike and hence local deformations decrease the area \cite[Corollary 3.24]{LTW24}. 

Maximal surfaces in $\mathbf{H}^{2,q}$ more generally, for $q\geq 2$, have become popular objects in higher rank Teichm\"uller theory and have appeared in many recent works \cite{Bar10, CTT19, TW24, LTW24,LT23, DL24, MV25, Mor25, Via25, Lem26}. Furthermore, the higher-dimensional case of $p$-dimensional maximal submanifolds in $\mathbf{H}^{p,q}$, for $p\geq 3$, has also been recently studied in \cite{SST23, MT25, Via25}.
Returning to the case $p=2$, every complete maximal surface $\Sigma$ in $\mathbf{H}^{2,q}$ is an embedded disk. The surface $\Sigma$ is also compactified by its \emph{asymptotic boundary}, which is a topological circle $\Lambda$ in the Einstein universe $\mathbf{Ein}^{1,q}$, the frontier of pseudo-hyperbolic space $\mathbf{H}^{2,q}$ in projective space $\RP^{2+q}$. Given this structure, one is naturally led to consider the \emph{asymptotic Plateau problem}. 
A comprehensive solution was given recently by Labourie-Toulisse-Wolf in \cite{LTW24}: they show that every complete maximal surface $\Sigma \subset \mathbf{H}^{2,q}$ has asymptotic boundary $\Lambda \subset \mathbf{Ein}^{1,q}$ that is a \emph{semi-positive loop} (see \textsection\ref{subsec:SemiPositive}), and conversely, every semi-positive loop in $\mathbf{Ein}^{1,q}$ uniquely bounds a complete maximal surface in $\mathbf{H}^{2,q}$. 
Henceforth, all surfaces in $\mathbf{H}^{2,q}$ will be embedded and complete unless otherwise stated. 

Given the uniqueness results of \cite{LTW24}, a natural problem is to describe all \emph{polygonal} maximal surfaces, namely all maximal surfaces that bound a semi-positive loop achieved by concatenation of projective line segments. 
Recently, the second named author determined some constraints for such surfaces. In particular, a maximal surface $\Sigma$ in $\mathbf{H}^{2,q}$ has polygonal boundary exactly when its induced metric $g_{\Sigma}$ has finite total curvature, or equivalently, when $g_{\Sigma}$ is asymptotically flat \cite{Mor25}.

We can also change our perspective as follows. We realize a maximal surface $\Sigma$ as the image of a conformal embedding $\sigma: X \hookrightarrow \mathbf{H}^{2,q}$ from a Riemann surface $X$. In this light, $\sigma$ is a harmonic map. While the classical Hopf differential of $\sigma$ vanishes, we instead 
obtain a higher order Hopf differential of sorts, namely a holomorphic quartic differential $Q_4=Q_4(\sigma) \in H^0(X, \K_{X}^4)$. 
In \cite{Mor25}, it was proven that every polygonal maximal surface $\Sigma$ is biholomorphic to the complex plane $\C$ and has holomorphic quartic differential $Q_4$  that is \emph{polynomial} in a uniformizing coordinate.
Moreover, its asymptotic boundary $\Lambda = \partial_{\infty}\Sigma$ is an $(n+4)$-gon, where $n= \deg(Q_4)$. 

This polynomial-polygon theme now has many manifestations in the geometric analysis of maps $\sigma: \C \rightarrow M$ to various spaces $M$ \cite{HTTW95,DW15,Tam19, TW24, Eva24, GS25, MNS26}. Excluding the exceptional case of Dumas-Wolf, studying affine spheres in $\R^3$ with polynomial Pick differential $Q_3$, the ambient space $M$ under consideration in each of these works was (the projectivization of) a quadric hypersurface of pseudo-Euclidean space, such as $\mathbf{H}^2, \mathbf{H}^{2,1}, \mathbf{H}^{2,2}, \mathbf{S}^{2,4}, \mathbf{H}^3$. We continue this direction of studying harmonic maps to non-compact spaces with polygonal asymptotic boundaries. 

\subsection{Main results} 
In this work, we study the moduli space $\mathcal{M}_{n+4}(\H)$ of all polygonal maximal surfaces in $\H$ with asymptotic boundary an $(n+4)$-gon in the Einstein universe $\Ein$, up to the action of $\SO_0(2,3)$. 
An important feature in this case is that $\mathcal{M}_{n+4}(\H)$ is \emph{disconnected} for $n\geq 2$. On the other hand, we show that the moduli space $\mathcal{M}_{n+4}(\mathbf{H}^{2,q})$ of polygonal maximal surfaces in $\mathbf{H}^{2,q}$ is connected when $q\geq 3$, $n\geq 0$. Thus, the case $q=2$ has some additional subtlety. 

In \cite{TW24}, Tamburelli and Wolf described one component of the moduli space $\mathcal{M}_{n+4}(\H)$, which we shall call the \emph{Hitchin component}. These special polygonal maximal surfaces correspond to solutions to the $\so(2,3)$ affine Toda equations, and are associated to Higgs bundles in the Hitchin section for $G=\SO_0(2,3)$ over the complex plane. However, there are many more polygonal maximal surfaces beyond the Hitchin case, which one can think of as analogues of the Gothen components \cite{Got01} in the moduli space of $\SO_0(2,3)$-Higgs bundles for a closed Riemann surface $\Sigma$ (see \cite[Section 2.5]{CTT19}). 

\subsubsection{Parametrizing the moduli space}\label{subsubsec:ParamModuliSpace}
We now describe $\mathcal{M}_{n+4}(\H)$ in terms of holomorphic data. 

Let us start in the double cover $\Hh$, the negative quadric in $\R^{2,3}$. We denote by $\Q$ the pseudo-Riemannian metric of $\Hh$, or the complex bilinear extension when appropriate. Take $\Sigma \subset \Hh$ to be a polygonal maximal surface and $\sigma:\C\rightarrow \Sigma$ to be a conformal parametrization. By working upstairs, $\sigma$ obtains a canonically oriented normal bundle $\N$. The oriented rank two real bundle $\N$ has structure group $\SO(2)$ and thus holds a complex structure $J$. Hence, its complexification $\N_{}^{\C}$ decomposes as a sum of line bundles, the $\pm i$-eigenspaces of $J$, denoted $\N'$ and $\N''_{}$, respectively. 
We can then decompose the $(2,0)$-part of the second fundamental form $\II^{2,0} \in \Omega^0(\K^2\otimes \N^{\C}_{})$ as:
\[ \II^{2,0} = \alpha + \beta\qquad \text{where}\qquad \; \alpha\in \Omega^0(\K^2_{\C}\otimes \N_{}')\ , \  \beta \in \Omega^0(\K^2_{\C}\otimes \N_{}'') \ .\]
The quartic differential of $\sigma$ is given by 
$Q_4=\mathbf{Q}(\alpha, \beta)$, as each of $\N'_{}$, $\N''_{}$ are $\Q$-isotropic. 
In fact, $\alpha$, $\beta$ are holomorphic after promoting the complex line bundles $\N'_{}$, $\N_{}''$ to holomorphic line bundles $\mathcal{N}'$, $\mathcal{N}''$.

By \cite{Mor25}, the holomorphic quartic differential $Q_4(\sigma)$ of the polygonal surface $\Sigma$ is \emph{polynomial}, and hence 
$\alpha$, $\beta$ each have finitely many zeros.
Now, relative to an appropriate pair of dual holomorphic trivializations of $\mathcal{N}'$, $\mathcal{N}''$, 
the bundle-valued differentials $\alpha$, $\beta$ become true holomorphic quadratic differentials, denoted $\alpha_{QD}=A(z)\dz^2$, 
$\beta_{QD} = B(z)\dz^2$. 
With $\sigma$ still fixed, by requiring that 
the resulting differentials are \emph{polynomial}, meaning $A=A(z),B=B(z)$ are polynomial, the pair $(A\dz^2,B\dz^2)$ is well-defined up to the $\C^*$-action by
$\lambda \cdot (A,B) = (\lambda A, \lambda^{-1}B)$. 
Across all such conformal parametrizations $\sigma$, the pair $(\alpha_{QD},\beta_{QD})$ remains well-defined 
up to the diagonal action of the group $\Aff(\C)$ of affine transformations by pushforward and the aforementioned $\C^*$-action. 

We show this embedding data yields a parametrization of the moduli space of polygonal surfaces.

\begin{theorem}[The moduli space of polygonal surfaces]
\label{thm:A}
Let $\mathcal{M}_n(\mathcal{PD}^{o})$ and $\mathcal{M}_n(\mathcal{PD})$ be the moduli spaces of ordered (resp. unordered) pairs of holomorphic polynomial quadratic differentials $A\dz^2$, $B\dz^2$ such that the product $AB$ is of degree $n$, up to the described action of the group $\Aff(\C) \times \C^*$. 
\begin{enumerate}[label=(\roman*),itemsep=3pt]
    \item For any $\delta \in \mathcal{M}_n(\mathcal{PD}^{o})$, there is unique maximal surface $\Sigma^{o}(\delta)\subset \Hh$, up to $\SO_0(2,3)$, that admits a conformal parametrization $\sigma:\C \rightarrow \Sigma^{o}(\delta)$ such that 
    $[ (\alpha_{QD}(\sigma),\beta_{QD}(\sigma)) ] = \delta.$
    \item The map $\mathbf{
    \Sigma}^o: \mathcal{M}_n(\mathcal{PD}^{o}) \rightarrow \mathcal{M}_{n+4}(\Hh)$ given by 
    $\mathbf{\Sigma}^o(\delta) = [\,\Sigma^{o}(\delta)\,]$ is a homeomorphism. 
    \item The map $\mathbf{\Sigma}^{o}$ descends to a homeomorphism
    $\mathbf{\Sigma}: \mathcal{M}_n(\mathcal{PD})\rightarrow \mathcal{M}_{n+4}(\H)$ by 
    \[ \mathbf{\Sigma}(\,\overline{\delta}\,) =  [\,\mathbf{\Sigma}^{o}(\delta)\,] \,,\  \]
    where $\delta \in \mathcal{M}_n(\mathcal{PD}^{o})$ is any lift of $\overline{\delta} \in \mathcal{M}_n(\mathcal{PD})$.
\end{enumerate}
\end{theorem}
We briefly remark that the labeling of $\alpha$, $\beta$ is not intrinsic to a parametrized surface in $\H$, but depends on the choice of lift to $\Hh$; the two antipodal lifts $\Sigma^{o}$, $-{\Sigma}^{o}$ have oppositely oriented normals and thus the labeled decomposition of $\II^{2,0}$ differs for these two surfaces. It is for this reason that exchanging the polynomial differentials leads to the same surface downstairs in $\H$. 

We can also parametrize the moduli spaces $\mathcal{M}_{n+4}(\Einh)$ and $\mathcal{M}_{n+4}(\Ein)$ of semi-positive polygons in $\Einh$ and $\Ein$, respectively, up to the $\SO_0(2,3)$-action.  
Indeed, a consequence of the uniqueness and compactness theorems of \cite{LTW24} is that the asymptotic boundary map defines a homeomorphism $\partial_{\infty}: \mathcal{M}_{n+4}(\H)\rightarrow \mathcal{M}_{n+4}(\Ein)$, as well as a  
similar map for the double covers. Thus, Theorem \ref{thm:A} is equivalent to the following result.

\begin{theorem}(The moduli space of semi-positive polygons) \label{thm:B}
The maps $\partial_{\infty}\circ \mathbf{\Sigma}^{o}: \mathcal{M}_n(\mathcal{PD}^{o}) \rightarrow \mathcal{M}_{n+4}(\Einh)$ and $\partial_{\infty}\circ \mathbf{\Sigma} : \mathcal{M}_n(\mathcal{PD}) \rightarrow \mathcal{M}_{n+4}(\Ein)$ are homeomorphisms. 
\end{theorem}
Theorems \ref{thm:A} and \ref{thm:B} describe the moduli space of polygonal maximal surfaces in $\H$, or equivalently, their boundaries, via holomorphic quartic differentials and their refinements $\alpha$, $\beta$. 

As already noted, \cite{Mor25} proved that a parametrized maximal surface in $\mathbf{H}^{2,q}$ with polygonal asymptotic boundary has a polynomial quartic differential. However, the converse was not addressed. We obtain this result for $q=2$, completing the polynomial-polygon correspondence for $\H$.

\begin{corollary}[Polynomial $\leftrightarrow$ Polygon]\label{thm:C}
Let $\Sigma \subset \H$ be a maximal surface. Then $\Sigma$ has polygonal asymptotic boundary if and only if $\Sigma$ admits a conformal parametrization $\sigma: \C \rightarrow \Sigma$ with polynomial quartic differential $Q_4(\sigma)$.  
\end{corollary}

\subsubsection{Connections with spherical geometry}\label{Sec:IntroSpherical}
There is another layer to the story, perhaps of independent interest, purely involving the geometry of the sphere $\S^2$. To avoid writing $n+4$, in the discussion below, we take $N=n+4$. Let us recall that the double cover $\Einh$ is topologically $\S^1 \times \S^2$, and that semi-positive $N$-gons, for $N\geq 4$, are exactly graphs of $N$-piecewise isometries from $\S^1$ to $\S^2$ in this splitting. Thus, starting from a semi-positive $N$-gon $(p_i)_{i \in \Z_{N}}$ in $\Einh$, which carries a natural orientation, we obtain a 
\emph{spherical polygon} by projecting to $\S^2$. 
Here, the natural notion of polygon in this situation is  as follows.

\begin{definition}[Spherical polygons]\label{defn:sphericalIntro}
Let $N \geq 4$. A \textbf{spherical $N$-gon} $\Delta$ is a cyclically ordered $N$-tuple $(v_i)_{i \in \Z_{N}}$  of points in $\S^2$, containing no antipodal pairs, such that
\begin{enumerate}[label=(\roman*)]
    \item $\Delta$ has total length $2\pi$, namely $\sum_{i=1}^{N} d_{\S^2}(v_i,v_{i+1})=2\pi$, 
    \item No vertex $v_i$ lies on the minimizing geodesic segment between $v_{i-1}$ and $v_{i+1}$.
\end{enumerate}
\end{definition}
Here, we conflate the vertex set of the polygon with the cyclically oriented loop in $\S^2$ obtained as the concatenation of the unique minimizing geodesic segments between consecutive vertices.

Note that the second condition asserts that the polygon \emph{must} make a nonzero turn at each vertex $v_i$, but is permitted to head directly backwards along the same great circle.
Condition (i) says our polygons have flexible edge lengths, but rigid total length. Thus, Definition \ref{defn:sphericalIntro} differs meaningfully from the notion of `polygons' that have traditionally been studied, often called `polygonal linkages', in which \emph{all} side lengths are fixed. See \cite{Gal,KM99} for such studies of polygonal linkages on $\S^2$ and \cite{KM95, KMT00, CDR03} for polygonal linkages in hyperbolic or Euclidean space.

Let us now denote by $\mathcal{C}_{N}(\cdot)$ and $\mathcal{M}_{N}(\cdot)$ the configuration space and moduli space, respectively, of either spherical $N$-gons in $\S^2$ or semi-positive $N$-gons in $\Ein$ or $\Einh$. For the former moduli space, we quotient by $\SO(3)$ and for the latter by $\SO_0(2,3)$. 
Now, we introduce a \emph{turning invariant} $\Theta$ 
that labels the connected components of $\mathcal{M}_{N}(\S^2)$ and $\mathcal{M}_{N}(\Einh)$. 
 
To describe $\Theta$, we recall that any spherical polygon, in the above sense, is contained in an open hemisphere $H\subset \S^2$. Loosely speaking, $\Theta(\Delta)$ measures the oriented number of turns $\Delta$ makes around any great circle $\partial H$ bounding an enclosing hemisphere $H$. However, we caution that our polygons can be degenerate, for example alternating forwards and backwards entirely along a single great circle when $N$ is even. Thus, $\Theta$ is not exactly a `winding number', but is something akin to one. A cyclically ordered $N$-tuple $\Delta$ in $\S^2$ obtains the extremal values $\Theta=1$ or $\Theta=N-1$ exactly when it is  convex. Thus, our spherical polygons cannot obtain such a turning invariant $\Theta$, as they would necessarily violate either the length $2\pi$ condition or the non-collinearity condition. As a consequence, $\Theta$ obeys the inequality $2\leq \Theta \leq N-2$. 

We now use $\Theta$ to count components. 
Here is the setup. 
Relative to an orthogonal space-time splitting $\R^{2,3}=\R^{2,0}\oplus \R^{0,3}$, we can identify
$\Einh \simeq \S^1\times \S^2$. 
Given a polygon $\Lambda$ in $\Einh$, via the projection $\pi_-: \Einh \rightarrow \S^2$, we then declare $\Theta(\Lambda):= \Theta(\pi_-(\Lambda))$. The resulting value $\Theta(\Lambda)$ is independent of the chosen splitting used to project. 

\begin{theorem}[Components of polygons]
Let $N \geq 4$. 
The turning invariant $\Theta$ defines continuous maps 
$\Theta: \mathcal{C}_{N}(\S^2)\rightarrow \{2,3,\dots, N-2\}$ and $\Theta: \mathcal{C}_{N}(\Einh) \rightarrow \{2,3,\dots, N-2\}$ that label the connected components of each space.
Moreover, these maps are $\SO(3)$ and $\SO_0(2,3)$-invariant, respectively. 
\end{theorem}

It is through this count that we find the number of components of polygonal surfaces upstairs. 

\begin{corollary}
Let $N \geq 4$. Then $\mathcal{M}_{N}(\Einh)$ and $\mathcal{M}_{N}(\Hh)$ have $N-3$ connected components. 
\end{corollary}

The turning invariant $\Theta$ of a spherical polygon $\Delta$ depends on both the orientation of $\Delta$ and the orientation of $\S^2$. 
For this reason, $\Theta$ can only be used to define an invariant for $\Einh$-polygons and not $\Ein$-polygons. This situation is easily remedied: given a semi-positive $N$-gon $\Lambda \subset \Ein$, its two antipodal lifts $\hat{\Lambda},-\hat{\Lambda}\subset \Einh$ 
satisfy $\Theta(\hat{\Lambda}) + \Theta(-\hat{\Lambda})=N$.
Hence, we can simply imagine the turning invariant for polygons in $\Ein$ to live in the space $\{2,3,\dots, N-2\}/\sim$, where $k \sim N-k$. Consequently, the connected components are cut roughly in half after descent. 

\begin{corollary}
The spaces $\mathcal{M}_{n+4}(\Ein)$ and $\mathcal{M}_{n+4}(\H)$ have $\lceil \frac{n+1}{2}\rceil$ connected components. 
\end{corollary}
Thus, the components of polygonal maximal surfaces in $\H$ are labeled by the $\Z_2$-orbit of turning invariants achieved by the pair of antipodal lifts $\hat{\Lambda}$, $-\hat{\Lambda}$ of the asymptotic boundary $\Lambda = \partial_{\infty}\Sigma$. 

To finish this discussion, we describe the relation between invariants introduced thus far. 

\begin{theorem}[Holomorphic invariants to turning invariant]\label{thm:ConvertingInvariants}
Let $\delta =[(A\dz^2,B\dz^2)] \in \mathcal{M}_n({\mathcal{PD}}^o)$ and $\Sigma \subset \Hh$ be a maximal surface representing the class $\mathbf{{\Sigma}}^{o}(\delta)$ from Theorem \ref{thm:A}. Suppose that $\Sigma$ has asymptotic boundary $\Lambda =\partial_{\infty}\Sigma \subset \Einh$. Then $\Theta(\Lambda) =\deg(A)+2=n-\deg(B)+2$. 
\end{theorem}
In other words, the vanishing data of the complexified second fundamental form governs the turning of the boundary of $\Sigma$ at infinity. 

\subsubsection{Special polygonal surfaces}
As noted earlier, \cite{TW24} described a special component of polygonal maximal surfaces in $\H$, which we call the \emph{Hitchin component}, via holomorphic quartic differentials. 
Their surfaces correspond to those with $\alpha_{QD}=\dz^2$ or $\beta_{QD} =\dz^2$ in our notation. A matter left open in their work was which component of semi-positive polygons in $\Ein$ is realized by these special maximal surfaces.
We obtain the following geometric characterization. 

\begin{corollary}[Hitchin polygonal surfaces]\label{cor:HitchinSurfaces}
A polygonal maximal surface $\Sigma$ in $\H$ is Hitchin exactly when the turning invariant $\Theta$ of either lift of its asymptotic boundary to $\Einh$ is extremal. 
\end{corollary}

In \cite{Tam19}, Tamburelli studied polygonal maximal surfaces in anti-de-Sitter 3-space $\mathbf{H}^{2,1}$. 
In our case, we find that $\alpha$, $\beta$ determine which polygonal surfaces in $\H$ fail to be \emph{linearly full} and are instead contained in a copy of $\mathbf{H}^{2,1}$. 

\begin{corollary}[Reducible polygonal surfaces]\label{cor:ADS}
Let $\Sigma \subset \H$ be a polygonal maximal surface with asymptotic boundary $\Lambda = \partial_{\infty}\Sigma$ a $2k$-gon in $\Ein$. Then the following are equivalent: 
\begin{enumerate}[label=(\roman*)]
    \item $\Sigma$ is contained in a copy of $\mathbf{H}^{2,1}$.
    \item $\partial_{\infty}\Sigma$ is contained in a copy of $\mathbf{Ein}^{1,1}$.
    \item For any conformal parametrization $\sigma:\C \rightarrow \Sigma$, we have $\mathrm{div}(\alpha)=\mathrm{div}(\beta)$. 
\end{enumerate}
\end{corollary}
Here, the equivalence (i) $\leftrightarrow$ (ii) is an immediate consequence of \cite{Tam19, LTW24}, which we  state solely for completeness. 
The new content of the result is as follows: the polygonal maximal surfaces $\Sigma$ of $\H$ that are contained in a copy of $\mathbf{H}^{2,1}$ are exactly those with  $\Lambda=\partial_{\infty}\Sigma$ of even size $2k$,  $\Theta(\Lambda)=k+2$, and additionally with equal divisors $\div(\alpha)$, $\div(\beta)$ for any conformal parametrization. 

\subsection{Techniques and related work}\label{Sec:Technique}
We begin with a comparison to the closely related work \cite{TW24}. In this paper, Tamburelli-Wolf consider the polygonal maximal surfaces arising in $\H$ from the $\so(2,3)$ affine Toda equations. Their work is analytic in nature: in studying this $(2\times 2)$-system of elliptic PDE, they produce sub and super-solutions along with precise error estimates, which are leveraged to understand the asymptotic geometry of the associated maximal surfaces. 

We use an entirely different approach, which is based on the recent advancements in \cite{Mor25}. This change is 
necessary as the analytic arguments in \cite{TW24} do not seem to have obvious analogues for the accompanying systems of PDE governing non-Hitchin maximal surfaces. 
Avoiding the analytic bottleneck, we instead use our newfound capacity to certify polygonal surfaces by verifying that their induced metrics are asymptotically flat. This allows us to deduce much more quickly that the surfaces described in Theorem \ref{thm:A} are polygonal. We remark further on this point momentarily.

Despite the fact that we do not consider the $\so(2,3)$ affine Toda equations, other affine Toda equations nevertheless indirectly play a key role in our work  through the analysis of Li-Mochizuki in \cite{LM20b, LM25}. 
In these two works, the authors study the existence, uniqueness, and asymptotics of solutions to the $\mathfrak{sl}_n\C$ affine Toda equations on non-compact Riemann surfaces. Through the exceptional isomorphism $\mathfrak{sp}_4\R\cong \so(2,3)$, we can invoke their work on the $\mathfrak{sl}_4$ affine Toda equations to deduce the existence of solutions to the integrability equations for any maximal surfaces in $\H$ with polynomial quartic differentials. Thus, the study of Higgs bundles over the complex plane plays a major role for us, as in previous works studying planar surfaces with polygonal boundaries \cite{TW24, Eva24}. A commonality in these works is the study of \emph{wild} Higgs bundles, which was initiated in \cite{BB04} by Biquard-Boalch, and continued over $\CP^1$ by Fredrickson-Neitzke \cite{FN21}. 

In pursuit of a self-contained article, we recall in Section \ref{Sec:Preliminaries} the correspondence between harmonic $\SO_0(2,3)$-bundles and parametrized maximal surfaces in $\Hh$. 
We should emphasize that we do not explicitly invoke the framework of wild Higgs bundles over $\CP^1$. Instead, we prefer the approach suggested by Li-Mochizuki in their study of 
Higgs and harmonic bundles on non-compact surfaces as in \cite{LM20b, LM25}. In these works, they prove that certain Higgs bundles over non-compact Riemann surfaces carry harmonic metrics that are uniquely characterized by certain completeness conditions. We find this viewpoint quite attractive. 
In our setting, we then take the following approach. We begin with an $\SO_0(2,3)$-Higgs bundle $(\E, \Phi)$ over $\C$ with polynomial Higgs field $\Phi=\Phi(\alpha_{QD},\beta_{QD})$ encoded by a pair of polynomial quadratic differentials. We then search for a `natural' harmonic metric $h$, encoded by certain reality and completeness properties. 
The essential difficulty here is to carefully compare various notions of completeness, and eventually obtain equivalence with a certain `strong' completeness. 
We then verify that the obtained harmonic bundle $(\E,\Phi,h)$ describes a complete maximal surface $\Sigma \subset \H$ that is asymptotically flat and hence polygonal. In particular, this process is used to define the map $\mathbf{\Sigma}$ in Theorem \ref{thm:A}. To verify that the associated differential construction 
$\boldsymbol{\delta}: \mathcal{M}_{n+4}(\H) \rightarrow \mathcal{M}_n(\mathcal{PD})$ outlined preceding the statement of Theorem \ref{thm:A} is truly the inverse of $\boldsymbol{\Sigma}$, we show also that the harmonic bundle $(\E,\Phi,h)$ naturally associated to a parametrized polygonal surface $\sigma: \C\rightarrow \Hh$ has such a strongly complete harmonic metric $h$. 

Tamburelli and Wolf showed in \cite{TW24} that the moduli space $\mathcal{M}_{n+4}^{pt}(\Ein)$ of \emph{pointed} semi-positive polygons, up to the $\SO_0(2,3)$-action, is a union of cells. Here, `pointed' means that these polygons now have a marked vertex. However, they obtained an explicit component count only for $0\leq n\leq 2$, and conjectured that $\mathcal{M}_{n+4}^{pt}(\Ein)$ is disconnected for $n\geq 2$ generally. We affirm their conjecture. Indeed, the space $\mathcal{M}_{n+4}^{pt}(\Ein)$ is the universal cover of the orbifold $\mathcal{M}_{n+4}(\Ein)$, and our approach classifies the components of both via the turning invariant $\Theta$. To guarantee all possible $\Theta$-invariants are realized, we construct special ``symmetric'' maximal surfaces $\Sigma$ in $\Hh$ whose asymptotic boundaries $\partial_{\infty}\Sigma$ realize all the values $2\leq \Theta \leq n+2$. 
A natural strategy to show $\Theta$ classifies components is then to take an arbitrary $\Ein$-polygon and attempt to deform to a symmetric one. Indeed, a similar idea was recently proposed independently by Ma in \cite{Ma26}. It is non-trivial to define such a deformation algorithm that works in general, for polygons in $\Einh$, $\Ein$, or $\S^2$, by working \emph{only} in these spaces. We use the relation with maximal surfaces in Theorems \ref{thm:A}, \ref{thm:B} to prove that $\Theta$ offers a classification. By deforming through the space of polynomial quadratic differentials, we reach  monomials $\alpha_{QD} =z^k\dz^2$ and $\beta_{QD} = z^{n-k}\dz^2$, whose associated asymptotic boundary polygons in $\Einh$, $\Ein$, and the projections to $\S^2$, are symmetric. 

In \cite{Eva24}, the first named author studied some closely related objects: polynomial alternating almost-complex curves in the pseudosphere $(\mathbf{S}^{2,4}, J)$, which has automorphism group $\mathsf{G}_2'$, the split real adjoint group of type $G_2$. 
In that work, the polynomial curves associated to the $\g_2$ affine Toda equation are considered, in connection with the $\mathsf{G}_2'$-Hitchin system. The appropriate asymptotic boundaries of such objects are found to be \emph{annihilator polygons} in the Einstein universe $\mathbf{Ein}^{2,3}$. An obvious problem is to characterize all such polynomial almost-complex curves. 
However, there is substantially less known about alternating curves in $\S^{2,4}$ in general compared to maximal surfaces in $\mathbf{H}^{2,q}$. Indeed, there is no work on the Plateau problem for alternating curves in $\S^{2,4}$, such complete and immersed surfaces are not obviously embedded, nor are their asymptotic boundaries obviously topologically embedded, nor is the analogue of semi-positivity in $\mathbf{Ein}^{2,3}$ for annihilator polygons evident.
It would be interesting to explore these issues further.

The origin of the polynomial-polygon theme was the work \cite{HTTW95} of Han-Tam-Treibergs-Wan concerning harmonic maps $\C \rightarrow \mathbf{H}^2$ with polynomial Hopf differentials. 
Recently, there has been a resurgence of interest in this direction. 

In \cite{GS25}, Gupta-Sau study harmonic maps $\C \rightarrow \mathbf{H}^3$ with polynomial Hopf differentials and describe the possible resulting \emph{twisted polygons} that can be realized as asymptotic boundaries. In particular, they extend the picture developed by \cite{HTTW95} to the codimension one setting. 

More recently, Li-Liu-Xu study harmonic maps $\sigma:X \setminus D \rightarrow \mathbf{H}^2$ from punctured Riemann surfaces, where $X$ is closed and $D \subset X$ is an appropriate finite set, with asymptotics at each puncture prescribed by a certain polynomial harmonic map $\C\rightarrow \mathbf{H}^2$ \cite{LLX26}. This work extends prior study of harmonic maps from punctured surfaces to $\mathbf{H}^2$ by Huang in \cite{Hua16}.

Finally, we note the forthcoming work \cite{MNS26}, where Maloni-Nguyen-Sagman consider \emph{non-immersed} polynomial harmonic maps $\sigma :\C \rightarrow \mathbf{H}^2$. They have informed us that in the case of such \emph{branched} polynomial harmonic maps a certain turning invariant also appears. 
In their context, the `turning invariant' of interest encodes the number of turns the asymptotic boundary of the image of the harmonic map makes around the visual boundary $\partial_{\infty} \mathbf{H}^2$. The turning is then related to the branching data of the harmonic map. 
\section*{Organization of the paper}
The layout of the paper is as follows. 
\begin{itemize}
    \item In Section \ref{Sec:Preliminaries}, we introduce the basic objects and background, including maximal surfaces in $\H$, semi-positive loops in $\Ein$, and the work of \cite{LTW24} on the Plateau problem.
    \item In Section \ref{Sec:PolygonalSurfaces}, we describe the correspondence between harmonic $\SO_0(2,3)$-bundles and maximal surfaces in $\Hh$. Via Higgs bundles with polynomial Higgs field, we then produce polygonal maximal surfaces in $\H$. 
    \item In Section \ref{Sec:Polygon}, we study the moduli space of semi-positive polygons in $\Ein$ (and $\Einh)$ and the related moduli space of spherical polygons in $\S^2$, and count the components of each. 
    \item In Section \ref{Sec:ModuliSpaces}, we prove the main results: Theorems \ref{thm:A} and \ref{thm:B}, the parametrization of $\mathcal{M}_{n+4}(\H)$ and $\mathcal{M}_{n+4}(\Ein)$ via polynomial data. We then deduce Theorem \ref{thm:ConvertingInvariants} and Corollaries \ref{cor:HitchinSurfaces}, \ref{cor:ADS}. 
\end{itemize}

\section*{Acknowledgments}

We are grateful to Jérémy Toulisse, who proposed this problem to us, and for hosting the first named author in Nice in 2023 at the start of this project. We thank Andrea Seppi, Andrea Tamburelli, and Enrico Trebeschi for useful comments on an earlier draft of this work. We additionally thank Francesco Bonsante, Brian Collier, Colin Davalo, Timoth\'e Lemistre, François Labourie, Qiongling Li, Sara Maloni, Lam Nguyen, Nathaniel Sagman, Jérémy Toulisse and Mike Wolf for interesting discussions and valuable remarks.

A.~Moriani acknowledges funding by the European Research Council under ERC-Advanced grant 101095722 AnSur (Geometric Analysis and Surface Groups) and ERC-consolidator grant 101124349 GENERATE (GeomEtry and aNalysis for $(G, X)$-structurEs and their
defoRmATion spacEs). Views and opinions expressed are however those of the author(s) only and do not necessarily reflect those of the European Union or the European Research Council Executive Agency. Neither the European Union nor the granting authority can be held responsible for them.

\section*{AI Declaration} 

No AI was used in the creation of this article.

\section{Preliminaries}\label{Sec:Preliminaries}

\subsection{Pseudo-hyperbolic space and its boundary}

Let $\R^{2,3}$ denote the real vector space $\R^5$ endowed with the signature $(2,3)$ quadratic form $\Q$ whose matrix in the canonical basis of $\R^5$ is
\[\begin{pmatrix}
    I_{2} & \\
     & -I_3
\end{pmatrix}\ .\]
We endow $\R^{2,3}$ with its canonical space and time orientation. The action of the group $\SO_0(2,3)$ on $\R^{2,3}$ respects both the space and time orientations. The models for the pseudo-hyperbolic space and the Einstein universe will be built using $\R^{2,3}$.

\subsubsection{Models for the pseudo-hyperbolic space}

We define the \emph{pseudo-hyperbolic space} of signature $(2,2)$, and its double cover, denoted $\H$ and $\Hh$, respectively, as follows:
\begin{definition}
    $\Hh$ is the negative quadric in $\R^{2,3}$ and $\H$ is its projectivization. That is:
    \begin{align*}
         \Hh &= \{x\in\R^{2,3}\mid \Q(x)=-1\}\ ,\\
    \H &= \Hh/\{\pm 1\} = \{[x]\in\R\P^4\mid \Q(x)<0\}\ .
    \end{align*}
   
\end{definition}
The space $\Hh$ is equipped with the metric induced by $\Q$, which descends to $\H$, and is denoted by the same symbol there. Both $\H$ and $\Hh$ are pseudo-Riemannian manifolds of signature $(2,2)$, geodesically complete, and of constant curvature $-1$. 

The space $\H$ is an open submanifold of $\R\P^4$, whose boundary is the so-called \emph{Einstein universe} $\Ein$ of signature $(1,2)$, which we describe now.

\subsubsection{Models for the Einstein universe}
\begin{definition}
$\Einh$ is the set of isotropic rays in $\R^{2,3}$ and $\Ein$ its projectivization. That is:
\begin{align*}
     \Einh &=\{x\in\R^{2,3}\setminus\{0\}\mid \Q(x)=0\}/\R_{>0}\ ,\\
    \Ein &= \Einh/\{\pm 1\} = \{[x]\in\RP^4\ |\ \Q(x)=0\}\ .
\end{align*}
\end{definition}
The quadratic form $\Q$ restricts to a conformal class of signature $(1,2)$ pseudo-Riemannian metrics on $\Ein$. Thus, the type of tangent vectors (lightlike, timelike, spacelike) is canonically defined, but norms require the choice of a conformal representative.

\subsubsection{Photons} We define here photons, that are projective lines included in the Einstein universe.

\begin{definition}
A \textbf{photon} in $\Einh$ is the projection of an isotropic plane under the projection $\R^{2,3}\to\Einh$. Similarly, a \textbf{photon} in $\Ein$ is the projectivization of an isotropic plane.
\end{definition}

\subsection{Maximal surfaces in \texorpdfstring{$\H$}{H(2,2)}}

We now describe the class of surfaces of interest.

\begin{definition}
Let $\Sigma \subset \H$ be a smoothly immersed surface. Then 
\begin{itemize}
    \item $\Sigma$ is \textbf{spacelike} when $ \mathrm{T}_p\Sigma$ is spacelike for every point $p$ in $\Sigma$, thus when the induced metric from $\H$ is Riemannian,
    \item a spacelike surface $\Sigma$ is \textbf{maximal} when $\Sigma$ is a critical point for the area functional with respect to compactly supported deformations, 
    \item a spacelike surface $\Sigma$ is \textbf{complete} when the induced Riemannian metric from $\H$ defines a complete metric on $\Sigma$. 
\end{itemize}
\end{definition}
We shall be interested \emph{only} in \emph{complete maximal surfaces}. Going forward, by default, all maximal surfaces are complete unless otherwise stated. We shall note any  occasion that we remove this assumption. 

\begin{remark}\label{rem:MaximalVanishMeanCurvature}
Maximal surfaces are exactly the spacelike surfaces whose mean curvature vector vanish. The name `maximal' is due to the fact that the second variation of such a surface, under any compactly supported deformation, is negative. See \cite[Lemma 3.22 and Corollary 3.24]{LTW24}. 
\end{remark}

We may take also the following perspective. Beginning with  a Riemann surface $X=(S,J)$, we consider conformal harmonic maps $\sigma: X \rightarrow \H$ that are spacelike, with image $\Sigma =\sigma(X)$ an immersed maximal surface. 
We introduce some terminology for this situation. 

\begin{definition}
Let $X$ be a Riemann surface and $\sigma:X\to\H$ be a harmonic map and an immersion. 
\begin{itemize}
    \item When $\sigma$ is spacelike and conformal, we call $\sigma$ a \textbf{conformal maximal immersion}. 
    \item When $\sigma$ is a conformal maximal immersion and the induced metric $g_{\sigma}: = \sigma^*\Q$ is complete, we call $\sigma$ a \textbf{conformal maximal embedding}. 
    \item Instead, beginning with a complete maximal surface $\Sigma \subset \H$, if $\sigma$ is a conformal embedding, then we call it simply a \textbf{conformal parametrization} (of $\Sigma$). 
\end{itemize}
\end{definition}

\subsubsection{Tangent and normal bundles}
First, we introduce some notation for the decomposition of the pullback of the tangent bundle of $\H$. 
Let $\sigma:X\to\H$ be a conformal maximal immersion. 

\begin{definition}
We write $\sigma^* \mathrm{T}\H = \T_{\sigma} \oplus \N_{\sigma}$, where $\T_{\sigma}$ is the tangent bundle to $\sigma$, of signature $(2,0)$ and $\N_{\sigma}$ is the orthogonal complement to $\T_{\sigma}$, of signature $(0,2)$.
\end{definition}

On a related note, we may consider the trivial bundle $\underline{\R^{2,3}} \rightarrow \H$, whose pullback decomposes as $\sigma^*\underline{\R^{2,3}}= \L_{\sigma} \oplus \T_{\sigma} \oplus \N_{\sigma}$, where $\L_{\sigma}|_{p} = \R \{\sigma(p)\}$ is the position line. 

\subsubsection{Second fundamental form}
Recall that the second fundamental form of $\sigma$ is defined as follows: given $u$, $v$ tangent vectors of $X$ at a point $x$,
\[\II(u,v):=\left(\nabla_{\sigma_*u} \sigma_*v\right)^{\N_{\sigma}}\ ,\]
where $\nabla$ is the Levi-Civita connection of $\H$ and the superscript $\N_{\sigma}$ denotes the orthogonal projection to $\N_{\sigma}$. The second fundamental form is a symmetric covariant 2-tensor with values in $\N_{\sigma}$.

The complexification of $\II$ is denoted by $\II^{\C}$ and can be decomposed by type according to the complex structure of $X$. In a conformal coordinate $z$, we have
\[\II^\C=\II^\C(\partial_z,\partial_z)\dz^2 + \II^\C(\partial_z,\partial_{\zbar}) \dz\dzb + \II^\C(\del_{\zbar},\del_{\zbar}){\dzb}^2\ .\]
The $(2,0)$ part is then the smooth section of $\K_X^2\otimes\N_{\sigma}$ defined in coordinates by
\[\II^{2,0}:=\II^\C(\del_z,\del_z)\dz^2\ .\]

\subsubsection{Structures on the normal bundle}
We first recall a basic differential geometry construction. 
Let $E \rightarrow M$ be a real rank two vector bundle. Suppose $E$ is equipped with a positive-definite metric $g$ and an orientation $\omega$. Together, the pair $(g,\omega)$ reduces the structure of $E$ to $\SO(2,\R)\cong \U(1)$ and thus induces an almost-complex structure $J: E \rightarrow E$.  
Hence, we may decompose the complexification $E^{\C}$ by
$E^{\C} = E'\oplus E''$, where the complex line bundles $E'$ and $E''$ are the $\pm i$-eigenspaces of $J$.
We shall now apply this discussion to maximal surfaces momentarily. 

At this point, we wish to consider maximal surfaces $\Sigma$ in the double cover $\Hh$. This choice 
is for a simple reason: a surface upstairs carries a 
canonical orientation for $\N_{\sigma}$, as we now note. 

\begin{remark}[Oriented normal bundle]
Let $\sigma: X \rightarrow \Hh$ be a conformal maximal immersion. The splitting $\sigma^*\underline{\R^{2,3}}=\L_{\sigma}\oplus \T_{\sigma}\oplus \N_{\sigma}$ prescribes a canonical orientation on $\N_{\sigma}$ as follows: 
\begin{itemize}
    \item The position line $\L_{\sigma}$ is oriented via the trivializing tautological section given by $\sigma$. 
    \item The tangent plane $\T_{\sigma}$ inherits a fixed orientation from the tangent bundle $\mathrm{T}\Sigma$, inherited from the space orientation of $\R^{2,3}$.
    \item The normal plane $\N_{\sigma}$ obtains a canonical orientation from the fixed orientation on $\R^{2,3}$ and the prescribed orientations on $\L_{\sigma}, \T_{\sigma}$. 
\end{itemize}
\end{remark}
Equivalently, the lift orients $\L_{\sigma}$, which orients $\N_{\sigma}$ according to the time-orientation on $\L_{\sigma} \oplus \N_{\sigma}$. \medskip 

With this discussion in place, we shall henceforth regard the normal bundle $\N_{\sigma}$ of a conformal maximal immersion $\sigma: X \rightarrow \Hh$ as having a fixed orientation. As this minor detail plays an important role going forward, we felt this structure worthwhile to highlight.

Now, we resume our discussion, with a conformal maximal immersion $\sigma: X \rightarrow \Hh$ given. Note that $-\Q$ defines a positive-definite metric on $\N_{\sigma}$. With its metric and its orientation, $\N_{\sigma}^\C$ can be decomposed as $\N_{\sigma}^\C=\N_{\sigma}'\oplus\N_{\sigma}''$, where $\N_{\sigma}',\N_{\sigma}''$ are respectively the $i$ and $-i$ eigenspaces of the induced almost-complex structure. Moreover, $(\N_{\sigma},-\Q)$ comes equipped with a metric connection, namely the normal connection $\nabla^\N$. Thus each of $\N_\sigma'$ and $\N_\sigma''$ admit a natural holomorphic structure coming from the $(0,1)$-part of $\nabla^\N$. When we consider the holomorphic bundles, we write $\mathcal{N}'_\sigma$ and $\mathcal{N}''_\sigma$ instead of $\N_\sigma'$ and $\N_\sigma''$, respectively. Since $\mathcal{N}_\sigma'$ and $\mathcal{N}_\sigma''$ are dual with respect to the holomorphic bilinear pairing $\Q$, these line bundles are holomorphically dual. Going forward, we may sometimes write $\mathcal{N}$ for  $\mathcal{N}_{\sigma}'$ and $\mathcal{N}^{-1}$ for $\mathcal{N}_{\sigma}''$ for brevity.

\subsubsection{Associated holomorphic differentials}\label{subsubsec:HolDiff}

Recall that $\II^{2,0}$ is the $(2,0)$ part of $\II^\C$, which is a section of the rank two complex bundle $\K_X^2\otimes\N^{\C}_{\sigma}$. Via the decomposition of $\N_{\sigma}^{\C}$, we can split $\II^{2,0}$ into two parts. 

\begin{definition}[Associated Differentials]\label{def:HoloDiffAlphaBeta}
    Let $\sigma:X\to\Hh$ be a conformal maximal immersion.  Define $\alpha(\sigma)$ and $\beta(\sigma)$ as the components of $\II_\sigma^{2,0}$ with respect to the decomposition $\N_{\sigma}^\C = \N_{\sigma}'\oplus\N_{\sigma}''$. That is:
    \[\alpha(\sigma) \in \Omega^0(X, \K^2_X\otimes\N_{\sigma}'), \qquad \beta(\sigma)\in \Omega^0(X, \K^2_X\otimes\N_{\sigma}'')\ .\]
\end{definition}
In fact, $\alpha(\sigma)$ and $\beta(\sigma)$ are holomorphic.

\begin{proposition}\label{Prop:AlphaBetaHolomorphic}
Let $\sigma: X \rightarrow \Hh$ be a conformal maximal immersion. Then $\alpha \in H^0(\K^2_X\otimes \mathcal{N})$ and $\beta \in H^0(\K_{X}^2 \otimes \mathcal{N}^{-1})$. 
\end{proposition}

\begin{proof}
The result is an immediate consequence of $d^{\nabla}\II=0$. Here, parallelism of $\II$ follows from the Codazzi equation and the minimality of $\Sigma$. See \cite[Proposition 4.7, Lemma 5.6]{LT23}. 
\end{proof}

\subsubsection{Associated holomorphic quartic differential}\label{subsubsec:quarticdiff} The vanishing of the Hopf differential of a maximal surface leads to the following higher order holomorphic differential.  
\begin{definition}[Associated holomorphic quartic differential]\label{def:QuarticHoloDif}
Let $\sigma: X \rightarrow \Hh$ be a conformal maximal immersion. Then define its associated quartic differential $Q_4=Q_4(\sigma)$ as follows: 
\[ Q_4 := \mathbf{Q}(\II^{2,0}, \II^{2,0})\ .\]
\end{definition}

\begin{remark}\label{Remk:Q4AlphaBeta}
    Let $\alpha$, $\beta$ and $Q_4$ be defined as above. Then $Q_4=\mathbf{Q}(\alpha, \beta)$.
    Indeed, this follows from the fact that $\N_{\sigma}'$ and $\N_{\sigma}''$ are each $\Q$-isotropic. 
\end{remark}
\begin{proposition}
    Let $\sigma:X\to\Hh$ be a conformal maximal embedding. Then $Q_4\in H^0(\mathcal{K}_X^4)$.
\end{proposition}
\begin{proof}
    This follows immediately from Proposition \ref{Prop:AlphaBetaHolomorphic} and Remark \ref{Remk:Q4AlphaBeta}. 
\end{proof}

In other words, $Q_4 = \hat{Q}_4(z)\dz^4$ locally, where $\hat{Q}_4= \mathbf{Q}\big(\II^{\C}(\del_{z}, \del_{z}), \II^{\C}(\del_{z},\del_{z})\big)$ is a holomorphic function.

\subsection{Splittings and graphs}

For each orthogonal decomposition of $\R^{2,3}$ into a positive definite plane and a negative definite 3-space, the quadric $\Hh$ obtains a description as a product $\S^2_+\times\S^2$, as we explain now. For additional references, see \cite[Section 3.2]{SST23} or \cite[Section 3.1]{CTT19}.\medskip 

\subsubsection{Warped product structure and splitting}\label{sec:splitting}

We first introduce a basic geometric notion. 
\begin{definition}
Let $\R^{2,3}=l\oplus P\oplus R$ be an orthogonal decomposition with $l$ a negative line, $P$ a positive-definite plane, and $R$ a negative-definite plane. Note that 
$\mathbb{P}(l \oplus R)$ is a hyperbolic plane in $\H$, with a fixed basepoint $\mathbb{P}(l)$. 
We call $(l, P,R)$ a \textbf{pointed hyperbolic plane} in $\H$.

If $l$ is moreover endowed with an orientation, meaning the choice of a lift $o \in \Hh$ of $l$, then $(l,P,R)$ lifts to a \textbf{pointed hyperbolic plane} in $\Hh$ with fixed basepoint $o$. 
\end{definition}

Let $(l,P,R)$ be a pointed hyperbolic plane. Denote by $\S^2$ the unit sphere of $l\oplus R$ and by $\D^2$ the open unit disc in $P$. Now the map
\[\psi : \D^2\times\S^2\to\Hh\ , \qquad (x,v)\mapsto \frac{1+\|x\|^2}{1-\|x\|^2}\left(\frac{2x}{1+\|x\|^2},v\right)\ ,\]
is a diffeomorphism, which expresses the metric $\Q$ as a warped product as follows:
\[\psi^*Q_{(x,v)} = \left(\frac{1+\|x\|^2}{1-\|x\|^2}\right)^2\left[ \bigg(\frac{2}{1+\|x\|^2}\bigg)^2g_{\D^2}{}_{|x}\oplus-g_{\S^2}{}_{|v} \right]\ ,\]
where $g_{\D^2}$ denotes the Euclidean metric on $\D^2$ and $g_{\S^2}$ denotes the round metric on $\S^2$.

Using the quadratic form $-\Q$ on $l$ and $\Q$ on $P$, the sum $l\oplus P$ is positive-definite. Choose an orientation on $l$ (that is, chose a unit norm vector $o$ in $l$) and write $\S^2_+$ for the open upper hemisphere of $(l\oplus P,-\Q_{|l}+\Q_{|P})$, centered at $o$. The stereographic projection from $\S^2_+$ to $\D^2$ goes as follows: both $\S^2_+$ and $\D^2$ lie in $l\oplus P$, and an affine line passing through $-o$ and a point $u$ in $\S^2_+$ intersects $\D^2$ at a unique point $s(u)$. 
Now the map $s:\S^2_+\to\D^2$ is a diffeomorphism. More explicitly, any point $u$ in $\S^2_+$ writes uniquely as $u=x+ao$ where $x\in\D^2$ and $a\in\R$ and we have
\[s(u) = \frac{1}{1+a}x\quad\mbox{and}\quad s^{-1}(x)=\frac{2}{1+\|x\|^2}x+\frac{1-\|x\|^2}{1+\|x\|^2}o\ .\]

If we express the warped product given by $\psi$ instead on $\S^2_+\times \S^2$, using the stereographic projection $s:\D^2\to\S^2_+$, the pullback metric is a conformal metric: 
\[\left(\psi\circ(s,\id_{\S^2})\right)^*\Q_{\Hh}{}_{|(u,v)}=\left(\frac{1+\|s(u)\|^2}{1-\|s(u)\|^2}\right)^2(g_{\S^2_+}{}_{|u}-g_{\S^2}{}_{|v})\ ,\]
where $g_{\S^2_+}$ and $g_{\S^2}$ denote the round metrics on the upper hemisphere and on the sphere, respectively.

Moreover, the splitting $\psi\circ(s,\id_{\S^2})$ extends continuously to to $\partial(\S^2_+\times\S^2)=\S^1\times\S^2$ as the diffeomorphism
\[\S^1\times\S^2\to\Einh,\qquad (u,v)\mapsto u+v\ .\]
We can summarize the above discussion as follows. 
\begin{proposition}
    The data of a pointed hyperbolic plane in $\Hh$ gives rise to two conformal diffeomorphisms $\Hh\simeq\S^2_+\times\S^2$ and $\Einh\simeq\S^1\times\S^2$, called \textbf{splittings}, where $\S^1$, $\S^2_+$ and $\S^2$ are endowed with the round metrics $g_{\S^1}$, $g_{\S^2_+}$ and $-g_{\S^2}$, respectively. Moreover, the latter splitting is the continuous extension of the former.
\end{proposition}

\subsubsection{Complete spacelike surfaces}
We now highlight the role completeness plays for surfaces in pseudo-hyperbolic space. The following is a weaker version of \cite[Proposition 3.11]{LTW24}.

\begin{proposition}[Graphical surfaces]\label{prop:CompleteGraph}
    Let $\Sigma$ be a complete spacelike surface in $\Hh$. Then over any splitting $\Hh \simeq \S^2_+\times\S^2$, the surface $\Sigma$ is the graph of a strictly 1-Lipschitz map from $\S^2_+$ to $\S^2$.
\end{proposition}

\begin{corollary}\label{Cor:CompleteSpacelike}
    Let $\Sigma$ be a complete spacelike surface in $\Hh$. Then 
    \begin{enumerate}[label=(\roman*)]
        \item $\Sigma$ is a smoothly embedded disk.
        \item The frontier $\Lambda = \partial_{\infty}\Sigma$ of $\Sigma$ in $\Einh$ is the graph of a 1-Lipschitz map from $\S^1$ to $\S^2$, relative to any splitting. 
        \item $\Lambda$ is a topologically embedded circle in $\Einh$. 
    \end{enumerate}
\end{corollary}

\subsection{Semi-positive loops}\label{subsec:SemiPositive}
Due to prior work on maximal surfaces \cite{CTT19, LTW24, LT23, SST23}, the notion of a \emph{semi-positive loop} now admits many characterizations. We begin with the following description. 

\begin{definition}[Positivity of triples]
    Let $\mathcal{\tau}=(u,v,w)$ be a triple of points in $\Ein$ or $\Einh$. We call $\tau$ to be 
            \begin{itemize}
                \item \textbf{negative} when $\spann \{u,v,w\}$ is a subspace in $\R^{2,3}$ of signature $(1,2)$, 
                \item \textbf{positive} when $\spann\{u,v,w\}$ is a subspace in $\R^{2,3}$ of signature $(2,1)$. 
    \end{itemize}
\end{definition}
One way to define semi-positivity of a loop is by the configuration of triples contained inside. 
\begin{definition}
    A \textbf{semi-positive loop} in $\Ein$ or $\Einh$ is a topologically embedded circle that contains no negative triple and at least one positive triple.
\end{definition}

\subsubsection{Alternative description of semi-positive loops} The given definition of semi-positivity has the benefit of involving only the intrinsic geometry of the Einstein universe. However, we shall later be interested in another characterization involving the structure of the curve relative to a splitting. 

\begin{proposition}[Equivalence of semi-positivity]\label{Prop:SemiPositive}
Let $\Lambda \subset \Einh$ be topological circle. Then the following are equivalent:
\begin{enumerate}[label=(\roman*)]
    \item $\Lambda$ is a semi-positive loop,
    \item With respect to any splitting,  $\Lambda$ is the graph of a 1-Lipschitz map $\S^1\rightarrow \S^2$ that  contains no antipodal points. 
\end{enumerate}
\end{proposition}

\begin{proof}
The equivalence is proven by \cite[Proposition 2.23]{LT23}. 
\end{proof}

Note that combining Proposition \ref{Prop:SemiPositive} and Corollary \ref{Cor:CompleteSpacelike}, we find that the asymptotic boundary $\Lambda =\partial_{\infty}\Sigma$ of a complete maximal surface $\Sigma$ is necessarily a semi-positive loop. Next, we shall discuss the converse. 

\subsection{Asymptotic Plateau problem}

The asymptotic Plateau problem in $\mathbf{H}^{2,q}$ was solved completely by Labourie-Toulisse-Wolf. 

\begin{theorem}[Asymptotic Plateau problem in $\mathbf{H}^{2,q}$ {\cite{LTW24}}]\label{Thm:Plateau}
Let $\Lambda \subset \mathbf{Ein}^{1,q}$ be a semi-positive loop. Then there exists a unique complete maximal surface $\Sigma$ with asymptotic boundary $\Lambda$. 
\end{theorem}

The compactified surface $\overline{\Sigma} = \Sigma \sqcup \Lambda$ is topologically a closed 2-disk, as $\Sigma$ is an embedded 2-cell, and $\Lambda \cong \S^1$. However, the conformal type of the surface $(\Sigma,[g])$ is not fixed: it can either be biholomorphic to the disk or to the complex plane. For example, in the former case, $\Sigma$ could be a totally geodesic hyperbolic plane in $\H$, and in the latter case, $\Sigma$ could be flat, such as the \emph{Barbot surface} (see Figure \ref{fig:BarbotSurface}).
\begin{figure}[h!]
    \centering
    \includegraphics[width=0.4\linewidth]{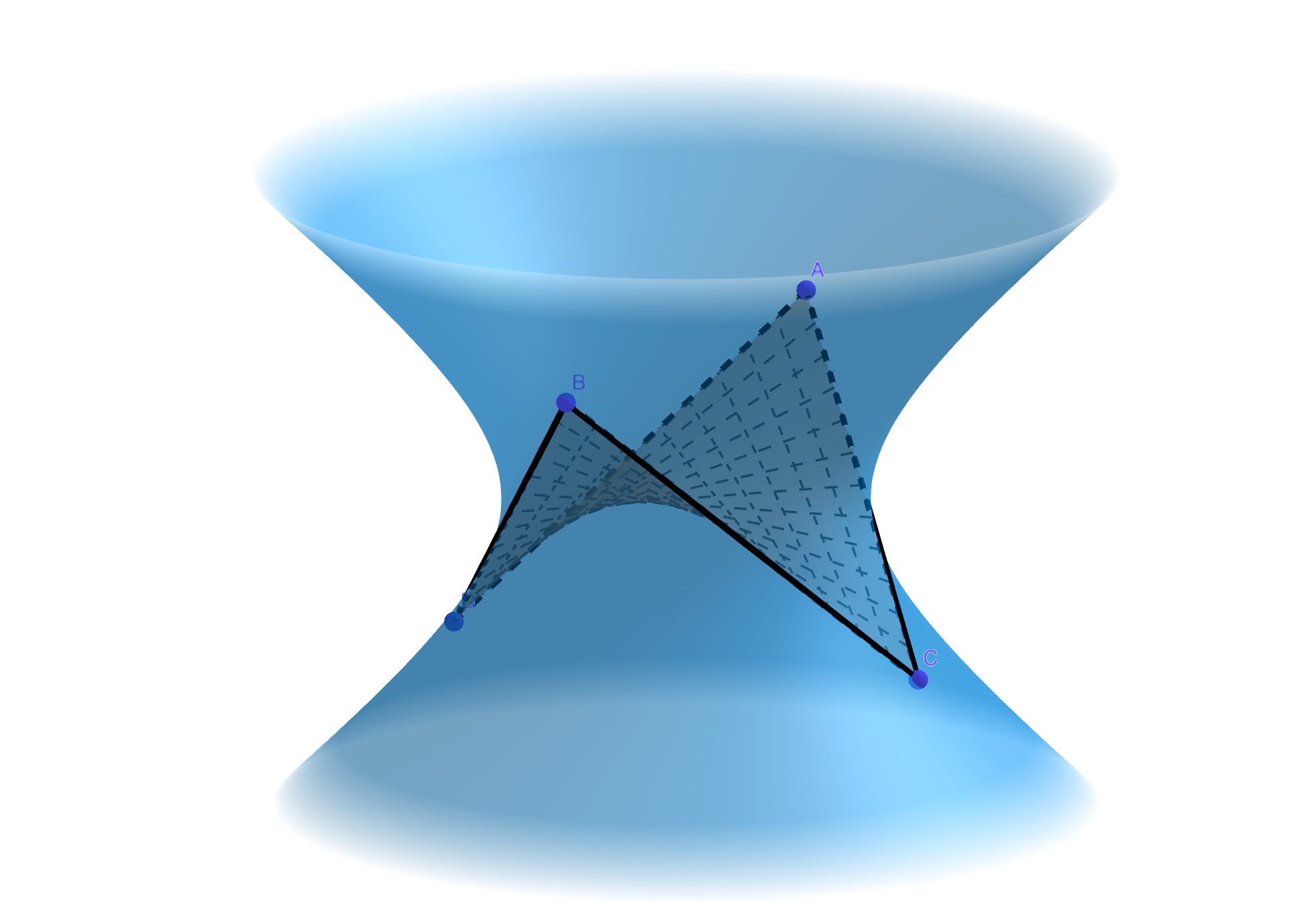}
    \caption{The Barbot surface $\Sigma$ in $\mathbf{H}^{2,1}$, that can be visualized at \href{https://www.geogebra.org/classic/bhxfrxkz}{Geogebra}. The asymptotic boundary $\Lambda=\partial_{\infty}\Sigma$ is a semi-positive $4$-gon in $\mathbf{Ein}^{1,1}$.}
    \label{fig:BarbotSurface}
\end{figure}

\section{Polynomial maximal surfaces in \texorpdfstring{$\H$}{H(2,2)}}\label{Sec:PolygonalSurfaces}

In this section, we build all maximal surfaces in $\H$ with polynomial quartic differential via Higgs bundles over the complex plane. We then prove that the resulting surfaces are complete, asymptotically flat, and as a consequence of \cite{Mor25}, have polygonal asymptotic boundaries in the Einstein universe $\Ein$.

\subsection{Higgs bundles to maximal surfaces}\label{subsec:HiggstoMaximal}
We shall now describe $\SO_0(2,3)$-Higgs bundles $(\mathcal{E},\varphi)$ on $\C$, which will correspond to the polygonal maximal surfaces $\sigma: \C \rightarrow \H$ of interest.
\subsubsection{Higgs bundles}
We first recall the notion of a Higgs bundle, in both the vector and principal bundle settings. Below, $\K= \K_X$ denotes the canonical bundle of a Riemann surface $X$.
\begin{definition}[Higgs bundle]
Let $X$ be a Riemann surface (possibly non-compact). 
\begin{itemize}
    \item A \textbf{Higgs (vector) bundle} on $X$ is a pair $(\E,\varphi)$ consisting of a holomorphic vector bundle $\E$ on $X$ and a \textbf{Higgs field} $\varphi \in H^0(\End(\E) \otimes \K)$. 
    \item More generally, given a real reductive Lie group $\G$, a $\boldsymbol{G}$\textbf{-Higgs bundle} $(\mathcal{P}, \Phi)$ on $X$ consists of a holomorphic principal $\sf K^{\C}$-bundle $\mathcal{P}$ on $X$, where $\sf K < \G$ is a maximal compact subgroup, and a Higgs field $\Phi \in H^0(\mathcal{P}[\mathfrak{p}^{\C}]\otimes \K_{X})$, where 
$\g^{\C} = \mathfrak{k}^{\C} \oplus \mathfrak{p}^{\C}$ is the complexified Cartan decomposition. 
\end{itemize}
\end{definition}
We will follow the usual convention of freely passing between the vector and principal bundle formulations. For example, given a rank $n$ Higgs vector bundle $(\E,\varphi)$, then setting $\mathcal{P}$ to be the holomorphic frame bundle of $\E$, one can identify
$\mathcal{P}[\gl_n\C ]\cong \End(\E)$, and thus $(\E,\varphi)$ becomes a principal $\GL(n,\C)$-Higgs bundle. 
In similar fashion, for $\G <\GL(n,\C)$ a real Lie subgroup, we can often describe $\G$-Higgs bundles as `decorated' Higgs vector bundles, whereby the decorations encode the structure required to build the $\sf K^{\C}$-bundle $\mathcal{P}$. We shall illustrate this idea for $\G=\SO_0(2,3)$ now for clarity, though this approach is now well-known. See \cite{Col19} for further examples.

\begin{example}[$\SO_0(2,3)$-Higgs bundles]
Let $\G = \SO_0(2,3)$. The maximal compact subgroup $\sf K$, isomorphic to $\SO(2)\times \SO(3)$, is the stabilizer of an orthogonal space-time splitting $\R^{2,3} = U\oplus V$, where $U \cong \R^{2,0}$ and $V\cong \R^{0,3}$. Thus, $\sf K^{\C} \cong\SO(2,\C) \times \SO(3,\C)$. Hence, an $\SO_0(2,3)$-Higgs bundle on a Riemann surface $X$ is described by a tuple $(\mathcal{U}, \mathcal{V}, \mathcal{Q}_{\mathcal{U}}, \mathcal{Q}_{\mathcal{V}}, \eta)$, where 
\begin{itemize}
    \item $\mathcal{U}, \mathcal{V}$ are holomorphic vector bundles of ranks $2, 3$, such that  $\det(\mathcal{U})\cong \mathcal{O}\cong \det(\mathcal{V})$,  
    \item $\mathcal{Q}_{\mathcal{U}}, \mathcal{Q}_{\mathcal{V}}$ are non-degenerate holomorphic bilinear forms on $\mathcal{U}, \mathcal{V}$, 
    \item $\eta  \in H^0(\Hom(\mathcal{U},\mathcal{V})\otimes \K)$.
\end{itemize}
Indeed, $\mathcal{P}$ is then the frame bundle of $(\mathcal{U}\oplus \mathcal{V},\mathcal{Q}_{\mathcal{U}} \oplus \mathcal{Q}_{\mathcal{V}}, \omega_{\mathcal{U}} \oplus \omega_{\mathcal{V}})$, where $\omega_{\mathcal{U}},\omega_{\mathcal{V}}$ are volume forms, and the Higgs field is $\Phi = \begin{pmatrix} & -\eta^{*\mathcal{Q}} \\ \eta & \end{pmatrix}  \in H^0(\mathcal{P}[\mathfrak{p}^{\C}]\otimes \K)$, where $\mathcal{Q} = \mathcal{Q}_{\mathcal{U}}\oplus (-\mathcal{Q}_{\mathcal{V}})$. 
\end{example}

We now describe the $\SO_0(2,3)$-Higgs bundles $(\mathcal{E},\varphi)$ of interest on $X =\C$. 
First, we make a small remark regarding notation. While all holomorphic vector bundles over $\C$ are holomorphically trivial, we employ similar notation here as in the case of Higgs bundles on closed Riemann surface, where the corresponding bundles can be non-trivial. Hopefully, the notational parallels aid the reader. 

Now, the Higgs bundles $(\E, \varphi)$ of interest obtain the following form:
\begin{equation}\label{SO(2,3)Higgs}
\begin{tikzcd}
\mathcal{N}^{-1} \arrow[r, "\alpha"] & \mathcal{K}\arrow[r, "2"]&\mathcal{O}\arrow[r, "2"]&\mathcal{K}^{-1} \arrow[r, "\alpha"]\arrow[bend left=20, lll, "\beta"]&\mathcal{N}\arrow[bend left=20, lll, "\beta"]
\end{tikzcd}\ .
\end{equation}
Here, $\alpha\in H^0(\K^2\otimes \mathcal{N})$ and $\beta\in H^0(\K^2\otimes \mathcal{N}^{-1})$. Hence $\alpha\beta\in H^0(\K^4)$. 
Going forward, we denote by $Q_4$ the holomorphic quartic differential $Q_4=\alpha\beta$. 

The $\SO_0(2,3)$-structure on $(\E,\varphi)$ is described as follows: 
\begin{itemize}
    \item $\mathcal{U} = \mathcal{K} \oplus \mathcal{K}^{-1}$,
    \item $\mathcal{V} = \mathcal{N}^{-1} \oplus \mathcal{O}\oplus \mathcal{N}$,
    \item $\mathcal{Q}_{\mathcal{U}} = \begin{pmatrix} & \mathbbm{1} \\ \mathbbm{1} &  \end{pmatrix} $, $\mathcal{Q}_{\mathcal{V}} = \begin{pmatrix} &  & \mathbbm{1} \\  & 1 & \\ \mathbbm{1} &  & \end{pmatrix} $,
\end{itemize}
where $\mathbbm{1}$ denotes the canonical pairing between a vector bundle and its dual. 

The appearance of `$2$', rather than `$1$', in \eqref{SO(2,3)Higgs} is due to the correspondence with the associated $\mathbb{P}\Sp(4,\R)$-Higgs bundle, under the isomorphism $\mathbb{P}\Sp(4,\R) \cong \SO_0(2,3)$, as we explain in Section \ref{Subsec:CorrespondingSpHiggsB}. 

The following insightful remark was shared with us by Andrea Tamburelli, which highlights an important distinction between maximal surface in $\H$ versus in the double cover $\Hh$. 
\begin{remark}
Observe the following symmetry in \eqref{SO(2,3)Higgs}: by exchanging $\mathcal{N}, \mathcal{N}^{-1}$ and $\alpha,\beta$, the same Higgs bundle is achieved. 
This symmetry is geometric: exchanging $\mathcal{N}$ and $\mathcal{N}^{-1}$ amounts to switching the holomorphic and anti-holomorphic normal line bundles, at which point $\alpha$ and $\beta$ switch roles.
\end{remark}

\subsubsection{Hitchin's equations} As in the \emph{non-abelian Hodge correspondence}, one may search for a hermitian metric $h$ on $\mathcal{E}$ that is \emph{harmonic}.
To find $h$, one must solve a system of PDE called \emph{Hitchin's equations} \eqref{HitchinsEquation} for the flatness of the connection $\nabla = \nabla^C + \varphi + \varphi^{*h}$, where $\nabla^C$ is the Chern connection of the hermitian holomorphic bundle $(\mathcal{E},h)$. The system of PDE of interest is as follows: 
\begin{align}\label{HitchinsEquation}
    F_{\nabla^C}+ [\varphi,\varphi^{*h}] = 0\ .
\end{align} 
As is well-known, the above equation \eqref{HitchinsEquation}, expressing $F_{\nabla} =0$, is equivalent to the harmonicity of $h$, meaning that $h$ encodes a harmonic map $h: \tilde{X} \rightarrow \GL(n,\C)/\mathrm{U}(n)$.
However, in the case $(\E, \varphi)$ is a $\G$-Higgs bundle, the harmonic metric $h$ in fact describes a map $h: \tilde{X} \rightarrow \X_{\G}$, where $\X_{\G}$ denotes the Riemannian symmetric space of $\G$, seen as a totally geodesic subspace of $\X_{\GL(n,\C)}\cong \GL(n,\C)/\U(n)$. 

\subsubsection{Harmonic bundles} Let us introduce another standard definition. 
\begin{definition}
A \textbf{harmonic bundle} $(\E,\varphi, h)$ consists of a Higgs bundle  with a harmonic metric $h$. 
\end{definition}

\subsubsection{From harmonic bundles to maximal surfaces}

Let $(\E,\varphi)$ be the $\SO_0(2,3)$-Higgs bundle \eqref{SO(2,3)Higgs}. We will consider the situation when this Higgs bundle is promoted to a harmonic bundle $(\E,\varphi, h)$ such that the metric $h$ is diagonal and real. In particular, canonically associated to such a harmonic bundle is a parametrized maximal surface $\sigma: \C \rightarrow \Hh$, up to post-composition by $\G =\SO_0(2,3)$. The harmonic map 
$h: \C \rightarrow \X_{\SO_0(2,3)}$ associated to $h$ is then simply described in terms of $\sigma$. 

To describe the correspondence of interest, we first describe one feature of interest of the flat connection associated to a particular harmonic metric.  
\begin{lemma}[Real locus]
Let $h=\diag(h_1,h_2,1, h_2^{-1},h_1^{-1})$ be a diagonal, real harmonic metric on the Higgs bundle $(\E,\varphi)$ from \eqref{SO(2,3)Higgs}, with flat connection $\nabla := \nabla^{C} +\varphi +\varphi^{*h}$. Then the involution $\lambda: \E\rightarrow \E$ such that 
\[ h(\cdot, \cdot) = Q(\cdot, \lambda(\cdot))\ ,\]
is $\nabla$-parallel. 
\end{lemma}
Such a real structure $\lambda$ is called \emph{adapted} in \cite[Section 2.3]{CTT19}, where the above result is proven. Note that in the case $X$ is compact and $(\E,\varphi)$ is \emph{stable}, a unique harmonic metric $h$ exists and obtains the stated form. However, in the non-compact setting, one must prove existence and uniqueness. Such a task has been undertaken under \emph{completeness} hypotheses by Li and Mochizuki \cite{LM20b, LM25} from which we shall extract in Section \ref{Subsec:HarmonicMetric} a preferred diagonal, real harmonic metric $h$ for the polynomial Higgs bundles of interest. \medskip 

Now, write $\E^{\R} \cong \R^{2,3}$ for the $\nabla$-parallel real locus.
Using this real structure, we now see the associated maximal surface. This construction is well-known since the influential thesis of Baraglia \cite{Bar10}. 

\begin{proposition}[Associated maximal surface]\label{Prop:HiggsToMaximalSurface}
Let $h=\diag(h_1,h_2,1, h_2^{-1},h_1^{-1})$ be a diagonal, real harmonic metric on the Higgs bundle $(\E,\varphi)$ from \eqref{SO(2,3)Higgs}, with flat connection $\nabla := \nabla^{C} +\varphi +\varphi^{*h}$. 

Let $s \in \Omega^0(\C, \E)$ be the tautological section $s(p)=(0,0,1,0,0)$, and let $p_0$ denote the origin in $\C$. 
Fix an identification $\E^{\R}_{\mid p_0}\cong \R^{2,3}$. 
Then $\sigma: \C \rightarrow \Hh$ defined by  
$\sigma(p) := \mathscr{P}^{\nabla}_{p,p_0}(s(p))$, where $\mathscr{P}^{\nabla}$ denotes $\nabla$-parallel transport, is a conformal maximal immersion. Moreover, $\alpha$ and $\beta$ are equal to $\alpha(\sigma)$ and $\beta(\sigma)$, respectively,  from Definition \ref{def:HoloDiffAlphaBeta}.
\end{proposition}

\begin{proof}
Note that since $\nabla$ is flat and $\C$ is simply connected, the $\nabla$-parallel transport $\mathscr{P}_{p,p_0}^{\nabla}$ from $\E_{p}$ to $\E_{p_0}$ is well-defined and independent of path. 

Next, observe that $\mathcal{Q}(s,s) =-1$. Since $h$ is $\nabla^{C}$-parallel and $\varphi+\varphi^{*h} \in \so(\mathcal{Q})$, the bilinear form $\mathcal{Q}$ is $\nabla$-parallel. Thus, $\sigma: \C \rightarrow \Hh$ is well-defined. 
We now verify its adjectives. 

The differential $d\sigma$ corresponds to $\nabla \sigma $, and also implies that $\sigma$ is immersed. The plane  
$\spann_{\R} \{ \nabla_{\der{x}}\sigma, \nabla_{\der{y}}\sigma\}$ corresponds to $\T_{\sigma}$ and hence $\sigma$ is spacelike. The condition $ \mathcal{Q}(\nabla_{\der{x}}\sigma) = \mathcal{Q}( \nabla_{\der{y}}\sigma) $ shows $\sigma$ is conformal.
Let us write $\varphi = \hat{\varphi}dz,\, \varphi^{*h}= \hat{\varphi}^{*h}d\zbar $. Finally, 
$ \nabla_{\zbar}\nabla_{z} \sigma = \hat{\varphi}^{*h}\circ \hat{\varphi}(\sigma) $ has zero $\mathcal{Q}$-projection onto $\sigma^\bot$, which shows that $\sigma$ is harmonic, and thus maximal. The `moreover' statement follows from a direct calculation. 
\end{proof}

Observe that the only choice in Proposition  \ref{Prop:HiggsToMaximalSurface} was the framing of $\E^{\R}_{\mid p_0}$, which is an $\SO_0(2,3)$-torsor. This explains the uniqueness of $\sigma$ up to post-composition. 

\subsection{Maximal surfaces to Higgs bundles}\label{Subsec:MaximaltoHiggs}

We now describe the converse correspondence: given a conformal maximal embedding $\sigma: \C \rightarrow \Hh$, we recover a $4$-cyclic $\SO_0(2,3)$-harmonic bundle $(\E,\Phi, h)$ over $\C$ that obtains the form described in the previous subsection. While this discussion is mildly adapted to the setting of $X=\C$, minor modifications will treat also the equivariant case, where $\sigma: \tilde{X} \rightarrow \H$ is $\rho: \pi_1(X) \rightarrow \SO_0(2,3)$-equivariant, and $(\E,\Phi)$ is a Higgs bundle on a closed Riemann surface $X$. This correspondence is well-known to experts, and runs in parallel to \cite[Section 3.5]{CT24}, where the analogous but more involved case of alternating holomorphic curves in $\mathbf{H}^{4,2}$ is studied in relation to harmonic $\mathsf{G}_2'$-bundles. A similar discussion occurs in \cite[Section 2.2]{LM19} for equivariant harmonic maps to hyperbolic space $\mathbf{H}^n$. 

\subsubsection{The differential geometry setup}
Now, let $\sigma: \C \rightarrow \Hh$ be a parametrized maximal surface. We consider the trivial pseudo-Euclidean vector bundle $E \rightarrow \Hh$ given by $E:=\Hh \times \R^{2,3}$. 
Here, we endow $E$ with the usual signature $(2,3)$-metric $\Q$, and also with the trivial flat connection $D$. We can consider the complexified vector bundle $E^{\C}$, where $E^{\C}$ carries the complex linear and bilinear extensions of $D$ and $\Q$ respectively, both of which we  denote by their original symbols. 

The pullback bundle $\sigma^*E$ decomposes under the Frenet framing of $\sigma$ into
$\sigma^*E= \L \oplus \T \oplus \N$, where $\L$, $\T$ and $\N$ are the position line, tangent plane and normal plane of $\sigma$, respectively. 

By construction of the Frenet framing, the pullback connection $\sigma^*D$ decomposes as: 
\begin{align}
    \sigma^*D = \begin{pmatrix} \nabla_{\L} & -\I^{*_{\Q}} & \\ \I & \nabla_{\T} & -\II^{*_{\Q}}\\
    & \II & \nabla_{\N}\end{pmatrix}\ .
\end{align}
Given the Frenet frame splitting, we can convert $\Q$ into a Euclidean metric $\mathbf{g}$ by flipping signature on the negative subspaces. That is,
\[ \mathbf{g} := (-\Q|_{\L}) \oplus \Q|_{\T} \oplus (-\Q|_{\N})\ .\]
We extend $\mathbf{g}$ complex bilinearly to $E^{\C}$. 

Now, let us define the $\C$-anti-linear involution $\lambda: \sigma^*E^{\C} \rightarrow \sigma^*E^{\C}$ that fixes $E$. 
Observe that there is a hermitian metric $h$ on $E^{\C}$ given by 
\[ h( \cdot, \cdot ) = \mathbf{g}( \cdot, \lambda(\cdot ))\ .\]

\subsubsection{The harmonic bundle via harmonicity of $\sigma$}

We now decompose the pullback flat connection $\sigma^*D$ on $\sigma^*E^{\C}$. Indeed, we can write 
\begin{align}
    \sigma^* D &= D_h \oplus \Psi_h\ ,
\end{align}
where the two terms are as follows:
\begin{align}
    D_h &=  \begin{pmatrix} \nabla_{\L} &  & \\  & \nabla_{\T} & \\
    & & \nabla_{\N}\end{pmatrix}\ , \\
    \Psi_h &=  \begin{pmatrix}  & -\I^{*_{\Q}} & \\ \I &  & -\II^{*_{\Q}}\\
    & \II & \end{pmatrix}\ .
\end{align}
Here, $D_h$ is a $\sf K\cong \SO(2)\times \SO(3)$-connection.
The 1-form $\Psi_h \in \Omega^1(\C, \End(\sigma^*E^{\C}))$ is $\Q$-anti-self-adjoint, or equivalently, $h$-self-adjoint. Thus, $\Psi_h \in \Omega^1(\so(\mathbf{Q}))$.

Finally, form the associated Higgs bundle $(\E,\Phi)$ as follows. Decompose $D_h$ and $\Psi_h$ by type and define:
\begin{align}\label{MetrictoHiggsBundle}
\begin{cases}
    \E &= (\sigma^*E^{\C}, D_h^{0,1})\ ,\\
    \Phi &= \Psi_h^{1,0}\ .
\end{cases}
\end{align}
The following proposition verifies the validity of the construction. While this construction is completely general and treated as such in \cite{Li19}, we can handle the current case more directly. 

\begin{proposition}[Parametrized maximal surface to Higgs bundle]\label{Prop:HarmonicMetricToHiggs}
Given a conformal maximal immersion $\sigma: \C \rightarrow \Hh$, the pair $(\E, \Phi)$ in \eqref{MetrictoHiggsBundle} is a Higgs bundle. That is, $D_h^{0,1}\big(\Psi_h^{1,0}\big)=0$. 
\end{proposition}

\begin{proof}
The desired result is tantamount to $D^{0,1}_h\I^{1,0} =0$ and $D^{0,1}_h\II^{1,0}=0$. 

In fact, the equation $D_h^{0,1}\I^{1,0}=0$
is \emph{precisely} the harmonicity constraint for $\sigma$, as it is equivalent to the assertion that $\I^{1,0}(\der{z},\cdot)$ maps $\sigma \in H^0(\L^{\C})$ to $\sigma_z \in H^0(\T^{\C})$. Equivalently, in coordinates $\nabla_{\T^{\C}}(\der{\zbar},\sigma_z) = \pi_{\T^{\C}}(\sigma_{z\zbar})=0$. 
Hence, we conclude $D_h^{0,1}\I^{1,0}=0$ holds. \medskip

Let us split $\T^{\C} = \T' \oplus \T''$. Using harmonicity, one finds $\II^{1,0} \in \Omega^{1,0}(\Hom(\T^{\C},\N^{\C}))$ satisfies $\T'' \subset \ker(\II^{1,0})$. Hence, $\II^{1,0}$ can be conflated with $\II^{2,0} \in \Omega^{2,0}(\N^\C)$. To verify $\II^{2,0}$ is $D_h^{0,1}$-holomorphic, note that $\II^{1,0}(\der{z},\cdot)$ sends the holomorphic section $\sigma_z$ spanning $\T^{1,0}$ to $\II^{2,0}(\der{z},\der{z}) = \alpha(\der{z},\der{z})+\beta(\der{z},\der{z})$, which is holomorphic, by Proposition \ref{Prop:AlphaBetaHolomorphic}.
Hence, $D_h^{0,1}\II^{1,0}=0$.
\end{proof}

We now explain why $\E$ decomposes into line bundles such that $\Phi$ obtains the form $\Phi = \Phi(\alpha,\beta)$ as in \eqref{SO(2,3)Higgs}.
To this end, recall the splittings $\N^{\C} = \N' \oplus \N''$ and $\T^{\C} = \T'\oplus \T''$. We can more simply describe these sub-bundles $\T'$, $\T''$, $\N'$, $\N''$ as those determined by
$\sigma_z$, $\sigma_{\zbar}$, $\pi_{\N^{\C}}(\sigma_{zz})$, $\pi_{\N^{\C}}(\sigma_{\zbar \, \zbar})$. 
We have both the hermitian metric $h$ and the holomorphic structure $D_h^{1,0}$ on $\E$. 
We may then regard $\L^{\C}$, $\T^{\C}$, $\N^{\C}$ as holomorphic sub-bundles of $\E$ now. 
Under this decomposition, note that $\T^{\C}$ decomposes into holomorphic sub-bundles $(\T^{\C}, D_h^{1,0})= \mathcal{T}\oplus \mathcal{T}^{-1}$ and 
$(\N^{\C}, D_h^{1,0})= \mathcal{N}\oplus \mathcal{N}^{-1}$,
where $\mathcal{T}$, $\mathcal{T}^{-1}$, $\mathcal{N}$, $\mathcal{N}^{-1}$ have underlying complex line bundles $\T'$, $\T''$, $\N'$, $\N''$, respectively. 

\begin{proposition}[Parametrized maximal surface to harmonic bundle]\label{Prop:MaximaltoHiggs}
Let $\sigma: \C \rightarrow \Hh$ be a conformal maximal immersion. 

Define $(\E,\Phi)$ as in \eqref{MetrictoHiggsBundle}. Under the decomposition 
\[ \E = \mathcal{N}^{-1} \oplus \mathcal{T}^{-1} \oplus \mathcal{L} \oplus \mathcal{T} \oplus \mathcal{N}\ ,\] 
the Higgs field $\Phi$ and harmonic metric $h$ obtain the form:
\begin{equation}
    \begin{tikzcd}
\mathcal{N}^{-1} \arrow[r, "\alpha"] & \mathcal{T}^{-1}\arrow[r, "\gamma"]&\mathcal{L}\arrow[r, "\gamma"]&\mathcal{T} \arrow[r, "\alpha"]\arrow[bend left=20, lll, "\beta"]&\mathcal{N}\arrow[bend left=20, lll, "\beta"]\ ,
    \end{tikzcd} 
\end{equation} 
\begin{equation}
 h = \diag(h_1,h_2, 1,h_2^{-1},h_1^{-1})\ .
\end{equation}
Here, $\gamma$ is non-vanishing. Thus, writing $\mathcal{L}\cong \mathcal{O}$, $\mathcal{T}\cong \K^{-1}$, the Higgs bundle obtains the form \eqref{SO(2,3)Higgs}. 
\end{proposition}

\begin{proof}
The argument is essentially contained in the above remarks, so we only elaborate briefly. Now, since $\mathcal{T}$ and $\mathcal{T}^{-1}$ are each $\Q$-isotropic, we see $\Q$ defines a non-degenerate pairing between the two that exhibits them as holomorphic duals. 
Similar remarks apply to $\mathcal{N}$ and $\mathcal{N}^{-1}$. Since the real structure $\lambda$ exchanges $\mathcal{N}$, $\mathcal{N}^{-1}$ and $\mathcal{T}$, $\mathcal{T}^{-1}$, we see these pairs are $h$-orthogonal. Thus, $h$ obtains the desired form. 

By definition, the Higgs field $\Phi$ restricts to obtain the form 
$\Phi|_{\mathcal{T}}: \mathcal{T} \rightarrow (\mathcal{N}^{-1} \oplus \mathcal{N})\otimes \K$, where 
$\Phi|_{\mathcal{T}} = (\beta, \alpha)$ in our earlier definition. We also saw in Proposition \ref{Prop:HarmonicMetricToHiggs} that $\mathcal{T}^{-1} \in \ker(\II ^{1,0})$. Similarly, $\Phi|_{\mathcal{L}} = \I^{1,0} = \gamma$ in this notation, with $\gamma$ is non-vanishing since $\sigma$ is immersed. 

Finally, the second copies of $\alpha$, $\beta$, $\gamma$ in the Higgs field $\Phi$ correspond to $-(\I^{1,0})^{*_{\Q}}$ and $-(\II^{1,0})^{*_{\Q}}$. 
\end{proof}

\subsubsection{Completeness and asymptotic flatness} 
We now impose a restriction on the correspondence and consider only parametrized maximal surfaces that are \emph{complete} and \emph{asymptotically flat}. In this case, we shall show the associated harmonic bundle $(\E,\varphi, h)$ obeys some additional completeness conditions. 

To proceed, we shall need Hitchin's equations for the Higgs bundle governing maximal surfaces. 
Before stating the result, we recall some notation. 
\begin{remark}
Note that in a frame $(s_i)_{i=1}^5$ for $\E = \bigoplus_{i=1}^5 \mathcal{L}_i$, if $h$ is written as $h= \diag(\hat{h}_i)_{i=1}^5$ and $\phi \in H^0(\Hom(\mathcal{L}_{i},\mathcal{L}_{i+1})\otimes \K)$ is written as $\phi=\hat{\phi}\,dz$, where $\hat{\phi}$ is a  section of $\Hom(\mathcal{L}_{i},\mathcal{L}_{i+1})$, then we obtain an associated metric 
\[ ||\phi||^2_h=|\hat{\phi}|^2\hat{h}_i\hat{h}_{i-1}^{-1}|dz|^2\ . \]
For the metric $g_0= |dz|^2$, we also obtain the function 
$||\phi||:=||\phi||_{h,g_0}$ given by $||\phi||=  |\hat{\phi}|^2\hat{h}_i\hat{h}_{i-1}^{-1}$.  
More generally, for any conformal metric $g=\hat{g}\,|dz|^2$, we may consider the function  
\[ ||\phi||^2_{h,g} := \frac{||\phi||^2_h}{g} = \frac{|\hat{\phi}|^2\hat{h}_i\hat{h}_{i-1}^{-1}}{\hat{g}}\ .\]
\end{remark}

Now, Hitchin's equations can be described as follows.

\begin{proposition}[Hitchin's equations for maximal surfaces]
\label{prop:HitchinsEquationsSO(2,3)}
Let $(\E,\varphi)$ be the Higgs bundle \eqref{SO(2,3)Higgs}. Let $h= \diag(h_1, h_2,1,h_{2}^{-1},h_{1}^{-1})$ be a diagonal and real hermitian metric. Then $h$ is harmonic if and only if 
\begin{align}\label{Eq:HitchinsEquationsSO(2,3)}
    \begin{cases}
        \log(h_1)_{z\zbar} = ||\beta||^2-||\alpha||^2\ ,\\
        \log(h_2)_{z\zbar} = ||\alpha||^2+||\beta||^2-||2||^2\ . 
    \end{cases}
\end{align}
\end{proposition}

\begin{proof}
This is a direct calculation using $F_{\nabla^C} = \log(h)_{z\zbar} d\zbar \wedge dz $
and $[\varphi, \varphi^{*h}]= [\hat{\varphi},\hat{\varphi}^{*h}]dz\wedge d\zbar$. 
\end{proof}

Next, we consider the associated metrics $||\alpha||^2_h$ and $||\beta||^2_h$ of complete and asymptotically flat maximal surfaces. Here, we denote by $\mathcal{Z}(\cdot)$ the vanishing set of a given tensor. 

\begin{corollary}[Mutual boundedness]\label{cor:MutualBoundedness}
Let $\sigma: \C \rightarrow \Hh$ be a conformal maximal embedding, $g$ its complete induced metric, and $(\E,\varphi)$ the associated Higgs bundle via Proposition \ref{Prop:MaximaltoHiggs}, written in the form \eqref{SO(2,3)Higgs}. If $g$ is asymptotically flat, then it is mutually bounded with the metrics $||\alpha||_h^2$ and $||\beta||_h^2 $ on $\C \setminus K$ for any relatively compact open set $K$ containing the zero set $\mathcal{Z}(Q_4)$. 
\end{corollary}

\begin{proof}
Take $\sigma$ and $K \subset \C$ as described in the hypotheses. 
We compute the curvature of $g$ using that $g= h_{2}^{-1}$ in the notation of Proposition \ref{prop:HitchinsEquationsSO(2,3)} and \ref{prop:HitchinsEquationsSO(2,3)}. We find, using Equation \eqref{Eq:HitchinsEquationsSO(2,3)}, that the curvature $\kappa_g$ of $g$ is given by 
\[ \kappa_g = -\frac{2}{g}(\log g)_{z\zbar} = \frac{2}{g}\big(||\alpha||^2+||\beta||^2-||2||^2\big)\ .\]
Thus, $||\alpha||^2+||\beta\|^2$ is mutually bounded with $g$ on $\C \setminus K$. In particular, $||\alpha||_{h,g}^2$ and $||\beta||^2_{h,g}$ are bounded above. 
We shall seek nonzero lower bounds for $||\alpha||_{h,g}^2$, $||\beta||_{h,g}^2$ to complete the proof. 

Next, we invoke \cite[Proposition 5.2.1]{Mor25}, which asserts that $g$ is mutually bounded with $|Q_4|^{1/2}$ on $\C \setminus K$. 
Observe that $||\alpha||^2= |\alpha|^2h_1^{-1}h_2$
and $||\beta||^2 = |\beta|^2h_1h_2$.
Hence, we obtain 
\[ g^2\approx|Q_4| \approx \frac{|Q_4|^2}{g^2}=\frac{|\alpha|^2|\beta|^2}{g^2} = ||\alpha||_{ }^2||\beta||_{}^2\ . \] 
In particular, on $\C \setminus K$, we have $||\beta||^2 \geq \frac{Cg^2}{||\alpha||^2}$ for some $C> 0$. Hence, 
when $|z|$ is large, we see
\[ 0 \approx \frac{1}{2}\kappa_g \geq \frac{||\alpha||^2}{g}+\frac{Cg}{||\alpha||^2}-1\ .\]
This implies that $||\alpha||^2_{h,g}$ is bounded away from zero on $\C \setminus K$. The same reasoning exchanging $\alpha$ and $\beta$ shows that $||\beta||_{h,g}^2$ admits a nonzero lower bound on $\C\setminus K$ as well. 
\end{proof}

\subsection{Corresponding \texorpdfstring{$\mathsf{Sp}(4,\R)$}{Sp(4,R)}-Higgs bundles}\label{Subsec:CorrespondingSpHiggsB}

Next, we recall the isomorphism $\mathbb{P}\Sp(4,\R)\cong \SO_0(2,3)$ and use this equivalence to translate the Higgs bundles from both perspectives.  

\subsubsection{$\Sp(4,\R)$-Higgs bundles} We begin by describing $\Sp(4,\R)$-Higgs bundles from the vector bundle perspective. Before the example, we review some Lie theory. Write $W= V \oplus V^*$, so that $\omega = \begin{pmatrix} &- \mathbbm{1} \\ \mathbbm{1} & \end{pmatrix}$ is a symplectic form on $W$. Define for later $Q = \begin{pmatrix} & \mathbbm{1} \\ \mathbbm{1} & \end{pmatrix}$. 
In the case $W$ is complex, the complexified Cartan decomposition $\spl(W) = \mathfrak{k}^{\C} \oplus \mathfrak{p}^{\C}$, is as follows: 
\begin{itemize}
    \item $\mathfrak{k}^{\C} = \bigg\{\begin{pmatrix} * & \\ & * \end{pmatrix} \in \spl(W)\bigg\} = \spl(W)\cap \so(Q)$,
    \item $\mathfrak{p}^{\C} = \bigg\{\begin{pmatrix} &* \\ *&  \end{pmatrix} \in \spl(W)\bigg\} = \spl(W)\cap \Sym_{Q}(W)$.
\end{itemize}
With these preliminaries in mind, we proceed. 
\begin{example}
Let $G =\Sp(4,\R)$. Then $\sf K=\U(2)$ and hence $\sf K^{\C} = \GL(2,\C)$. 

Thus, an $\Sp(4,\R)$-Higgs bundle can be described by a triple $(V, \gamma, \delta)$, where $V$ is a rank two holomorphic vector bundle, $\gamma \in H^0(\Sym^2(V^*)\otimes \K)$ and $\delta \in H^0(\Sym^2(V)\otimes \K)$.

The associated $\SL(4,\C)$-Higgs bundle $(\E, \varphi)$ to $(V,\gamma,\delta)$ is as follows: 
\begin{itemize}
    \item $\E = V\oplus V^*$,
    \item $\varphi = \begin{pmatrix} 0 & \delta \\ \gamma & 0 \end{pmatrix} \in H^0(\End(\E) \otimes \K)$.
\end{itemize}
Note that the associated $\SL(4,\R)$-Higgs bundle $(\E,\varphi,\mathcal{Q})$ is given by taking $\mathcal{Q} = \begin{pmatrix} &\mathbbm{1} \\ \mathbbm{1} & \end{pmatrix}$. 
\end{example}

 \begin{remark}
We emphasize that
$\gamma$ and $\delta$ should be viewed as $\mathcal{Q}$-symmetric bundle maps twisted by $\K$ of type $\gamma \in H^0(\Hom(V, V^*)\otimes \K)$
and $\delta \in H^0(\Hom(V^*,V)\otimes \K)$. That is, we interpret `Sym' to mean $\mathcal{Q}$-symmetric, so that $\Sym^2(V) $ identifies with $\mathcal{Q}$-symmetric maps in  $\Hom(V^*,V)$ and similarly $\Sym^2(V^*) \simeq\Hom_{\mathcal{Q}-sym}(V,V^*)$. 
 \end{remark}

We now describe the $\Sp(4,\R)$-Higgs bundles $(V,\beta,\gamma)$ of interest over $X=\C$:
\begin{itemize}
    \item  $V = \K^{1/2}\mathcal{N}^{1/2} \oplus \K^{1/2}\mathcal{N}^{-1/2}$,
    \item $\gamma = \begin{pmatrix} 0 & 1 \\  
                1 & 0\end{pmatrix} $,  $\delta =\begin{pmatrix} \alpha & 0 \\ 0& \beta\end{pmatrix} $. 
\end{itemize}
Here, $\alpha \in H^0(\K^2\otimes \mathcal{N})$ and $\beta \in H^0(\K^2\otimes \mathcal{N}^{-1})$, as before. \medskip

Thus, the associated $\SL_4(\C)$-Higgs bundle is as follows:
\begin{align}\label{SL4CHiggsPolynomial}
    \mathcal{E} =(\K^{1/2}\mathcal{N}^{1/2} \oplus \K^{1/2}\mathcal{N}^{-1/2})\oplus (\K^{-1/2}\mathcal{N}^{-1/2} \oplus \K^{-1/2}\mathcal{N}^{1/2})\ , \;\; \varphi=\begin{pmatrix} & \begin{matrix} \alpha & 0 \\ 0& \beta \end{matrix} \\
\begin{matrix} 0 & 1 \\  
                1 & 0 \end{matrix} & \end{pmatrix} \ .
\end{align}

\subsubsection{Exceptional isomorphism} We now briefly recall the details of the exceptional isomorphism $\mathbb{P}\Sp(4,\R) \rightarrow \SO_0(2,3)$,
as it is needed to translate between Higgs bundles of the two groups. 

For the statement, recall that the space $\Lambda^2(\R^4)$ has an inner product $g$ of signature $(3,3)$ given by
$v\wedge w= g(v,w) \mu$, where $\mu$ is a volume form on $\R^4$. Hence, there is a homomorphism 
$\SL(4,\R) \rightarrow \SO_0(3,3)$ given by $g \mapsto \Lambda^2(g)$. Note that a symplectic form $\omega \in \Lambda^2(\R^4)$ is a spacelike element. We can restrict the homomorphism $\SL(4,\R)\rightarrow \SO_0(3,3)$ to define the isomorphism of interest. 

\begin{proposition}[Exceptional isomorphism]
The homomorphism $\Lambda^2:\Sp(4,\R) \rightarrow \SO_0(2,3)$ 
by $g \rightarrow \Lambda^2(g)|_{\omega^\bot}$ is surjective, has kernel $\pm \id$, and induces an isomorphism 
$\mathbb{P}\Sp(4,\R)\cong \SO_0(2,3)$. 
\end{proposition}

\subsubsection{The correspondence} We now explain how to convert our $\Sp(4,\R)$-Higgs bundle into an $\SO_0(2,3)$-Higgs bundle. 
The key representation-theoretic facts used here are as follows: for $\sf K^{\C} \cong \GL(2,\C)$ the complexified maximal compact subgroup of $\Sp(4,\R)$, the second exterior power of the standard representation $V$ decomposes as a $\sf K^{\C}$-representation as 
\[ \Lambda^2(V) = \C\{\omega\} \oplus \big( \,[\Lambda^2(V)\oplus \Lambda^2(V^*)] \oplus \Hom_0(V,V) \big)\ ,\]
where $\C\{\omega\}$ is the trivial sub-representation corresponding to $\omega$, and the orthogonal complement decomposes into two sub-representations of dimensions 2 and 3, respectively. 
See \cite[Proposition 6.2.1]{Col16} for further details on the correspondence. 
We note also the isomorphisms 
$\Hom_0(V,V) \cong \Lambda^2(V)\otimes \Sym^2(V^*) \cong \Lambda^2(V^*) \otimes \Sym^2(V)$ of $\sf K^{\C}$-modules.

We now convert the $\Sp(4,\R)$-Higgs bundles $(V,\beta,\gamma)$ of interest into the $\SO_0(2,3)$ side.  
To this end, write $V = \mathscr{L}_1\oplus \mathscr{L}_2$, where $\mathscr{L}_1 = \mathcal{N}^{1/2}\K^{1/2}$ and $\mathscr{L}_2=\mathcal{N}^{-1/2}\mathcal{K}^{1/2}$. 
In this notation, the Higgs bundle \eqref{SL4CHiggsPolynomial} obtains the form 

\begin{equation}\label{Sp(4,R)HiggsRewritten}
\begin{tikzcd}
\mathscr{L}_1 \arrow[bend right=20, rrr, "1"] & \mathscr{L}_2 \arrow[r, "1"] & \mathscr{L}_1^{-1} \arrow[bend right=25, ll,swap, "\alpha"]&\mathscr{L}_2^{-1}\arrow[bend right=25, ll,swap, "\beta"]
\end{tikzcd}.
\end{equation}
Note that $\Lambda^2(V) = \mathscr{L}_1\mathscr{L}_2$ and $\Lambda^2(V^*) = \mathscr{L}_{1}^{-1}\mathscr{L}_2^{-1}$. Then the associated $\SO_0(2,3)$-Higgs bundle is as follows: 

\begin{equation}\label{SO(2,3)HiggsRecovered}
\begin{tikzcd}
\mathscr{L}_1^{-1}\mathscr{L}_2 \arrow[r, "\alpha"] & \mathscr{L}_1\mathscr{L}_2 \arrow[r, "2"]&\mathcal{O}\arrow[r, "2"]& \mathscr{L}_1^{-1}\mathscr{L}_2^{-1} \arrow[r, "\alpha"]\arrow[bend left=20, lll, "\beta"]&\mathscr{L}_1\mathscr{L}_2^{-1}\arrow[bend left=20, lll, "\beta"]
\end{tikzcd}.
\end{equation}
The $\SO_0(2,3)$-Higgs bundle structure is given by: 
\begin{itemize}
\item $\mathcal{U} = \mathscr{L}_1\mathscr{L}_2 \oplus \mathscr{L}_1^{-1}\mathscr{L}_2^{-1}$
\item $\mathcal{V} = \mathscr{L}_1\mathscr{L}_2^{-1} \oplus \mathcal{O}\oplus \mathscr{L}_1^{-1}\mathscr{L}_2$, 
\item$Q_{\mathcal{U}} = \begin{pmatrix} & \mathbbm{1} \\ \mathbbm{1} &  \end{pmatrix} $, $Q_{\mathcal{V}} = \begin{pmatrix} &  & \mathbbm{1} \\  & 1 & \\ \mathbbm{1} &  & \end{pmatrix} $,
\end{itemize}
In particular, the Higgs bundle \eqref{SO(2,3)HiggsRecovered} recovers the Higgs bundle \eqref{SO(2,3)Higgs} introduced earlier. For additional clarity, we emphasize the correspondence between harmonic metrics on the corresponding bundles. 
\begin{proposition}[Correspondence of harmonic metrics]\label{Prop:CorrespondenceMetrics}
Let 
\[ h = \diag(h_1, h_2,h_1^{-1},h_2^{-1})\]
be a diagonal and real hermitian metric on $\mathscr{L}_1\oplus \mathscr{L}_2\oplus \mathscr{L}_1^{-1}\oplus\mathscr{L}_2^{-1}$.

Define the diagonal, real hermitian metric
\[ \overline{h} = \diag(h_1^{-1}h_2, h_1h_2, 1, h_1^{-1}h_2^{-1}, h_1h_2^{-1})\] 
on $\mathscr{L}_1^{-1}\mathscr{L}_2\oplus \mathscr{L}_1\mathscr{L}_2 \oplus \mathcal{O}\oplus\mathscr{L}_1^{-1}\mathscr{L}_2^{-1}\oplus \mathscr{L}_1\mathscr{L}_2^{-1}$. 

Then $h$ is harmonic on the Higgs bundle \eqref{Sp(4,R)HiggsRewritten} if and only if $\overline{h}$ is harmonic on the Higgs bundle \eqref{SO(2,3)HiggsRecovered}.
\end{proposition}

This is a straightforward calculation that we include solely to clarify that the `$1$' in the $\Sp(4,\R)$-Higgs bundle must become a `$2$' in the $\SO_0(2,3)$-Higgs bundle for this correspondence of harmonic metrics to hold. 

\begin{proof}
The $h$-harmonicity equations are as follows, where $F$ denotes the curvature tensor: 
\begin{align}\label{h-harmonic}
    \begin{cases}
        F_{h_1} = 1 \wedge 1^{*h} -\beta\wedge \beta^{*h}\ ,\\
        F_{h_2} = 1\wedge 1^{*h}-\alpha \wedge \alpha^{*h}\ .
    \end{cases}
\end{align}
On the other hand, the $\overline{h}$-harmonicity equations are:  
\begin{align}\label{hbar-harmonic}
    \begin{cases}
        F_{h_1^{-1}h_2} = \alpha \wedge \alpha^{*\bar{h}} -\beta\wedge \beta^{*\bar{h}}\ ,\\
        \, F_{h_1h_2} \ = 2\wedge 2^{*\bar{h}}-\alpha \wedge \alpha^{*\bar{h}}-\beta \wedge \beta^{*\bar{h}}\ .
    \end{cases}
\end{align}
Now, using that $F_{g_1g_2}=F_{g_1}+F_{g_2}$ for metrics $g_i$ on a line bundle $\mathcal{L}$, and the fact that for 
$\eta \in \{1,\alpha,\beta\}$, we have 
$\eta\wedge \eta^{*h} = \eta\wedge \eta^{*\bar{h}}$, the claim follows.
\end{proof}

The $\Sp(4,\R)$-Higgs bundle perspective is necessary to invoke the results of Li-Mochizuki. Indeed, the Higgs bundle \eqref{SO(2,3)Higgs} is  a 4-cyclic $\SL(5,\C)$-Higgs bundle, yet \eqref{SL4CHiggsPolynomial} is a 4-cyclic $\SL_4(\C)$-Higgs bundle. 

\subsection{A preferred harmonic metric}\label{Subsec:HarmonicMetric}
We shall apply the machinery of Li-Mochizuki from \cite[Section 7]{LM20b} and \cite{LM25} to prove uniqueness of a \emph{complete} diagonal solution to Hitchin's equations in our setting. Most of the subtlety involves comparing different notions of completeness. We shall endeavor, for the most part, to follow their original notation. 

\subsubsection{Cyclic Higgs bundles} To start, we consider more general $4$-cyclic Higgs bundles. Namely, we shall consider $\SL_4(\C)$-Higgs bundles $(\E, \varphi)$ of the form $\E= \bigoplus_{i=1}^4 \mathcal{L}_i$ and $\varphi=\sum_{i=1}^4 \varphi_i$, where 

\begin{equation}\label{SL4Cyclic}
    \begin{tikzcd}
 \mathcal{L}_1 \arrow[r, "\varphi_1"] & \mathcal{L}_2 \arrow[r, "\varphi_2"]& \mathcal{L}_{3} \arrow[r, "\varphi_3"]&\mathcal{L}_{4}\arrow[lll, bend left=20,  "\varphi_4"]
\end{tikzcd}.
\end{equation}
More precisely, we consider a \emph{fixed} ordering of the line bundles $\mathcal{L}_i$. \medskip 

Now, going forward, we consider the open Riemann surface $X = \C$. On $X$, we study the $\Sp(4,\R)$-Higgs bundle $(\E, \varphi)$ below. In particular, we re-write $\varphi$ to obtain the usual form as a cyclic transformation. Crucially, due to the $\Sp(4,\R)$-symmetries, we have \emph{two ways} to view the Higgs bundle \eqref{SL4CHiggsPolynomial} in the form \eqref{SL4Cyclic}, corresponding to two possible permutations of the $\mathcal{L}_i$'s in our case.
To this end, we set 
\[ \begin{cases} 
        \mathcal{L}_1:=\K^{1/2}\mathcal{N}^{-1/2}\ ,\\
        \mathcal{L}_2:=\K^{-1/2}\mathcal{N}^{1/2}\ ,\\
        \mathcal{L}_3:=\K^{1/2}\mathcal{N}^{-1/2}\ ,\\
        \mathcal{L}_4:=\K^{-1/2}\mathcal{N}^{1/2}\ .
\end{cases} \]
Observe the following cyclic description of the Higgs bundle:
\begin{equation}\label{Cyclic1}
    \begin{tikzcd}
 \mathcal{L}_1 \arrow[r, "1"] & \mathcal{L}_2 \arrow[r, "\alpha"]& \mathcal{L}_{3} \arrow[r, "1"]&\mathcal{L}_{4}\arrow[lll, bend left=20,  "\beta"]
\end{tikzcd}.
\end{equation}
By exchanging $\mathcal{L}_i \leftrightarrow \mathcal{L}_{i+2}$, we obtain another cyclic description as
\begin{equation}\label{Cyclic2}
    \begin{tikzcd}
 \mathcal{L}_3 \arrow[r, "1"] & \mathcal{L}_4 \arrow[r, "\beta"]& \mathcal{L}_{1} \arrow[r, "1"]&\mathcal{L}_{2}\arrow[lll, bend left=20,  "\alpha"]
\end{tikzcd}.
\end{equation}

\subsubsection{Completeness conditions} While the two Higgs bundles \eqref{Cyclic1} and \eqref{Cyclic2} are obviously the same, the ordering of the line bundles $\mathcal{L}_i$ affects the relevant completeness condition from \cite{LM20b}, which will play a major role for us.
We recall this notion of completeness now.

\begin{definition}[Completeness]\label{Defn:Complete}
Let $(\E, \varphi)$ be a cyclic Higgs bundle on $X=\C$, 
with (ordered) cyclic decomposition $\E = \bigoplus_{i=1}^4\mathcal{L}_i$ as in
\eqref{SL4Cyclic}. 
Define $\psi: = \varphi_3\circ \varphi_2\circ \varphi_1 \in H^0(\Hom(\mathcal{L}_1, \mathcal{L}_4)\otimes \K^3)$
and assume that $\psi$ has a finite zero set, denoted
$\tilde{D} :=\mathcal{Z}(\psi)$. 
Let $N \subset \C$ be a relatively compact open set containing $\tilde{D}$. 

Then we call a hermitian metric $h=\diag(h_i)_{i=1}^4$ on $(\E,\varphi)$ \textbf{complete} when $||\varphi_i||^2_{h}$, for $1 \leq i \leq 3$, are complete metrics on $\C \setminus N$. 
\end{definition}

We emphasize: the \emph{actual order}, not just the cyclic order, of the line bundles $\mathcal{L}_i$ is essential to the definition of completeness. For clarity, we explicitly record the two different completeness conditions of interest, relative to the two different ordered splittings introduced above.

\begin{remark}
Let $(\E,\varphi)$ be the $\SL_4(\C)$-Higgs bundle \eqref{SL4CHiggsPolynomial} on $\C$. 
Let $\mathcal{C}_1$ and $\mathcal{C}_2$ be the ordered splittings in \eqref{Cyclic1} and \eqref{Cyclic2}, respectively. 
Let $h$ be a harmonic metric on $(\E,\varphi)$. Note that
\begin{itemize}
    \item $h$ is $\boldsymbol{\mathcal{C}_1}$\textbf{-complete} if and only if $||1_{12}||_h^2, ||\alpha||_h^2,||1_{34}||_h^2$ are complete on $\C \setminus N(\mathcal{Z}(\alpha))$, 
    \item $h$ is $\boldsymbol{\mathcal{C}_2}$\textbf{-complete} if and only if $||1_{12}||_h^2, ||\beta||_h^2,||1_{34}||_h^2$ are complete on $\C \setminus N(\mathcal{Z}(\beta))$.
\end{itemize}
\end{remark}
Recall that $\mathcal{Z}(\cdot)$ denotes the vanishing set of a given tensor. 

For later, we introduce a stronger notion of completeness.
\begin{definition}[Strong completeness]\label{Defn:StrongCompleteness}
Let $(\E,\varphi)$ be the $\SL_4(\C)$-Higgs bundle \eqref{SL4CHiggsPolynomial} on $\C$. We call a harmonic metric $h$ \textbf{strongly complete} if $h$ is $\mathcal{C}_1$-complete and $\mathcal{C}_2$-complete. 

When $(\E, \varphi)$ is an $\Sp_4(\R)$-Higgs bundle, we may use the same terminology for a harmonic metric $h$ on the associated $\SO_0(2,3)$-Higgs bundle \eqref{SO(2,3)Higgs}, via the equivalence in Proposition \ref{Prop:CorrespondenceMetrics}. 
\end{definition}

In this case, we obtain the following existence and uniqueness result for complete metrics from \cite[Theorem 7.13]{LM20b}, with respect to a fixed splitting. We state their result only in our case, though it applies more generally to $n$-cyclic $\SL_n(\C)$-Higgs bundles on open Riemann surfaces $X$. 

\begin{theorem}[Unique complete + diagonal harmonic metric]\label{Thm:Sp4CompleteSolution}
Let $(\E,\varphi)$ be a Higgs bundle on $\C$ with cyclic decomposition $\mathcal{C}$ as in \eqref{SL4CHiggsPolynomial}
such that $\varphi_3\circ \varphi_2\circ \varphi_1$ has finite zero set $\tilde{D}$. Fix any relatively compact open set $N \subset \C$ containing $\tilde{D}$. Then there exists a unique diagonal harmonic metric $h=(h_i)_{i=1}^4$ on $\E=\bigoplus_{i=1}^4\mathcal{L}_i$ that is $\mathcal{C}$-complete. Moreover, $||\varphi_i||^2_h$ are mutually bounded on $X \setminus N$, for $1\leq i\leq3$, and $\frac{||\varphi_{4}||^2_h}{||\varphi_i||^2_h}$, for $1\leq i\leq 3$, are bounded on $X \setminus N$.
\end{theorem}
Note that the theorem does \emph{not} assert that $||\varphi_4||_h^2$ and $||\varphi_i||_h^2$, for $1\leq i \leq 3$, are mutually bounded, which is why the particular ordering of the $\mathcal{L}_i$'s is so important. That is, we must earn a nonzero lower bound for the ratio $||\varphi_4||^2_h/||\varphi_i||^2_h$. We shall address this issue momentarily.

\subsubsection{Reality of diagonal+complete metrics} By the $\Sp(4,\R)$-symmetries of the Higgs bundle $(\E,\varphi)$, we can deduce that the unique complete solution, relative to \emph{any} ordered splitting, is real. 
\begin{corollary}[Reality conditions]\label{Cor:Reality}
Let $(\E,\varphi)$ be the Higgs bundle \eqref{SL4CHiggsPolynomial} on $X=\C$. Fix an ordered splitting $\mathcal{C}$ from either \eqref{Cyclic1} or \eqref{Cyclic2}. 
Then the unique diagonal $\mathcal{C}$-complete solution $h=\diag(h_1,h_2,h_3,h_4)$ on $(\E,\varphi)$ is \textbf{real}, i.e., $h_3=h_2^{-1}$ and $h_4= h_1^{-1}$. 
\end{corollary}

The proof is nearly the same as \cite[Corollary 3.27]{LM25}. 
\begin{proof}
Let such $h$ be given. 
Consider the dual metric $h^{\vee} := \diag(\frac{1}{h_4},\frac{1}{h_3},\frac{1}{h_2},\frac{1}{h_1})$ on $\E$. 
One directly verifies that $h$ is $\mathcal{C}$-complete if and only if $h^{\vee}$ is $\mathcal{C}$-complete, and $h$ solves Hitchin's equation \eqref{HitchinsEquation} if and only if 
$h^{\vee}$ does. Thus, $h^{\vee}$ is a $\mathcal{C}$-complete solution and $h=h^{\vee}$ by Lemma \ref{Thm:Sp4CompleteSolution}. 
\end{proof}

We now have two harmonic metric to compare. 

\begin{definition}
Let us denote $h^{(i)}$ as the unique diagonal and $\mathcal{C}_i$-complete harmonic metric on $(\E,\varphi)$, which is automatically real by Corollary \ref{Cor:Reality}.
We shall write 
\[ h^{(i)} := \diag(h_1^{(i)},h_2^{(i)},h_3^{(i)},h_4^{(i)})\ ,\] where  $h_j^{(i)}$ is the induced metric on $\mathcal{L}_j$.
\end{definition}

Now, let us fix a harmonic metric $h \in \{h^{(1)},h^{(2)}\}$ on $(\E,\varphi)$. 
By the reality conditions of Corollary \eqref{Cor:Reality},
note that both the bundle endomorphisms $1_{ij} \in H^0(\Hom(\mathcal{L}_i,\mathcal{L}_j)\otimes \K)$,
where $ij=12$ or $ij=34$, satisfy $||1_{12}||_{h}^2 = || 1_{34}||_{h}^2$. 
Thus, we may unambiguously write $||1||_h^2$ to refer to their common value. 
Now, the three metrics under consideration are as follows: 
\begin{itemize}[itemsep=3pt]
    \item $||1||^2_{h}=h_2h_1^{-1}|dz|^2,$
    \item $||\alpha||_{h}^2= |\alpha|^2{h_2^{-2}}\,|dz|^2$, 
    \item $||\beta||_{h}^2= |\beta|^2h_1^2\,|dz|^2$.
\end{itemize}

\subsubsection{Verifying strong completeness} 

We now can finally show $h^{(1)}=h^{(2)}$ and hence deduce completeness properties of the common diagonal harmonic metric. Here, as in Definition \ref{Defn:Complete}, we let $N(K)$ denote any relatively compact open subset of $\C$ containing a given compact set $K \subset \C$. 

\begin{lemma}[The true complete solution] \label{Lem:TrueComplete}
Let $(\E,\varphi)$ be the Higgs bundle \eqref{SL4CHiggsPolynomial} on $X = \C$, with $\alpha$, $\beta$ each polynomial. Then $h^{(1)}= h^{(2)}$. 

Thus, there exists a unique diagonal harmonic metric $h$ on $(\E,\varphi)$ that is strongly complete. Moreover, $h$ is real and $||1||_h^2, ||\beta||_h^2,||\alpha||_h^2$ are mutually bounded
on $\C \setminus N(\mathcal{Z}(Q_4))$ with $|Q_{4}|^{1/2}$.
\end{lemma}

\begin{proof}
We follow the strategy of \cite[page 108]{LM20b}, but with some simplifications in our setting. 
Let $Q_{4} =\alpha\beta$ and denote $Y := \C\setminus \mathcal{Z}(Q_{4})$.  
We study the Hitchin section on $Y$. That is, consider the Higgs bundle 
$(\mathbb{K}_{Y,4}, \phi_{\mathrm{Hit}}(Q_{4}|_{Y}))$ given by 
\[ \mathbb{K}_{Y,4} =  \K_{Y}^{3/2} \oplus \K_{Y}^{1/2}\oplus \K_{Y}^{-1/2} \oplus \K_{Y}^{-3/2}, \ \]
where $\phi_{\mathrm{Hit}}$ is described by the following diagram: 
\begin{equation}\label{SL4Hitchin}
    \begin{tikzcd}
  \K_{Y}^{3/2} \arrow[r, "1"] & \K_{Y}^{1/2} \arrow[r, "1"]& \K_{Y}^{-1/2} \arrow[r, "1"]&\K_{Y}^{-3/2}\arrow[lll, bend left=20,  "Q_4"]
\end{tikzcd}.
\end{equation}
We shall write $\phi_{\mathrm{Hit}}=\sum_{i=1}^4\phi_i$ to notationally distinguish its components from the $\varphi_i$'s of $(\E,\varphi)$.

We now introduce some notation. Let $\mathrm{Harm}(Q_{4}|_{Y})$ denote the space of harmonic metrics on the Higgs bundle 
$(\mathbb{K}_{Y,4}, \phi_{\mathrm{Hit}}(Q_{4}|_{Y}))$ on the open Riemann surface $Y$. 
We add the decorations $\Harm_{\R}^{cyclic}$ to denote \emph{real, cyclic (diagonal)} solutions. 
On the other hand, we let 
$\Harm^{cyclic}_{\R}\big( (\E,\varphi )|_{Y}\big)$ denote the space of diagonal and real harmonic metrics on the restricted Higgs bundle $(\E,\varphi)|_Y$. 
We now observe that there are bijections
\[ \Upsilon_i: \Harm_{\R}^{cyclic}(Q_{4}|_{Y}) \stackrel{\simeq}{\longrightarrow}  \Harm^{cyclic}_{\R}\big((\E,\varphi )|_{Y}\big)\ ,\]
given explicitly by 
\begin{align}
    \Upsilon_1:=& \big( \diag(h_1,h_2,h_2^{-1},h_{1}^{-1}) \longmapsto \diag(\,|\alpha|h_1, |\alpha|h_2,|\alpha|^{-1}h_2^{-1} \label{Upsilon1}, |\alpha|^{-1}h_1^{-1}) \, \big)\ , \\
    \Upsilon_2:=& \big(\diag(h_1,h_2,h_2^{-1},h_{1}^{-1}) \longmapsto \diag(\, |\beta|h_2^{-1}, |\beta|h_1^{-1},|\beta|^{-1}h_1, |\beta|^{-1}h_2) \, \big)\label{Upsilon2}\ .
\end{align}
Here, one can quickly and directly verify that $\Upsilon_i$ is well-defined: that is, $\hat{h}:=\Upsilon_i(h)$ is an (obviously real, cyclic) harmonic metric on $(\mathbb{K}_{Y,4}, \phi_{\mathrm{Hit}}(Q_4))$. 

By the conversion process, we can invoke Li-Mochizuki's earlier work on the $\mathfrak{sl}_n\C$ affine Toda equations, namely the harmonicity equations for Higgs bundles $(\mathbb{K}_{Y,n}, \phi_{\mathrm{Hit}}(Q_n))$ in the Hitchin section on non-compact Riemann surfaces $Y$. 
Let us define $\hat{h}^{(i)}:=\Upsilon_i^{-1}(h^{(i)})$ and write 
\[ \hat{h}^{(i)}= \diag\bigg(\hat{h}^{(i)}_1,\hat{h}^{(i)}_2,\frac{1}{\hat{h}^{(i)}_1},\frac{1}{\hat{h}^{(i)}_2}\bigg)\ . \]
Then we can apply \cite[Proposition 3.37]{LM25} to 
$\hat{h}^{(1)},\hat{h}^{(2)} \in \Harm(Q_4|_Y)$, since the flat metric $|Q_4|^{1/2}|_{\overline{Y}_1}$, 
restricted to the bordered Riemann surface $\overline{Y}_1:=\C \setminus N(\mathcal{Z}(Q_4))$, is complete. 
Their result asserts that $\hat{h}^{(1)},\hat{h}^{(2)} \in \Harm(Q_{4}|_{Y})$ are mutually bounded 
on $\overline{Y}_1$. Moreover, each of the associated metrics $||\phi_j||_{\hat{h}^{(i)}}^2$, for $1\leq j\leq 3$ and $i \in \{1,2\}$, are mutually bounded with respect to $|Q_4|^{1/2}$ on $\overline{Y}_1$. In other words, the metrics $||\phi_j||_{\hat{h}^{(i)}}^2$ are each bi-Lipschitz with $|Q_4|^{1/2}$ on $\overline{Y}_1$.

We can now finally show $h^{(1)}=h^{(2)}$.
To see the equality, by Lemma \ref{Thm:Sp4CompleteSolution}, we need only see $h^{(2)}$ is $\mathcal{C}_1$-complete (which then implies $h^{(1)}=h^{(2)}$ is $\mathcal{C}_2$-complete). We now observe the following:
\[ ||\alpha||_{h^{(2)}}^2= ||\phi_4||^2_{\hat{h}^{(2)}}=|Q_4|^2 \frac{\hat{h}^{(2)}_1}{\hat{h}^{(2)}_4}=|Q_4|^2 \frac{\hat{h}^{(2)}_1}{\hat{h}^{(2)}_2}\frac{\hat{h}^{(2)}_2}{\hat{h}^{(2)}_3}\frac{\hat{h}^{(2)}_3}{\hat{h}^{(2)}_4}= |Q_4|^2 ||\phi_1||_{\hat{h}^{(2)}}^{-2}||\phi_2||_{\hat{h}^{(2)}}^{-2}||\phi_3||_{\hat{h}^{(2)}}^{-2}\]
Thus, $||\alpha||_{h^{(2)}}^2$ is mutually bounded with  $|Q_{4}|^{1/2}$ on $\overline{Y}_1$. Hence, $h^{(2)}$ is $\mathcal{C}_1$-complete, and by the above reasoning, $h^{(2)}=h^{(1)}$. 

We have already shown that $h:=h^{(1)}=h^{(2)}$ is a diagonal harmonic metric that is $\mathcal{C}_1$ and $\mathcal{C}_2$-complete, which is real. 
The uniqueness statement follows immediately from Lemma \ref{Thm:Sp4CompleteSolution}. The mutual boundedness of $||1||_h^2$, $||\alpha||_h^2$, $||\beta||_h^2$, with $|Q_{4}|^{1/2}$ on $\overline{Y}_1$ was proven earlier. 
\end{proof}

\subsection{Polygonal maximal surfaces}\label{Subsec:PolynomialSurfaces}

We can now deduce that the maximal surfaces $\Sigma$ associated to the polynomial Higgs bundles of interest are polygonal. Recall that a complete maximal surface in $\Hh$ or $\H$ is said to be \emph{polygonal} whenever its asymptotic boundary is a polygon in $\Einh$ or $\Ein$, respectively. 

The proof below closely follows the strategy used in \cite[Section 5.3 to 5.5]{Mor25} to show that an asymptotically flat maximal surface is indeed a polygonal surface. Here is the outline. By Lemma \ref{Lem:TrueComplete}, the induced metric $g$ of $\Sigma$ and the flat metric $|Q_{4}|^{1/2}$ are bi-Lipschitz. Hence, so are their Tits boundaries. Thus, the Tits boundary of the induced metric $g$ is compact and so $\Sigma$ is polygonal.

\begin{lemma}[Polynomial to polygon]\label{Lem:PolynomialPolygon}
    Let $(\E, \varphi(\alpha,\beta))$ be the Higgs bundle \eqref{SO(2,3)Higgs} with polynomial differentials $\alpha$, $\beta$ and quartic differential $Q_4=\mathbf{Q}(\alpha,\beta)$ of degree $n$. Let $h$ be the  diagonal, real $\SO_0(2,3)$-harmonic metric described by Lemma \ref{Lem:TrueComplete} and Proposition \ref{Prop:CorrespondenceMetrics}. Define $\Sigma \subset \Hh$ as the maximal surface associated to the harmonic bundle $(\E, \varphi, h)$ by Proposition \ref{Prop:HiggsToMaximalSurface}.
    Then the asymptotic boundary of $\Sigma$ is a semi-positive polygon in $\Einh$ with $n+4$ vertices.
\end{lemma}

\begin{proof}
    Denote by $g$ the induced metric of $\Sigma$ (which is equal to $||1||_h^2$ in the notations of the previous section) and by $g_4$ the flat metric $|Q_{4}|^{1/2}$.
    By Corollary \ref{Lem:TrueComplete}, $g$ is bi-Lipschitz to $g_4$ outside a compact set. Denote by $d$ and $d_4$ the distances associated with the metrics $g$ and $g_4$, respectively.

    Following the proof of \cite[Lemma 5.5.2]{Mor25}, the identity map is a quasi-isometry from $(\Sigma,d)$ to $(\Sigma,d_4)$. By \cite[Section 5.4]{Mor25}, $d_4$ is a Hadamard distance with compact Tits boundary of perimeter $\frac{\pi}{2}(n+4)$.
    
    Following the proof of \cite[Corollary 5.3.3]{Mor25}, the Tits boundary of $(\Sigma,d)$ is compact and bi-Lipschitz to that of $(\Sigma,d_4)$. Thus, the Tits boundary of $(\Sigma,g)$ is a circle of finite perimeter, and hence $(\Sigma,g)$ has finite total curvature.

    We can now apply \cite[Main Theorem]{Mor25} showing that $\partial_\infty\Sigma$ is a polygon with $n+4$ vertices.
\end{proof}

\begin{corollary}\label{Cor:AllPolynomialMaximalSurfaces}
    Every polygonal maximal surface $\Sigma$ in $\Hh$ is realized by a Higgs bundles $(\E,\varphi(\alpha,\beta))$ of the form \eqref{SO(2,3)Higgs} with polynomial $\alpha$ and $\beta$. 
\end{corollary}

\begin{proof}
    Let $\Sigma$ be a polygonal maximal surface. Then by \cite[Main Theorem]{Mor25} $\Sigma$ is conformally of parabolic type and admits a conformal parametrization 
    $\sigma: \C \rightarrow \Sigma$ with polynomial quartic differential $Q_4(\sigma)$. Let us write $Q_4=\hat{Q}_4\dz^4$. 
    Let $\alpha$, $\beta$ be the components of $\II^{2,0}$ of this parametrization and $\eta \in H^0(\C, \mathcal{N})$ be a non-vanishing holomorphic section of $\mathcal{N}$. Then we can write $\alpha=a(z)dz^2\otimes \eta$ and $\beta=b(z)dz^2\otimes \eta^{-1}$ with $a$, $b:\C \rightarrow \C$ holomorphic functions. Note the equality $ab=ab\mathbf{Q}(\eta,\eta^{-1})=\hat{Q}_4$, by definition of $Q_4$ and $\eta^{-1}$. 
   As $\hat{Q}_4$ is a polynomial, each of $a$ and $b$ have a finite number of zeros. 
   Thus, we can write $a=A\exp(f)$ and $b=B\exp(g)$ with $A$ and $B$ complex polynomials and $f$, $g$ holomorphic functions. Now, since $ab=AB\exp(f+g)$ is a polynomial, $\exp(f+g)$ is a constant, thus $f=-g+c$ for $c\in\C$.

    Up to multiplying the trivializing section $\eta$ by $\exp(f)$, we have $a(z)\in\C[z]$ and $b(z)\in\C[z]$. Thus, in an appropriate $\SO(2,\C)$-frame for $\N^{\C}=\mathcal{N}\oplus \mathcal{N}^{-1}$, we 
    have $\alpha$, $\beta$ each represented by polynomial functions.

    Now, by the discussion in Section \ref{Subsec:MaximaltoHiggs}, we conclude that the associated Higgs bundle to $\Sigma$ can be gauged to the form $(\E, \Phi(\alpha,\beta))$ with $\alpha,\beta$ polynomial.
\end{proof}

As bundle valued differentials, we cannot compare $\alpha$'s, $\beta$'s for different maximal surfaces. To allow for such a comparison, we introduce a procedure that extracts an equivalence class of pairs of holomorphic quadratic differentials. 
We will need the following spaces for this conversion process.
\begin{definition}[Polynomial differential pairs]
Let $\mathcal{M}_{n}(\mathcal{PD}^{o})$ and $\mathcal{M}_n(\mathcal{PD})$ be the moduli spaces of oriented, (resp. unoriented) pairs of polynomial quadratic differentials $A\dz^2, B\dz^2$, up the diagonal action of $\Aff(\C)$ by pushforward and by the $\C^*$-action $\lambda\cdot (A\dz^2, B\dz^2)= (\lambda A \dz^2, \lambda^{-1}B\dz^2)$. 
\end{definition}

Now, here is the desired construction. We use the insights from Corollary \ref{Cor:AllPolynomialMaximalSurfaces} in this definition.

\begin{definition}[Associated polynomial differentials]\label{Defn:AssociatedDifferentials}
Let $\Sigma \subset\Hh$ be a polygonal maximal surface. 
Choose any conformal parametrization $\sigma:\C \rightarrow \Sigma$ and holomorphic section $\eta \in H^0(\mathcal{N})$ such that 
\begin{align}
    \alpha &= A\dz^2\otimes \eta\ ,\\
    \beta &= B \dz^2 \otimes \eta^{-1}\ ,
\end{align}
where $A, B:\C\rightarrow \C$ are polynomials. 
Then we define $\boldsymbol{\delta}^{o}(\Sigma) := [(A\dz^2, B\dz^2)] \in \mathcal{M}_n(\mathcal{PD}^o)$.\medskip 

We set $\boldsymbol{\delta}^{o}: \mathcal{C}_{n+4}(\Hh)\rightarrow \mathcal{M}_n(\mathcal{PD}^{o})$ to be the resulting map. \medskip 

We define $\boldsymbol{\delta}: \mathcal{C}_{n+4}(\H)\rightarrow \mathcal{M}_n(\mathcal{PD})$ analogously.
\end{definition}
With $\sigma$ fixed, note that the choice of $\eta$ leads to the $\C^*$-ambiguity, while the choice of conformal parametrization $\sigma$ itself leads to the $\Aff(\C)$-ambiguity. 
\begin{remark}\label{Defn:DescendedDifferentialMap}
The maps $\boldsymbol{\delta}^{o}$, $\boldsymbol{\delta}$ are each $\SO_0(2,3)$-invariant and thus induce maps of the form $\boldsymbol{\delta}^{o}: \mathcal{M}_{n+4}(\Hh)\rightarrow \mathcal{M}_n(\mathcal{PD}^{o})$ and $\boldsymbol{\delta}: \mathcal{M}_{n+4}(\H)\rightarrow \mathcal{M}_n(\mathcal{PD})$, respectively.
\end{remark}

\section{Polygons}\label{Sec:Polygon}

In this section, we define and study \emph{polygons} of two different flavors: semi-positive polygons in $\Ein$ and spherical polygons in $\S^2$. By introducing a turning invariant for the latter, we are able to count the connected components of each respective moduli space of polygons.

\subsection{Polygons in \texorpdfstring{$\Ein$}{Ein12}}

We recall that the definition of semi-positive loops in $\Ein$ and its equivalent formulations were presented in Proposition \ref{Prop:SemiPositive}. Now, here is some additional terminology.

\begin{definition}
    A \textbf{polygon} in $\Ein$ is a loop consisting of finitely many segments of photons. In a polygon, a maximal segment of photon is called an \textbf{edge} and the point joining two consecutive edges is called a \textbf{vertex}. We refer to a polygon with $N$ vertices as an $\boldsymbol{N}$\textbf{-gon}. 
\end{definition}

Throughout Section \ref{Sec:Polygon}, we shall write $N$ to denote the number of vertices of a semi-positive polygon $\Lambda$ in $\Ein$, so that $N=n+4$ when $\Lambda$ is the asymptotic boundary of a maximal surface $\Sigma$ with $\deg(Q_4)=n$.

\subsubsection{In a splitting}

Recall from Section \ref{sec:splitting} that in a splitting,  $\Hh\simeq\S^2_+\times\S^2$ and $\Einh\simeq\S^1\times\S^2$. Now, from this perspective, we can characterize semi-positive polygons.

\begin{proposition}\label{prop:polygon piecewise isometry}
    Let $\Lambda$ be a semi-positive $N$-gon in $\Einh$. In a splitting $\Einh\simeq\S^1\times\S^2$, the polygon $\Lambda$ is the graph of an $N$-piecewise isometry from $\S^1$ to $\S^2$. Moreover, the image of this piecewise isometry does not contain antipodal points.
\end{proposition}
\begin{proof}
    The first part of the proposition follows from \cite[Lemma 2.12]{LTW24}. For the last point, this follows from the fact that there are no semi-positive 2-gons \cite[Proposition 3.1.4]{Mor25} and a length 2$\pi$ loop on $\S^2$ containing two antipodal points must be a 2-gon or a great circle.
\end{proof}

\begin{remark}[Inherited orientation]\label{rem:orientation}
    A semi-positive loop $\Lambda$ in $\Einh$ carries a canonical orientation as follows. Fix an orthogonal space+time splitting of $\R^{2,3}= U\oplus V$, inducing a splitting $\Einh \simeq \S^1\times \S^2$, where $\S^1 \simeq Q_+(U)$, the unit vectors in $U$ and $\S^2\simeq Q_-(V)$, the unit vectors in $V$. Hence, we may write 
    $\Lambda = \mathrm{graph}(f:\S^1 \rightarrow \S^2)$ for a 1-Lipschitz map $f$. Thus, $\Lambda$ obtains an orientation from $\S^1$. 
    As the group $\SO_0(2,3)$ preserves the space-orientation of $\R^{2,3}$, the resulting orientation on $\Lambda$ is independent of the orthogonal space+time splitting. 
\end{remark}

\subsection{Spherical polygons}
We begin by introducing the spherical polygons of interest, which are exactly the projections to $\S^2$ of a polygonal semi-positive loop in $\Ein$. In particular, we shall be interested in fixing only the \emph{total length} of the polygon, but not the side lengths. 

\begin{definition}\label{Defn:SphericalPolygon}
    A \textbf{spherical polygon} $\Delta$ is an oriented piecewise geodesic loop in $\S^2$ of total length $2\pi$ that does not contain a pair of antipodal points. A maximal geodesic segment $e \subset \Delta$ is called an \textbf{edge},
    and the point $v \in \Delta$ joining two consecutive edges is called a \textbf{vertex}.
\end{definition}

We record a formal remark regarding some nonstandard notation. 
\begin{remark}
A spherical polygon $\Delta$ is entirely determined by its cyclically ordered vertex set. Hence, we may write $\Delta = (v_i)_{i \in I_N}$, where the index set $I_N$ is a $\Z_N$-torsor. In other words, $\Z_N$ acts simply transitively on $I_N$, so that if $i \in I_N$, then $i+1$, $i-1 \in I_N$ are defined, but there is no canonical choice of basepoint.
\end{remark}

We now introduce two configuration spaces of spherical polygons: 
\begin{itemize}
    \item $\mathcal{C}_{N}(\S^2)$: the configuration space of all spherical $N$-gons as in Definition \ref{Defn:SphericalPolygon}. These polygons have a fixed cyclic ordering $(v_i)_{i\in I_N}$ of their vertices. 
    \item $\mathcal{C}_{N}^{pt}(\S^2)$: the configuration space of all \emph{pointed} spherical $N$-gons. These polygons have a fixed enumeration $(v_1,\dots, v_N)$ of their vertices, coming from the marked vertex $v_1$. 
\end{itemize}
We shall denote by $\mathcal{M}_{N}(\S^2)$ and $\mathcal{M}_{N}^{pt}(\S^2)$ the corresponding moduli spaces  under the quotient by $\SO(3)$. 

\subsubsection{Turning invariant}
We now introduce an invariant for spherical polygons, which will classify the connected components of the corresponding moduli space. This turning invariant $\Theta(\Delta)$ of a spherical polygon $\Delta$ is defined as follows: for any enclosing hemisphere $H$ of $\Delta$, we measure how many turns $\Delta$ makes around the boundary circle $\mathcal{C} = \partial H$. This turning integer is invariant of the chosen hemisphere $H$. 
To measure this turning, we will need both the orientation the polygon $\Delta$ and the orientation of the ambient space $\S^2$. 

To start, we first show spherical polygons are indeed trapped in a hemisphere.

\begin{proposition}[Hemisphere confinement]\label{Prop:HemisphereContainment}
Let $\Delta \subset \S^2$ be a spherical polygon. Then there is an open hemisphere $H \subset \S^2$ such that $\Delta \subset H$. 
\end{proposition}

The result follows by the same reasoning as in \cite[Lemma 2.8]{Hei05}. However, in that lemma, only \emph{strictly} 1-Lipschitz maps are considered, so we include the necessary (minor) alterations here. 
 
\begin{proof}
$\Delta$ is by definition piecewise geodesic and of length $2\pi$. Thus, an arc length parametrization $f: \S^1\rightarrow \Delta$ is a 1-Lipschitz map. Since $f$ is piecewise geodesic,  for any $u,v \in \S^1$, the condition $d_{\S^2}(f(u),f(v)) = d_{\S^1}(u,v)$ implies that $f(u),f(v)$ lie on the same edge of $\Delta$. We shall invoke this observation later.

Now, to prove the claim, it suffices to show that the convex hull of $\Delta$ in $\R^3$ does not contain the origin. To prove this, we proceed by contradiction. 
Suppose the convex hull of $\Delta$ does contain the origin, and there are points $u_1,\ldots,u_k \in \Delta$ and scalars $\lambda_1,\ldots,\lambda_k \in [0,1]$ such that $\sum_{i=1}^k \lambda_k=1$ such that
\begin{align}\label{eq:EqualsZero}
     \sum_{i=1}^k \lambda_if(u_i)=0\ .
\end{align}

The fact that $f$ is 1-Lipschitz ensures that for all $u,v \in \S^1$, 
    \[d_{\mathbf{S}^2}(f(u),f(v)) \leq d_{\mathbf{S}^1}(u,v)\ ,\]
hence, applying the inverse of cosine, 
    \[\langle f(u),f(v)\rangle \geq \langle u,v \rangle\ . \]
Now, we observe that $\langle f(u_i),f(u_j)\rangle > \langle u_i, u_j\rangle$ must occur for some index $i\neq j$. Indeed, otherwise, by our observation from the beginning of the proof, we would find that all the points $(f(u_i))_{i=1}^k$ lie on the same edge of $\Delta$, which is impossible according to \eqref{eq:EqualsZero} and Definition \ref{Defn:SphericalPolygon}.

Fix any index $i$, then observe that 
    \[\sum_{j=1}^k\langle \lambda_i u_i,\lambda_j u_j\rangle \leq \sum_{j=1}^k\langle\lambda_if(u_i),\lambda_jf(u_j)\rangle=0\ .\]
By our earlier observation, for at least one index $i$, the above inequality must be strict.
Hence, 
    \[\|\lambda_1u_1+\cdots+\lambda_ku_k\|^2=\sum_{i=1}^k\sum_{j=1}^k\langle \lambda_iu_i,\lambda_ju_j\rangle<0\ ,\]
a contradiction.

Thus, the convex hull of $\Delta$ does not contain the origin and the proof is complete.
\end{proof}

Given a spherical polygon and an enclosing hemisphere, we can define a projection of the polygon to the boundary of the hemisphere. 
\begin{definition}[Boundary Projection]
Let $\Delta = (v_i)_{i\in I_N}$ be a spherical $N$-gon and suppose $\Delta \subset H$ for an open hemisphere $H$. Denote $\mathcal{C} = \partial H$ and define a projection map $\pr_{\mathcal{C}}: \Delta \rightarrow \mathcal{C}$ 
as follows: let $\pr_{\mathcal{C}} (v_i)$ be the first intersection of the geodesic ray from $v_i$ through $v_{i+1}$ with $\mathcal{C}$. 
Then the $\boldsymbol{H}$\textbf{-boundary projection} $\pr_{\mathcal{C}}(\Delta) $ is the cyclically-ordered $N$-tuple $(\pr_{\mathcal{C}}(v_i))_{i \in I_N} \subset \mathcal{C}$. See Figure \ref{fig:boundaryprojection}.
\end{definition}

\begin{figure}[h!]
    \centering
    \includegraphics[width=0.45\linewidth]{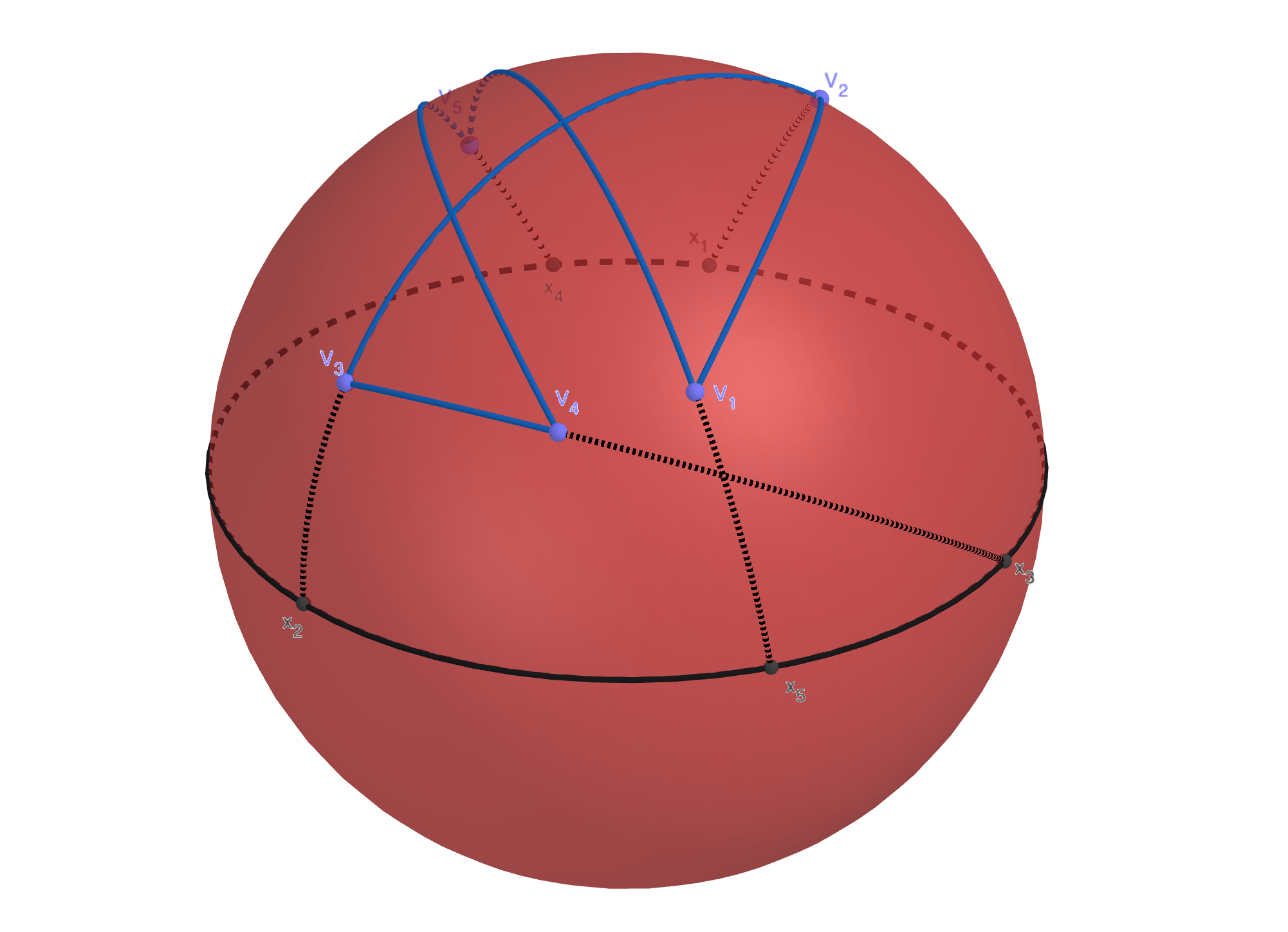}
    \caption{Boundary projection for a spherical 5-gon, that can be manipulated at \href{https://www.geogebra.org/classic/eum7urds}{Geogebra}.}
    \label{fig:boundaryprojection}
\end{figure}

Observe that $\pr_{\mathcal{C}}(v_i)\neq \pr_{\mathcal{C}}(v_{i+1})$ by definition. 
We can endow $\mathcal{C}$ with the unique orientation such that $\Delta$ lies in the positive direction of the normal vector of $\mathcal{C}$ in $\S^2$. 
Thus, given a pair $(\Delta,H)$, we obtain a cyclically ordered $N$-tuple of points on $\S^1$ such that neighboring points are distinct. There is a combinatorial invariant associated to such tuples as follows. We thank Timoth\'e Lemistre for his suggestion of this definition, which simplified our equivalent original notion. 

\begin{definition}
Let $\mathcal{C}_N(\S^1)$ denote the space of 
cyclically ordered $N$-tuples $\mathbf{x}=(x_i)_{i \in I_N}$ of points in $\S^1$ such that 
$x_i\neq x_{i+1}$. Define a map $\boldsymbol{\Theta}: \mathcal{C}_N(\S^1) \rightarrow \{1,2,\dots, N-1\}$ as follows. Identify $\S^1 \cong \R/\Z$ and let $\mathbf{x}=(x_i)_{i \in I_N} \in \mathcal{C}_N(\S^1)$ be arbitrary. For any index $i\in I_N$, define $a_i\in (0,1)$ as the unique representative of $(x_{i+1}-x_i)\in \R/\Z$ in $(0,1)$. Finally, we set 
\[ \Theta(\mathbf{x}) := \sum_{i\in I_N} a_i\ .\]
Note that $\Theta$ is an integer and $1\leq \Theta \leq N-1$ by definition. 
\end{definition}
In other words, $\Theta$ unrolls the configuration of points along the universal cover $\R$ of $\S^1$. 

We now define a `relative' turning invariant of a spherical polygon as advertised: it is the turning invariant of its boundary projection.  

\begin{definition}
Let $\Delta \subset \S^2$ be a spherical polygon. Then we define the \textbf{turning invariant of $\boldsymbol{\Delta}$} to be $\Theta(\Delta) := \Theta(\pr_{\partial H}(\Delta))$ for $H \subset \S^2$ an enclosing hemisphere of $\Delta$.
\end{definition}

The turning invariant of a spherical polygon is independent of the chosen enclosing hemisphere $H$.
\begin{lemma}[Well-definedness of turning invariant]
Let $\Delta\subset \S^2$ be a polygon. Then for any enclosing hemispheres $H_1$, $H_2$ of $\Delta$, we have $\Theta(\pr_{\partial H_1}(\Delta))=\Theta(\pr_{\partial H_2}(\Delta))$. 
\end{lemma}

\begin{proof}
Let $K$ be the convex hull of $\Delta$ in $\S^2$ and $H_1$, $H_2$ be any two distinct hemispheres enclosing $\Delta$. Define $C_i=\partial H_i$. Then $\S^2 \setminus (C_1\cup C_2)$ has four connected components and $K$ is contained in one of them. 

Write $C_i = \S^2 \cap P_i$ for two-planes $P_i \subset \R^3$. Then take $v \in P_1 \cap P_2$ and write 
$P_1 = \spann \{v,w_1\}$ and $P_2 =\spann \{ v,w_2\}$. Then consider the homotopies $C_t =\S^2 \cap P_t$ for $P_t = \spann \{v, (1-t)w_1+tw_2\}$ and $C_t'= \S^2 \cap P_t'$, where $P'_t=\spann\{ v, (1-t)w_1-tw_2\}$. Either $C_t$ or $C_t'$, say $C_t$, will not intersect $K$ and hence provides a homotopy of great circles not intersecting $K$ from $C_1$ to $C_2$. The integer $\Theta(\pr_{C_t}(\Delta))$ varies continuously in $t$ and hence $\Theta(\pr_{\partial H_1}(\Delta))=\Theta(\pr_{\partial H_2}(\Delta))$. 
\end{proof}

We now extend the notion of turning invariant to semi-positive polygons in $\Einh$. 

\begin{definition}
Let $\Lambda \subset \Einh$ be a semi-positive polygon. Then relative to any splitting $\Einh\simeq \S^1 \times \S^2$, we set
$\Theta(\Lambda):=\Theta(\Delta)$, where $\Delta$ is the projection of $\Lambda$ to $\S^2$. The resulting value is independent of splitting, by similar reasoning as in Remark \ref{rem:orientation}. 
\end{definition}

\subsubsection{Turning invariant under basic transformations}

\begin{proposition}[Turning invariant under antipodal map]\label{prop:turningantipodal}
    Let $\Delta\subset\S^2$ be a spherical $N$-gon and let $-\Delta$ be the antipodal polygon. Then $\Theta(-\Delta)=N-\Theta(\Delta)$.
\end{proposition}
\begin{proof}
    If $\Delta$ is enclosed in the hemisphere $H$, then $-\Delta$ will be enclosed in the complementary hemisphere $-H$, thus we can look at the boundary projections of $\Delta$ and $-\Delta$ on the same circle $C=\partial H$. By construction, the $C$-boundary projections of $\Delta$ and $-\Delta$ are antipodal, hence they are images of each other by an orientation preserving rotation.

    Finally, note that the orientation of $C$ seen from $H$ is opposite to its orientation seen from $-H$, thus $\Theta(\Delta) +\Theta(-\Delta)= N$.
\end{proof}
\begin{remark}[$\Theta$ does not descend]
Proposition \ref{prop:turningantipodal} implies that the turning invariant is only well-defined in $\Einh$. Indeed, if $\Lambda =\mathrm{graph}(f:\S^1\rightarrow \S^2)$ is a semi-positive $N$-gon, then the antipodal polygon $-\Lambda $ is the graph 
of $\overline{f}: \S^1 \rightarrow \S^2$ given by $\overline{f}=-f(-x)$. In particular, 
$\Theta(-\Lambda) = N-\Theta(\Lambda)$. 
\end{remark}
\begin{proposition}[Turning invariant under reflection]\label{Prop:ThetaReflection}
    Let $\Delta\subset\S^2$ be a spherical $N$-gon, enclosed in hemisphere $H$ and $\overline\Delta$ be the reflection of $\Delta$ across $\partial H$. Then $\Theta(\overline{\Delta})= N-\Theta(\Delta)$.
\end{proposition}

\begin{proof}
    The reflection is the composition of the antipodal map and of a rotation in $\SO(3,\R)$. Thus, it has the same effect as the antipodal map.
\end{proof}

We now highlight that the turning invariant depends on the  orientation of a given polygon. 

\begin{proposition}[Turning-invariant under reversal]
Let $\Delta \subset \S^2$ be a cyclically ordered $N$-gon and $\Delta^- $ its reversal. Then $\Theta(\Delta)+ \Theta(\Delta^-)=N$. 
\end{proposition}

\begin{proof}
    Denote by $(x_i)_{i\in I_N}$ the boundary projection of $\Delta$. Then the boundary projection of $\Delta^-$ is exactly $(-x_{i})_{i\in I_N}$. Hence $\Theta(\Delta)+\Theta(\Delta^-)=N$.
\end{proof}

\subsubsection{Minimal and maximal values of turning invariant}
Finally, we show the turning invariant $\Theta$ obeys a mild restriction on spherical polygons. 

\begin{proposition}
Let $\Delta \subset \S^2$ be a spherical $N$-gon. Then $2 \leq \Theta(\Delta) \leq N-2$. 
\end{proposition}

\begin{proof}
Let $\Delta$ be given and $H \subset \S^2$ be any enclosing hemisphere. If $\Theta(\Delta)=N-1$, then the reflected polygon $\overline{\Delta}$ satisfies $\Theta(\overline{\Delta}) =1$ by Proposition \ref{Prop:ThetaReflection}. Thus, it suffices to show that $\Theta(\Delta) =1$ never transpires. 

Now, suppose $\Theta(\Delta) = 1$. In this case, it means the boundary projection $\pr_{\partial H}(\Delta) = (v_i)_{i\in I_N}$ has the following property: the vertices $v_1, v_2,\dots,v_N$ are pairwise distinct and appear in this cyclic ordering along $\partial H$. However, this means that $\Delta$ is convex and is traversed only once (no self-intersection). The only convex curve of length $2\pi$ in $\S^1$, without self-intersection, is a great circle. In particular, such a curve is never a spherical $N$-gon under Definition \ref{Defn:SphericalPolygon}.
\end{proof}

\subsection{Symmetric polygons}

We now consider symmetric polygons in the Einstein universe and in the sphere. We begin with the former. 

\begin{definition}
    A semi-positive polygon with $N$ vertices in $\Einh$ is said to be \textbf{symmetric} when its $\SO_0(2,3)$-stabilizer is cyclic of order $N$.
\end{definition}

\begin{remark}
    A semi-positive polygon is cyclic whenever its $\SO_0(2,3)$-stabilizer is the biggest possible.
\end{remark}

We will show that symmetric semi-positive polygons in $\Einh$ exist and are classified by their turning invariant, up to ambient isometries.

\begin{proposition}\label{Prop:SymmetricLightlikePolygon}
    Given integers $N \geq 4$ and $\Theta \in \{2,\ldots,N-2\}$, there is a unique symmetric semi-positive $N$-gon in $\Einh$ with  
    turning invariant $\Theta$, up to the $\SO_0(2,3)$-action.
\end{proposition}

For this aim, we will show that symmetric spherical polygons are classified by the turning invariant, and then extend this result to semi-positive polygons in $\Einh$.

\subsubsection{Symmetric spherical polygons}

\begin{proposition}\label{Prop:SymmetricSphericalPolygon}
    Given integers $N \geq 4$ and $\Theta$ in $\{2,\ldots,N-2\}$, there exists a unique symmetric spherical polygon $\Delta \in \mathcal{C}_{N}(\S^2)$ with turning invariant $\Theta$, up to $\SO(3,\R)$. 
\end{proposition}

Recall that $\S^2$ is the unit 2-sphere in $\R^3$, so that $\S^1$ identifies with the unit circle in $\C$ and by stereographic projection with respect to the south pole there is a preferred meridian $\mathbf{m}$ on $\S^2$ and $\S^2_+$: the projection of $\R\subset\C$. Thus, complex conjugation is a well-defined map on each of $\S^1$, $\S^2_+$, and $\S^2$.

\begin{proof}
Let us denote by $(m_i)_{i=1}^N$ the ordered sequence of equidistributed meridians of $\S^2$ starting at the preferred meridian $m_1=\mathbf{m}$. Let $r$ be a positive number smaller than or equal to $\pi/2$. If $C_r$ denotes the radius $r$ circle in $\S^2$, centered at the north pole, then for every index $j$ the point $v_{j,r}:=C_r\cap m_j$ is a point of the upper hemisphere. We then define a piecewise geodesic loop $P_r$ with vertices $(v_{\Theta j,r})_{j=1}^N$ and edges the minimizing geodesic segments between neighboring vertices. 
Let $\mathcal{C}$ be the equator of $\S^2$. By construction, $\Theta(\pr_{\mathcal{C}}(P_r)) = \Theta$. For an appropriate choice of $r$,
we will show $P_r$ is our desired symmetric $N$-gon. 

The function $(0,\pi/2)\to (0,2\pi\Theta)$ given by $r\mapsto \mathrm{Length}(P_r)$ is continuous, increasing, and bijective. Thus, there is a unique positive number $R < \pi/2$ such that $P_r$ has length $2\pi$ and thus is a spherical $N$-gon. We claim that $P_{R}$ is symmetric. Indeed, the rotation of angle $2\pi\Theta/N$ centered at the north pole sends the point $v_{\Theta j,R}$ to $v_{\Theta(j+1),R}$ and hence restricts to a cyclic symmetry of $P_R$. Hence, $P_R$ is a symmetric $N$-gon of turning invariant $\Theta$. 

It is straightforward to show that every symmetric spherical polygon $\Delta^s$ is of this form, because a generator of $g \in \mathrm{Aut}(\Delta^s) \subset \SO(3)$ must send a vertex $v_i$ of such a polygon to $v_{i+1}$, thus the vertices $v_i$ are equidistributed along a circle on $\S^2$. (As $2\leq\Theta \leq N-2$, this circle is necessarily \emph{not} a great circle.)

For uniqueness, let $\Delta$ be a symmetric spherical polygon, which we can assume is centered on the north pole, up to the action of an element of $\SO(3,\R)$. By the previous paragraph, then $\Delta$ has its vertices equidistributed along a circle in the upper hemisphere. Hence, $\Delta$ obtains the form $\Delta = P_{r, \Theta}$ as described in the first two paragraphs. \end{proof}

\begin{remark}
By Proposition \ref{Prop:ThetaReflection}, we see that symmetric spherical polygons are truly unique only up to $\SO(3,\R)$ and not $\O(3,\R)$. Indeed, reflections in $\R^3$ exchange $\Theta$ and $N-\Theta$ for a $N$-gon. 
\end{remark}

\begin{figure}[h!]
    \centering
    \includegraphics[width=0.35\linewidth]{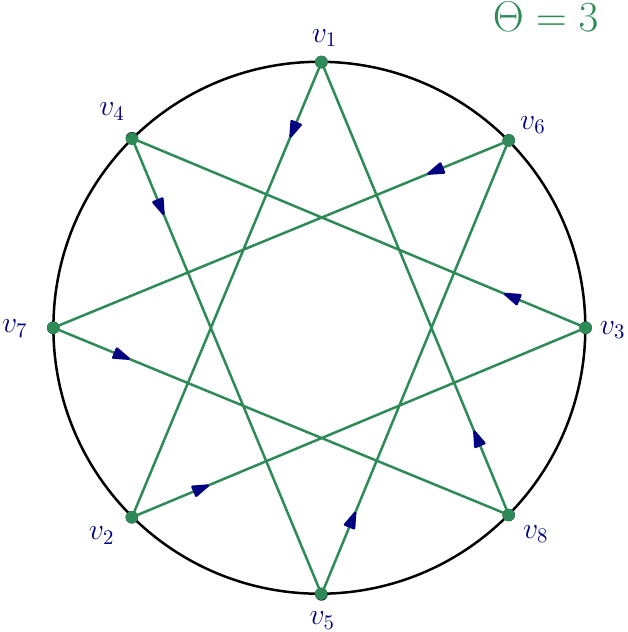}
    \caption{A representation of a symmetric octagon with turning invariant $3$.}
    \label{fig:symmetric octogon}
\end{figure}

\begin{proof}[Proof of Proposition \ref{Prop:SymmetricLightlikePolygon}, existence]
    Fix a splitting $\Einh \simeq \S^1\times\S^2$. Denote by $(u_k)_{k=1}^N$ the points:
    \[u_k:=\exp(2i\pi k/N)\in\S^1\ .\]
    Denote by $\Delta$ the symmetric spherical polygon constructed in the proof of Proposition \ref{Prop:SymmetricSphericalPolygon} above, with $N$ vertices and turning invariant $\Theta$. Denote by $(v_k)_{k=1}^N$ the ordered vertices of $\Delta$. The points $((u_k,v_k))_{k=1}^N$ are then the ordered vertices of a semi-positive polygon $\Lambda$ of $\Einh$.

    We claim that $\Lambda$ is symmetric. Denote by $\r \in \SO_0(2,3)$ the following isometry, in the splitting $\S^1\times\S^2$:
    \[\r(u,v) := \left(R_{2\pi/N}u,R_{2\pi\Theta/N}v\right)\ ,\]
    where $R_\theta$ denotes the rotation of angle $\theta$ of $\S^1$ or $\S^2$, as appropriate, that is centered at the north pole in the latter case. Then, by construction, $\Lambda$ is stabilized by $\r$, hence by the cyclic subgroup of $\SO_0(2,3)$ of order $N$ generated by $\mathbf{r}$ that acts faithfully on $\Lambda$.
\end{proof}

\subsubsection{Uniqueness of symmetric polygons}

\begin{proof}[Proof of Proposition \ref{Prop:SymmetricLightlikePolygon}, uniqueness]
    Let $\Lambda$ be a symmetric semi-positive polygon in $\Einh$ with $N$ vertices and turning invariant $\Theta$. Denote by $\Sigma$ the complete maximal surface in $\Hh$ bounded by $\Lambda$. Denote by $\r$ an order $N$ symmetry of $\Lambda$, generating the cyclic group of symmetries of $\Lambda$. Since $\r$ stabilizes $\Lambda$, it stabilizes $\Sigma$, and thus is a finite order isometry of $\Sigma$. Now, by \cite{Cheng} (or \cite{LT23}), $\Sigma$ is non-positively curved, hence by Cartan fixed point theorem (see \cite[Theorem 6.2.3]{PetersenRiemannianGeometry} for example) $\r_{|\Sigma}$ has a fixed point. Since the order of $\r$ is $N\geq 4$, the fixed point is unique. Indeed, if $\r$ fixed at least two points, then it would necessarily
    fix a geodesic and hence have order at most two. We now denote by $o$ the unique fixed point of $\r_{|\Sigma}$.

    Take the splitting of $\Hh$ corresponding to the orthogonal decomposition $\R^{2,3}=\T_o\Sigma\oplus\left(\R o\oplus\N_o\Sigma\right)$. Hence $\Hh$ is identified with $\S^2_+\times\S^2$ where $o$ corresponds to the point $(\hat{k},\hat{k})$ for $\hat{k}$ the north pole of $\S^2_+$ and $\S^2$. Since $\r$ fixes $o$, it stabilizes the two factors $\S^2_+$ and $\S^2$.

     Note that $\r_{|\Sigma}$ is determined by $\r_{|\S^2_+}$ through the derivative of $\r_{|\Sigma}$ at $o$. As an order $N$ isometry of $\Sigma$, the map $\r_{|\S^2_+}$ must be a rotation of angle $2\pi/N$ on $\S^2_+$ corresponding to a rotation in $\T_{o}\Sigma$. When restricted to $\S^2$, $\mathbf{r}$ is a counterclockwise rotation of a certain angle $\theta$ in $\frac{2\pi}{N}\Z$, centered at north pole. Write the vertices of $\Lambda$ as $((u_i,v_i))_{i\in I_N}$ in cyclic order. Then $\r(u_i,v_i)=(u_{i+1},v_{i+1})$. The turning invariant $\Theta$ measures exactly how many turns $(v_i)_{i=1}^N$ makes around the north pole. 
     Hence, $\theta= \frac{2\pi\Theta}{N}$. 
     Thus, up to rotations of both factors, preserving both the space and time-orientations, $\Lambda$ is exactly the symmetric polygon obtained in the previous paragraph, concluding the proof.
\end{proof}
\subsubsection{Differentials associated to symmetric polygons} Recall Definition \ref{Defn:AssociatedDifferentials} of the map $\boldsymbol{\delta}^{o}$ that associates to a polygonal maximal surface in $\Hh$ a pair of polynomial quadratic differentials, up to $\Aff(\C)\times\C^*$.
\begin{corollary}[Polynomials for symmetric polygons]\label{cor:SymmetricMonomials}
    Let $\Lambda \in \mathcal{C}_{n+4}(\Einh)$ be a symmetric semi-positive polygon with turning invariant $\Theta$, and $\Sigma \subset \Hh$ be the unique complete maximal surface with $\partial_{\infty}\Sigma = \Lambda$. 
    Then $\Sigma$ has differentials $\boldsymbol{\delta}^{o}(\Sigma) = [(z^{\Theta-2}\dz^2, z^{n-\Theta+2}\dz^2)]$.
\end{corollary}
\begin{proof}

    Let $\r: \R^{2,3} \rightarrow \R^{2,3}$ denote the rotation stabilizing $\Lambda$ and $o \in \Sigma$ denote the unique fixed point of $\r_{|\Sigma}$. Taking the decomposition $\R^{2,3}=\T_o\oplus(\R o\oplus\N_o)$, we obtain a splitting of $\Hh$ as $\S^2_+\times\S^2$ in which $\r_{|\S^2_+}$ is a rotation of angle $2\pi/(n+4)$ and $\r_{|\S^2}$ is a rotation of angle $2\pi\Theta/(n+4)$. Writing $\Sigma=\mathrm{graph}(\varphi)$, for $\varphi:\S^2_+\to\S^2$, we obtain from $\r\cdot\Sigma=\Sigma$ that
    \begin{equation}\label{eq: phi and R}
        R_{\frac{2\pi\Theta}{n+4}}\circ\varphi=\varphi\circ R_{\frac{2\pi}{n+4}}\ .
    \end{equation}
    
    Let $A$ and $B$ be polynomial representatives of $\alpha$ and  $\beta$, so that $\boldsymbol{\delta}^o(\Sigma)=[(A\dz^2,B\dz^2)]$. First, note that $Q_4=AB\dz^4$ vanishes only at the center $o$ of $\Sigma$. Indeed, if $Q_4$ has another zero $z_0 \neq 0$, then the $\r$-orbit of $z_0$  would provide $n+4$ distinct zeros, yet $\deg(Q_4)=n$. Thus, in any conformal parametrization where $\alpha$ and $\beta$ are polynomials, they must be monomials. We now determine the degrees of $\alpha$ and $\beta$ as functions of $\Theta$ and $n$.

    Let $\sigma:\C\to\Hh$ denote a conformal maximal embedding of $\Sigma$ sending $0$ to $o$.
    The map $\r$, restricted to $\Sigma$, is an isometry. Thus,  $f=\sigma^{-1}\circ\r\circ\sigma$ is a biholomorphism from $\C$ to $\C$. Hence, $f$ is an affine map, which
    we can express as $f=a\id + b$. 
    Since $\sigma(0)=o$ and $\r(o)=o$, we have $b=0$. Finally, $a=f_{z}(0)=\exp(2i\pi/n+4)$, meaning that 
    \begin{equation}\label{eq: sigma, r and R}
        \r\circ\sigma=\sigma\circ R_{\frac{2\pi}{n+4}}\ .
    \end{equation}

    Now, let $\eta$ be a holomorphic trivialization of $\mathcal{N}$, so that $\alpha=A\dz^2\otimes\eta$ and $\beta=B\dz^2\otimes\eta^{-1}$, with $A$ and $B$ polynomial.
    We denote $R=R_{\frac{2\pi}{n+4}}: \C \rightarrow \C$ going forward. 
    Since $\r: \Sigma \rightarrow \Sigma$ and $R: \C \rightarrow \C$ are isometries, the section 
    $\eta'$ given by $\eta'=d\r\circ \eta \circ R^{-1}$ satisfies $\eta' \in H^0(\C, \mathcal{N})$. Hence, $\eta' = w \cdot \eta$ for some holomorphic function $w:\C \rightarrow \C$. 
    Using that $||d\r||=1$, we find $|w(z)|^2= ||\eta(z)||_h^2/||\eta(R^{-1}z)||_h^2$. 
    Now, we apply the procedure from Section \ref{Subsec:MaximaltoHiggs} to describe $\Sigma$ by its associated strongly complete harmonic bundle $(\E,\varphi, h)$, where $h$ is obtained from Lemma \ref{Lem:TrueComplete}. Up to scalars, we find that 
    $\eta$ satisfies $||\eta||_h^2 =h_1^{-1}$, using the notation of Proposition \ref{prop:HitchinsEquationsSO(2,3)} to express $h$. 
    By Lemma \ref{Lem:TrueComplete} and the fact that Hitchin's equations are radially symmetric in this case, the function 
    $h_1:\C \rightarrow \R$ is radially symmetric. In particular,
    this means $1=|w(z)|^2$, so that 
    $\eta' = \lambda \eta$ for some unit scalar $\lambda$. 
    Since $\r(o)=o$ and $d_o\r_{|\N_o\Sigma}=R_{\frac{2\pi\Theta}{n+4}}$, we obtain
    \begin{equation}\label{eq: eta r}
       \eta' =\exp\bigg(\frac{2i\pi\Theta}{n+4}\bigg)\eta .
    \end{equation}
    Finally, use Equations \eqref{eq: sigma, r and R} and \eqref{eq: eta r} to conclude. Indeed, denoting $A(z)=z^k$ and $\zeta=\exp(2i\pi/n+4)$, 
    \[ \zeta^{\Theta} \zeta^{-k-2} z^k \dz^2 \otimes \eta = (R_{}^{-1})_*(z^k\dz^2) \otimes \eta' = \alpha(\r \circ \sigma \circ R^{-1}) = \alpha(\sigma)= z^k\dz^2 \otimes \eta\ .\]
    Hence, we conclude $k= \Theta-2$.
    \end{proof}
    
\subsection{Connected components of polygons in \texorpdfstring{$\mathbf{Ein}^{1,2}$}{Ein(1,2)}}

Denote by $\mathcal{C}_N(\S^2)$ the configuration space of spherical polygons with $N$ vertices, endowed with the topology coming from the following distance function: if $\Delta$, $\Delta' \in \mathcal{C}_{N}(\S^2)$ have vertex sets $(v_i)_{i\in \Z_N}$ and $(v_i')_{i\in \Z_N}$, respectively, then
\[\mathrm{dist}(\Delta,\Delta'):= \min\bigg\{\sum_{i=1}^N d_{\S^2}(v_i,v_{i+k}')\ \big|\ k\in\{1,\ldots,N\}\bigg\}\ .\]
That is to say: we take the minimum distance between any of the corresponding marked polygons. 

\subsubsection{Counting symmetric polygons}

As shown in Proposition \ref{Prop:SymmetricSphericalPolygon}, the symmetric spherical polygons are classified, up to $\SO(3,\R)$, by their number of vertices and turning invariant. Thus, there are $N-3$ distinct symmetric spherical polygons in $\mathcal{C}_N(\S^2)$. Since they all have different turning invariant, and since the turning invariant is invariant by the $\SO(3,\R)$ action, there are at least $N-3$ distinct connected components in $\mathcal{M}_N(\S^2) =\mathcal{C}_N(\S^2)/\SO(3,\R)$.

\subsubsection{Counting connected components}

We first characterize the Higgs bundle data corresponding to symmetric polygons. Then we 
count the components of all semi-positive polygons.
\begin{proposition}
[Characterizing symmetric polygons]\label{Prop:HiggsSymmetricPolygon}
 The asymptotic boundary $\Lambda \subset \Einh$ of a maximal surface $\Sigma\subset\Hh$ with associated polynomials
 $\boldsymbol{\delta}^o(\Sigma) = [(z^k\dz^2, z^{n-k}\dz^2)]$ is symmetric. 
 Conversely, every symmetric semi-positive polygon arises in this way. 
\end{proposition}

\begin{proof}
    By hypothesis, there is a conformal parametrization $\sigma:\C\to \Hh$ of $\Sigma$ and a non-vanishing holomorphic section $\eta \in H^0(\C,\mathcal{N})$ such that 
    $\alpha=z^kdz^2\otimes\eta$ and $\beta=z^{n-k}dz^2\otimes \eta^{-1}$. The composition $\sigma\circ \zeta$ of $\sigma$ and the homothety by the primitive $(n+4)^{th}$ root of unity $\zeta=\exp(2i\pi/(n+4))$ yields a maximal surface $\Sigma'$ with differentials $\alpha'=\zeta^{k+2}z^kdz^2\otimes\eta$ and $\beta'=\zeta^{n-k+2}z^{n-k}dz^2\otimes\eta^{-1}$. Now, up to changing the holomorphic trivialization of $\mathcal{N}$ by the multiplication by $\zeta^{-k-2}$, we obtain the same associated differentials. Thus $\Sigma$ and $\Sigma'$ are isometric, and there is an element $\gamma$ of $\SO_0(2,3)$ so that 
    $\sigma\circ\zeta=\gamma\circ\sigma$.
    One finds that $\gamma$ is an order $(n+4)$ symmetry of $\Sigma$, hence an order $(n+4)$ symmetry of $\partial\Sigma$, making $\partial\Sigma$ a symmetric polygon.
    
    Conversely, by Proposition \ref{Prop:SymmetricLightlikePolygon} and Corollary \ref{cor:SymmetricMonomials}, a complete maximal surface in $\Hh$ bounded by a symmetric $N$-gon $\Lambda$ with $\Theta= \Theta(\Lambda)$ has differentials $\boldsymbol{\delta}^o(\Sigma)=[(z^{\Theta-2}\dz^2, z^{n-\Theta+2}\dz^2)]$.
\end{proof}

\begin{theorem}\label{thm:EinhPolygonCount}
    Let $N \geq 4$. The configuration space $\mathcal{C}_N(\Einh)$ of semi-positive $N$-gons has $N-3$ connected components, which are labeled by the turning invariant $\Theta$. 
\end{theorem}

\begin{proof}
    We know that $\mathcal{C}_N(\Einh)$ has at least $N-3$ connected components, each of which contains a symmetric polygon. Now it suffices to show that each connected component contains a symmetric polygon.
    
    Let $\Lambda$ be a semi-positive polygon in $\Einh$ with $\Theta(\Lambda)=\Theta_0$, 
    and $\Sigma$ be the complete maximal surface of $\Hh$ bounded by $\Lambda$. 
    Choose any conformal parametrization
    $\sigma: \C \rightarrow \Sigma$. 
    Then $\sigma$ is associated with a Higgs bundle of the form \eqref{SO(2,3)Higgs} with polynomial quartic differential $Q_4=\alpha\beta$ of degree $n=N-4$, by Corollary \ref{Cor:AllPolynomialMaximalSurfaces}. Let $\alpha_{QD}=A\dz^2$, $\beta_{QD} =B\dz^2$ be the associated holomorphic quadratic differentials relative to a background dual trivialization of $\mathcal{N}$, $\mathcal{N}^{-1}$ for $\sigma$. Set $\delta := (\alpha_{QD},\beta_{QD})$.

    Now, we construct a continuous path $(\delta_{t})_{t\in[0,1]}$ in $(\C_{\Theta_0-2}[z]\mathrm{d}z^2, \C_{n-\Theta_0+2}[z]\dz^2)$ between the current differentials $\delta$
    and the symmetric differentials $\delta^s=(z^{\Theta_0-2}\dz^2, z^{n-\Theta_0+2}\dz^2)$. 
    We upgrade the Higgs bundle 
    $(\E,\varphi(\delta_t))$ to a harmonic bundle 
    $(\E,\varphi(\delta_t),h)$ by Lemma \ref{Lem:TrueComplete}. By Lemma \ref{Prop:HarmonicMetricToHiggs}, we have an $\SO_0(2,3)$-orbit $[\Sigma_t]$ of associated surfaces in $\Hh$. By the continuity
    of the resulting map $\mathcal{PD}^{o}_n \rightarrow \mathcal{M}_{n+4}(\Einh)$,
    verified in Lemma \ref{Lem:ContinuityPolynomialPolygon}, Corollary \ref{Cor:ContinuityAsymptoticBoundary}, 
    the boundaries $\Lambda_t=\partial_{\infty}\Sigma_t$ then define a continuous path from $\Lambda$ to a new semi-positive loop $\Lambda^s = \Lambda_1$ (up to gauging $\Lambda^s$ within its $\SO_0(2,3)$ orbit). By Proposition \ref{Prop:HiggsSymmetricPolygon}, $\Lambda^{s}$ is symmetric, and the proof is thus complete. 
\end{proof}

By similar reasoning in projecting to the 2-sphere, we find that the $\Theta$-invariant classifies the components of spherical polygons.
\begin{corollary}\label{Cor:ComponentCountSpherical}
The configuration space $\mathcal{C}_{N}(\S^2)$ of spherical $N$-gons has $N-3$ connected components, which are labeled by the turning invariant $\Theta$.
\end{corollary}

After quotienting by the connected group $\SO_0(2,3)$, we find the following. 
\begin{corollary}\label{Cor:ConnectedComponentCountUp}
The moduli space $\mathcal{M}_N(\Einh)$ has $N-3$ connected components, labeled by $\Theta$.
\end{corollary}

Note that the antipodal map exchanges the $\Theta =k$ component and the $\Theta = N-k$ component in $\mathcal{C}_{N}(\Einh)$. Hence,
we lose roughly half of the components after descending to $\Ein$. 

\begin{corollary}\label{Cor:ConnectedComponentCountDown}
The moduli space $\mathcal{M}_N(\Ein)$ has $\lceil \frac{N-3}{2} \rceil$ connected components, labeled by equivalence classes in $\{2,3,\dots, N-2\}/\sim$, where $k \sim N-k$. 
\end{corollary}

We note that Corollary \ref{Cor:ConnectedComponentCountDown} agrees with results in \cite{TW24} for $N=5$ and $N=6$.  \medskip 

For later, we shall consider \emph{pointed} polygons, which carry a marked vertex. Let $\mathcal{C}^{pt}_N(\cdot)$ and $\mathcal{M}_{N}^{pt}(\cdot)$ denote the configuration space and moduli space of pointed semi-positive polygons in $\Ein$ or $\Einh$. 
We quickly note that the above count applies also for pointed polygons. 
\begin{corollary}[Count of pointed polygons] \label{Cor:ComponentCountDownPointed}
The moduli space $\mathcal{M}_N^{pt}(\Ein)$ has $\lceil \frac{N-3}{2} \rceil$ connected components, labeled by equivalence classes in $\{2,3,\dots, N-2\}/\sim$, where $k \sim N-k$. 
\end{corollary}

\begin{proof}
It suffices to prove that $\mathcal{C}_N^{pt}(\Ein)$ has $\lceil \frac{N-3}{2} \rceil$ connected components, as the group $\SO_0(2,3)$ is connected. Now, the count for pointed polygons follows by similar reasoning as above. First, each pointed semi-positive polygon $\Lambda \in \mathcal{C}_{N}^{pt}(\Ein)$ is homotopic to a pointed polygon $\Lambda^s = (w_i)_{i=1}^{N}$ whose underlying unpointed polygon is symmetric. However, for any cyclic transformation $g \in \Z_{N}$ changing basepoints by $g\cdot (w_i)_{i=1}^N =(w_{i+j})_{i=1}^N$, one easily sees
$g \cdot \Lambda^s$ and $\Lambda^s$ are homotopic in $\mathcal{C}_{N}^{pt}(\Ein)$ by the description of symmetric polygons in the proof of Proposition \ref{Prop:SymmetricLightlikePolygon}. 
\end{proof}

Corollary \ref{Cor:ComponentCountDownPointed} confirms a conjecture of \cite{TW24} that this space is disconnected for $N \geq 6$.

\subsection{Connected components of polygons in  \texorpdfstring{$\mathbf{Ein}^{1,q}$}{Ein(1,q)}, \texorpdfstring{$q\geq3$}{q>=3} }

In this section, we show that the moduli space $\mathcal{M}_{N}(\mathbf{Ein}^{1,q})$ is connected when $q\geq 3$, in contrast to the case $q=2$. 
Again, we work mostly with the related configuration and moduli spaces $\mathcal{C}_N(\S^q)$ and $\mathcal{M}_N(\S^q)$ of spherical polygons and then deduce the desired result for the semi-positive polygons in the Einstein universe. 

An important notion for the study of spherical polygons in the higher-dimensional case is as follows. 

\begin{definition}
    Let $\Delta$ be a spherical polygon in $\S^{q}$, for $q\geq 2$ and let $v_1$, $v_2$, $v_3$ be three consecutive vertices of $\Delta$. The \textbf{ellipse} of $v_2$ is defined by
    \[\mathcal{E}(v_2):=\{x\in\S^{q}\mid d_{\S^{q}}(v_1,x)+d_{\S^{q}}(x,v_3) = d_{\S^{q}}(v_1,v_2)+d_{\S^{q}}(v_2,v_3)\}\ .\]
\end{definition}
In the sequel, we will continuously move spherical $N$-polygons using basic moves called \textbf{elliptical deformations}, namely moving a vertex along its ellipse. We want to perform this procedure while staying in the configuration space $\mathcal{C}_N(\S^{q})$. To do so, we need to avoid the following two barriers.
\begin{definition}
    Let $\Delta$ be a spherical polygon in $\S^{q}$, for $q\geq 2$, and let $v_1,\ldots,v_5$ be five consecutive vertices of $\Delta$. The \textbf{barriers} $\boldsymbol{B_-(v_3)}$ and $\boldsymbol{B_+(v_3)}$ are the points of $\E(v_3)$ defined as follows. 
    \begin{itemize}
        \item Let $B_-(v_3)$ be the first intersection of the geodesic ray $\gamma_-$ emanating from $v_2$ and aiming at $-v_1$ with the ellipse $\E(v_3)$. 
        \item Let $B_+(v_3)$ be the first intersection of the geodesic ray $\gamma_+$ emanating from $v_4$ and aiming at $-v_5$ with the ellipse $\E(v_3)$. 
    \end{itemize}
    We denote $\boldsymbol{B(v_3)} = \{ B_-(v_3),B_{+}(v_3)\}$ as the set of both barriers of $v_3$. 
\end{definition} 
The utility of elliptical deformations appears in $\S^{q}$ when $q\geq 3$, as in this case the freedom for elliptical deformations of a given vertex $v_i$ lies in $\mathcal{E}(v_i)\setminus B(v_i)$, which is connected, as opposed to the $q=2$ case.

Using elliptical deformations, we can verify the connectivity of the configuration space for $q \geq 3$.

\begin{proposition}
    The configuration space $\mathcal{C}_N(\S^q)$ is connected for $N\geq 4$, $q\geq 3$. 
\end{proposition}

\begin{proof}
    The proof goes as follows. We provide an inductive process to simplify an arbitrary initial $N$-gon $\Delta$.
    When $N$ is even, we continuously deform $\Delta$ 
    into an equilateral $N$-gon, which has all its vertices in a fixed great circle.
    When $N$ is instead odd, we can deform $\Delta$ to become equilateral, and with all of its vertices but one on a fixed great circle. 
    The described `normal form' after deformation is obviously unique (up to $\SO(q+1,\R)$). 
    Thus, once we have verified the possibility of such deformations, the proof is complete. 
    
    Let $\Delta$ be a spherical $N$-gon, with vertices ordered as $(v_1,\ldots,v_N)$, along the cyclic orientation. Denote by $C$ the great circle spanned by $v_1$ and $v_2$. We show by induction that $\Delta$ is homotopic to a polygon with every vertex in $C$ but the last. The algorithm is as follows.\medskip  

    \textbf{Initialization}. By definition, $v_1$ and $v_2$ lie on the great circle $C$. \medskip 
    
    \textbf{Inductive Step}. Take an index $i$ such that 
    $2\leq i \leq N-1$, and suppose that $\Delta$ is already homotopic, in $C_N(\S^q)$, to a polygon $\Delta_{i}$ with every vertex $w_l$ lying in $C$ for $1\leq l\leq i-1$. We would like to move $w_i$ along its ellipse $\mathcal{E}(w_i)$ until it reaches $C$. There are two obstructions: the two barriers in $\mathbf{B}(w_i)$. Note that the ellipse of $w_i$ intersects $C$ at two points $p_-$, $p_+$. Since $q\geq 3$, the path-connected space $\mathcal{E}(w_i)$ is of dimension $\geq 2$. Thus, $\mathcal{E}(w_i)\setminus\{p_-,p_+\}$ is still connected. As a consequence, it is always possible to perform the procedure unless these two points $p_-$ and $p_+$ are exactly equal to the barriers. If the unique obstruction $\mathbf{B}(w_i) = \{p_-,p_+\}$ transpires, then an adequate small perturbation of 
    $w_{i+1}$ will remove the obstacle. Indeed, after moving $w_{i+1}$ on its ellipse $\mathcal{E}(w_{i+1})$, the geodesic passing through $w_{i+2}$ and $w_{i+1}$ no longer crosses $C$. \medskip

    \textbf{Conclude inductive step}. 
    By induction, $\Delta$ is thus homotopic to a new spherical $N$-gon with all its vertices on the great circle $C$ but the last one. For simplicity, we write $v_1, \dots, v_N$ for the updated vertices obtained after the inductive step. \medskip

    \textbf{Finish the deformation}. If $N$ is odd, we have already managed to place all vertices $v_1,\dots,v_{N-1}$ onto a single great circle, and we can do no better.

    If $N$ is even, we now move $v_{N}$ onto $C$. In this case, the barriers $B_-(v_N),B_+(v_N)$ coincide. Indeed, all four vertices $v_1$, $v_2$, $v_{N-1}$, and $v_{N-2}$ lie on $C$ and the directions $\overrightarrow{v_2v_1}$, $\overrightarrow{v_{N-2}v_{N-1}}$ coincide. Thus, in this case $v_{N}$ can also be moved continuously by an elliptical deformation to $C$. \medskip 

    \textbf{Redistribute lengths}. We now transform the deformed polygons to become equilateral. 
    We have no choice at this point but to separately treat the even and odd cases. \medskip
    
    \textbf{Case 1}: Suppose that $N$ is even. Then up to homotopy $\Delta$ is a polygon having all its vertices $v_i$ in $C$. The polygon $\Delta$ is explicitly described by its first vertex $v_1$ and the collection of spherical distances $(l_i=d_{\S^{q}}(v_i,v_{i+1}))_{i=1}^{N}$. Change continuously all the lengths $l_i$ by $l_i(t):=(1-t)l_i+2\pi t/N$. This shows that $\Delta$ can further be deformed into an equilateral $N$-gon with all vertices on $C$. Such a polygon being unique up to $\SO(n+1,\R)$, the space of $N$-gons is connected.\medskip  

    \textbf{Case 2}: Suppose that $N$ is odd. Then up to homotopy $\Delta$ is a polygon all of whose vertices are in $C$, but its last vertex $v_{N}$. In this case, the barriers $\mathbf{B}(v_N)$ of $v_{N}$ are exactly the two points at intersection between $C$ and the ellipse of $v_{N}$. Up to moving $v_N$ along its ellipse, we can suppose that $d_{\S^{q}}(v_1,v_{N})=d_{\S^{q}}(v_{N-1},v_{N})$. Now, deform the polygon $\Delta$ as follows: fix $v_1$, and move $v_2,\ldots,v_{N-1}$ by modifying continuously the distances $(l_i=d_{\S^{q}}(v_i,v_{i+1}))_{i=1}^{N-2}$ by $l_i(t):=(1-t)l_i+ts$, where 
    \[s:=(2\pi - d_{\S^{q}}(v_1,v_N)-d_{\S^{q}}(v_N,v_{N-1}))/(N-2)\ .\]
    Along the process, $v_{N-1}$ may move, but always towards $v_1$. Thus it is always possible to move $v_{N}$ accordingly, without changing the lengths $d_{\S^{q}}(v_1,v_{N})$ or $d_{\S^{q}}(v_{N},v_{N-1})$. We thus moved the vertices $v_2,\ldots,v_{N}$ so that the edges of $\Delta$ have the same length, but maybe not its last two edges. Thus all odd-indexed vertices are equal, but $v_{N}$, and all even-indexed vertices are equal. It is then easy to redistribute length so that $\Delta$ is continuously deformed to an equilateral polygon having all its vertices, but the last, in $C$. Such a polygon is unique, up to $\SO(q+1,\R)$, hence the space of $N$-gons is connected.
\end{proof}\

\begin{corollary}
The moduli spaces $\mathcal{M}_{N}(\widehat{\mathbf{Ein}}{}^{1,q})$ and $\mathcal{M}_{N}(\mathbf{Ein}^{1,q})$ are connected. 
\end{corollary}

\begin{proof}
Note that $\mathcal{C}_N(\widehat{\mathbf{Ein}}{}^{1,q})$ is a circle bundle over $\mathcal{C}_N(\S^{q})$. Thus, the former configuration space is connected. Then we quotient by the connected group $\SO_0(2,q+1)$ to conclude that $\mathcal{M}_N(\widehat{\mathbf{Ein}}{}^{1,q})$ is connected. Finally, we quotient by the antipodal map to find that $\mathcal{M}_N({\mathbf{Ein}}^{1,q})$ is connected. 
\end{proof}

\section{Maps of moduli: polynomials to polygons}\label{Sec:ModuliSpaces}
In this section, we prove the main theorem. We begin with some notation, which we now reintroduce from the introduction. We shall mostly focus on the results from Theorem \ref{thm:A} for maximal surfaces in $\H$. As the companion results in the double cover are analogous, some may be omitted.

\begin{definition}
Let $n\geq 0$ be an integer. Denote by $\mathcal{C}_n(\mathcal{PD})$ and $\mathcal{C}_n(\mathcal{PD}^o)$ be the configuration space of unordered (resp. ordered) pairs $\{A(z)\dz^2,B(z)\dz^2\}$ of holomorphic polynomial quadratic differentials whose degrees sum to $n$. Let $\mathcal{M}_n(\mathcal{PD}^o)$ be the quotient of $\mathcal{C}_n(\mathcal{PD}^o)$ under the action of $\Aff(\C)\times \C^*$, where $\Aff(\C)$ acts by simultaneous pushforward on both pairs and $\lambda \in \C^*$ acts by $\lambda \cdot (A\dz^2,B\dz^2) = (\lambda A\dz^2, \lambda^{-1}B\dz^2)$.  
Then $\mathcal{M}_n(\mathcal{PD})=\mathcal{M}_n(\mathcal{PD}^o)/\Z_2$ is the moduli space of unordered pairs of such differentials.  
\end{definition}

We will also need a few different spaces of maximal surfaces, both pointed and unpointed.
In particular: 
\begin{itemize}
    \item Recall that $\mathcal{M}_{n+4}^{pt}(\H)$ is the \emph{moduli space} of pointed maximal surfaces with asymptotic boundary an $(n+4)$-gon in $\Ein$, up to the action of $\SO_0(2,3)$. 
    \item We now let $\Max$ denote the \emph{configuration space} of unpointed maximal surfaces in $\H$, and let $\mathcal{C}_{\geq0}(\Ein)$ denote the \emph{configuration space} of all semi-positive loops in $\Ein$. 
\end{itemize}
Recall from \textsection\ref{subsec:HiggstoMaximal} and \textsection\ref{Subsec:MaximaltoHiggs} that we have a correspondence between conformally parametrized maximal surfaces $\sigma: \C \rightarrow \Hh$, up to post-composition by $\SO_0(2,3)$, and Higgs bundles of the form \eqref{SO(2,3)Higgs}. 

In this section, we shall wish to convert the bundle-valued holomorphic differentials $\alpha$, $\beta$ into holomorphic polynomial quadratic differentials as in Definition \ref{Defn:AssociatedDifferentials}. This change necessitates that we reconsider the Higgs bundle \eqref{SO(2,3)Higgs} with a fixed background dual holomorphic trivialization $(\eta,\eta^{-1})$ for $\mathcal{N}\oplus \mathcal{N}^{-1}$, to form associated holomorphic differentials $(\alpha_{QD},\beta_{QD})=(A\dz^2, B\dz^2)$. We then have $\alpha =\alpha_{QD}\otimes \eta$, $\beta= \beta_{QD}\otimes \eta^{-1}$. The frame $(\eta,\eta^{-1})$ can be chosen such that $A$, $B$ are polynomial, and is well-defined up to the $\C^*$-action by $\lambda \cdot (\eta, \eta^{-1})=(\lambda \eta, \lambda^{-1} \eta^{-1})$. 
Finally, we re-gauge the Higgs bundle to obtain the following form:

\begin{equation}\label{SO(2,3)TrivializedNormal}
    \begin{tikzcd}
\mathcal{O} \arrow[r, "\alpha_{QD}"] & \mathcal{K}\arrow[r, "2"]&\mathcal{O}\arrow[r, "2"]&\mathcal{K}^{-1} \arrow[r, "\alpha_{QD}"]\arrow[bend left=20, lll, "\beta_{QD}"]&\mathcal{O}\arrow[bend left=20, lll, "\beta_{QD}"]\ .\\
    \end{tikzcd} 
\end{equation}

Our work in Section \ref{Subsec:PolynomialSurfaces} allows us to define the following map.

\begin{definition}[Polynomial maximal surfaces]\label{Defn:MS}
Define a map $\mathbf{\Sigma}^{pt}: \mathcal{M}_n(\mathcal{PD}) \rightarrow \mathcal{M}_{n+4}^{pt}(\H)$ as follows. Given $\delta =[\{\alpha_{QD},\beta_{QD}\}] \in \mathcal{M}_n(\mathcal{PD})$, choose
any lift $(\alpha_{QD},\beta_{QD})$ to $\mathcal{C}_n(\PD^{o})$,
and then set \\
$\mathbf{{\Sigma}}^{pt}(\delta) := [ \, \big(\Sigma(\alpha_{QD},\beta_{QD}), \sigma(0)\big) \,]$, where
\begin{itemize}
    \item $(\E,\varphi(\alpha_{QD},\beta_{QD}))$ is the $\SO_0(2,3)$-Higgs bundle \eqref{SO(2,3)TrivializedNormal} on $\C$,
    \item $h$ is the diagonal, real, strongly complete harmonic metric on $(\E,\varphi(\alpha_{QD},\beta_{QD}))$ from Lemma \ref{Lem:TrueComplete},
    \item $\sigma: \C \rightarrow \H$ is a conformal maximal embedding associated to $(\E,\varphi(\alpha_{QD},\beta_{QD}),h)$ by Proposition \ref{Prop:HiggsToMaximalSurface}, well-defined up to $\SO_0(2,3)$, and complete by the strong completeness of $h$,
    \item $\Sigma(\alpha_{QD},\beta_{QD})$ is the image of $\sigma$.
\end{itemize} 
By Lemma \ref{Lem:PolynomialPolygon}, $\mathbf{{\Sigma}}^{pt}$ is well-defined. We define maps $\boldsymbol{\Sigma}, \boldsymbol{\Sigma}^o$ to $\mathcal{M}_n(\H)$, $\mathcal{M}_n(\Hh)$ analogously. 
\end{definition}
In other words, this procedure engineers a maximal surface with prescribed polynomial quadratic differentials, as verified in Lemma \ref{Lem:DeltaInjective}.

We now define the related map from polynomial differentials to  semi-positive polygons. 
\begin{definition}
Let $\Xi: \mathcal{M}_n(\mathcal{PD}) \rightarrow \mathcal{M}_{n+4}(\Ein)$ be defined as follows: 
\[ \Xi(\,[\{\alpha_{QD},\beta_{QD}\}]\,)= [\partial_{\infty}\Sigma(\alpha_{QD},\beta_{QD})]\ . \]
\end{definition}

Recall (Definition \ref{Defn:AssociatedDifferentials}) that $\boldsymbol{\delta}:\mathcal{M}_{n+4}(\H)\to\mathcal{M}_n(\mathcal{PD})$ associates to an $\SO_0(2,3)$-orbit of maximal surfaces the associated polynomial holomorphic data, defined up to $\Aff(\C)\times\C^*$. We observe the first relation between $\boldsymbol{\delta}$ and $\boldsymbol{\Sigma}$.  

\begin{lemma}\label{Lem:DeltaInjective}
We have $\boldsymbol{\delta}^{o}\circ \boldsymbol{\Sigma}^{o} = \id$ and $\boldsymbol{\delta}\circ \boldsymbol{\Sigma} = \id$. 
\end{lemma}

\begin{proof}
We need only prove the former. \medskip 

Take any $[(\alpha_{QD},\beta_{QD})]=\delta \in \mathcal{M}_n(\mathcal{PD}^o)$. 
Then we build the surface $\Sigma = \boldsymbol{\Sigma}^o(\delta)$. By the `moreover' statement in Proposition \ref{Prop:HiggsSymmetricPolygon}, the associated surface $\Sigma$ satisfies $\boldsymbol{\delta}^{o}\circ \boldsymbol{\Sigma}^{o}(\delta) = \boldsymbol{\delta}^{o}(\Sigma)=\delta$.  
\end{proof}

Recall that the topology on $\Max$ is that of smooth convergence as graphs over compacts sets, in any splitting (see \cite{LTW24}). Moreover, the topology of the space of semi-positive loops is that of pointwise convergence as graphs. Given a semi-positive loop $\Lambda \subset \Ein$, any map $\partial \varphi: \S^1 \rightarrow \S^2$ whose graph is a lift of $\Lambda$ to $\Einh$ is 1-Lipschitz by Proposition \ref{Prop:SemiPositive}. Due to this  regularity, pointwise convergence of semi-positive loops is equivalent to uniform convergence. 

As an easy application of the work of \cite{LTW24}, the asymptotic boundary map is continuous. 
\begin{proposition}\label{Prop:AsymptoticBoundaryCnty}
    The map $\partial_\infty : \Max\to \mathcal{C}_{\geq 0}(\Ein)$ that associates to a maximal surface its asymptotic boundary is continuous.
\end{proposition}

\begin{proof}

For this proof, fix a warped product decomposition $\Hh=\S^2_+\times\S^2$ and the corresponding decomposition $\Einh=\S^1\times\S^2$.

    Let $(\Sigma_n)_{n\in\mathbb{N}}$ be a sequence of  maximal surfaces converging to the maximal surface $\Sigma$. Let $\Lambda_n:=\partial_\infty\Sigma_n$ and $\Lambda:=\partial_\infty\Sigma$. We want to show that $\lim_{n\to+\infty}\Lambda_n=\Lambda$.

    First, describe every surface and loop as graphs as follows:
    \[\Sigma_n = \mathrm{graph}(\varphi_n:\S^2_+\to\S^2),\ \Sigma = \mathrm{graph}(\varphi:\S^2_+\to\S^2)\ ,\]
    where $\varphi_n$ and $\varphi$ are strictly 1-Lipschitz maps, as in Proposition \ref{prop:CompleteGraph}. Since $\Lambda_n$ and $\Lambda$ are the asymptotic boundaries of $\Sigma_n$ and $\Sigma$, respectively, then
    \[\Lambda_n=\mathrm{graph}(\partial\varphi_n),\ \mbox{and}\ \Lambda=\mathrm{graph}(\partial\varphi)\ ,\]
    where $\partial f $ denotes the continuous extension of a 1-Lipschitz map $f:\S^2_+\to\S^2$ to $\S^1=\partial\S^2_+$.

    By hypothesis, $(\varphi_n)_n$ converges to $\varphi$ smoothly over compact sets. In particular, $(\varphi_n)_n$ converges pointwise to $\varphi$. Since $\S^2_+\cup\S^1$ and $\S^2$ are compact, and $\partial\varphi_n$ and $\partial\varphi$ are continuous 
    extensions of $\varphi_n$ and $\varphi$, respectively, $(\partial\varphi_n)_n$ converges pointwise to $\partial\varphi$. 
    That is, $(\Lambda_n)_{n\in\mathbb{N}}$ converges to $\Lambda$. 
\end{proof}

\begin{corollary}\label{Cor:ContinuityAsymptoticBoundary}
The map $\partial_\infty : \mathcal{M}_{n+4}(\H) \rightarrow \mathcal{M}_{n+4}(\Ein)$ is continuous. 
\end{corollary}

Next, we show the map constructing polynomials differentials from maximal surfaces is continuous.  
\begin{lemma}
The map $\boldsymbol{\delta}^{o}: \mathcal{M}_{n+4}(\Hh) \rightarrow \mathcal{M}_n(\mathcal{PD}^o)$ is continuous.
\end{lemma}

\begin{proof}
Suppose that the sequence $([\Sigma_j])_{j}$ converges to $[\Sigma_{\infty}]$. We choose representatives 
$\Sigma_j$ such that $\Sigma_j \rightarrow \Sigma_{\infty}$ smoothly over compact sets.
Because each $\Sigma_j$ and $\Sigma_\infty$ are conformal to $\C$, there are conformal parametrizations 
$\sigma_j: \C \rightarrow \Sigma_j$ and $\sigma_{\infty}:\C \rightarrow \Sigma_{\infty}$ such that $\sigma_j \rightarrow_{C^{\infty}_c}\sigma_{\infty}$.
Hence, we conclude that the associated quartic differentials $Q_{4,j}$ converge smoothly over compacta to $Q_{4,\infty}$. 
We can take non-vanishing sections $\eta_j \in H^0(\mathcal{N}_{\sigma_j})$ and write $\alpha_j = A_j\dz^2\otimes \eta_j$, $\beta_j = B_j\dz^2\otimes \eta_j^{-1}$, where $A_j, B_j$ are polynomials, and do the same for $\alpha_{\infty}$, $\beta_{\infty}$. 
Then up to adjusting the sections $\eta_j$, we can guarantee that $A_j, B_j$ converge smoothly over compact sets to $A_{\infty}, B_{\infty}$. In other words, 
$\boldsymbol{\delta}([\Sigma_j])  = [(A_j\dz^2, B_j\dz^2)]\rightarrow [(A_{\infty}\dz^2, B_{\infty}\dz^2)]= \boldsymbol{\delta}([\Sigma_{\infty}])$. 
\end{proof}

\begin{corollary}\label{CorollaryDeltaContinuous}
    The map $\boldsymbol{\delta}:\mathcal{M}_{n+4}(\H)\rightarrow \mathcal{M}_n(\mathcal{PD})$ is continuous.
\end{corollary}
\begin{proof}
    It suffices to take coherent lifts to the double cover.
\end{proof}

We can now verify that $\boldsymbol{\delta}$ is the true inverse of $\boldsymbol{\Sigma}$. 

\begin{lemma}\label{LemmaSigmaDeltaInverse}
We have $\boldsymbol{{\Sigma}}^{o} \circ \boldsymbol{\delta}^o =\id$ and $\boldsymbol{{\Sigma}}\circ \boldsymbol{{\delta}}= \id$. 
\end{lemma}

\begin{proof}
We prove only the result for $\boldsymbol{\delta}^{o}$, as the latter follows immediately from the former. 

Now, take $[\Sigma] \in \mathcal{M}_{n+4}(\Hh)$. We want to show that $\boldsymbol{\Sigma}^{o} \circ \boldsymbol{\delta}^{o}( [\Sigma] )=[\Sigma]$. We can understand the composition as follows.

Let us first take any representative $\Sigma$ for the class and a corresponding conformal parametrization $\sigma: \C \rightarrow \Sigma$. Choose a holomorphic trivialization
$\eta \in H^0(\mathcal{N}_{\sigma})$, yielding the trivialization $(\eta, \eta^{-1})$ for $\mathcal{N}\oplus \mathcal{N}^{-1}$, such that the associated quadratic differentials are polynomial. Relative to this parametrization, 
let the polynomial quadratic differentials of $\sigma$ be $(A \dz^2, B \dz^2)$, which are well-defined up to the $\C^*$-action. 
Then from $\sigma$, and this data, we obtain a harmonic bundle 
$(\E, \varphi(A\dz^2,B\dz^2),h)$, where $h$ is diagonal, real, and as described in \textsection\ref{Subsec:MaximaltoHiggs}.
Furthermore, $h$ is strongly complete in the sense of Definition \ref{Defn:StrongCompleteness} by Corollary \ref{cor:MutualBoundedness}. 
On the other hand, by Lemma \ref{Lem:TrueComplete}, 
the Higgs bundle $(\E, \varphi(A\dz^2,B\dz^2))$ 
admits a unique harmonic metric that is diagonal, strongly complete, and real.
Hence, we must have that 
\[ [\Sigma] =  \boldsymbol{\Sigma}^{o}(A\dz^2, B\dz^2)\ =\boldsymbol{\Sigma}^{o}\circ \boldsymbol{\delta}^{o}(\Sigma) .\qedhere\] 
\end{proof}

To deduce the continuity of $\Xi$, we shall use the following lemma.

\begin{lemma}\label{Lem:ContinuityPolynomialPolygon}
The map
$\mathbf{\Sigma}^{pt}:\mathcal{M}_n(\mathcal{PD}^o)\rightarrow \mathcal{M}_{n+4}^{pt}(\Hh)$
is continuous.
\end{lemma}

\begin{proof}
Suppose $[(\alpha_j,\beta_j)]_{j\in\mathbb{N}}$ is a sequence of elements in $\mathcal{M}_n(\mathcal{PD}^o)$ converging to $[(\alpha,\beta)]$.
Fix a pointed hyperbolic plane $x_0\oplus \T_0\oplus \N_0$. That is, $x_0 \in \H$, $\T_0$, $\N_0$ are pairwise orthogonal, where the latter two are spacelike and timelike 2-planes, respectively. 
We can choose normalized representatives $(\Sigma_j,x_0)$ of the class of pointed surfaces 
$\mathbf{{\Sigma}}^{pt}(\alpha_j, \beta_j)$ such that $\mathrm{T}_{x_0}\Sigma_j= \T_0$. 
The $\SO_0(2,3)$-freedom remaining lies in the maximal compact torus $\SO(2)\times \SO(2)$, which merely rotates $\T_0$ and $\N_0$. 
Here, we proceed with any such pointed surface $(\Sigma_j,x_0)$ normalized as described. 

We then want to show that the sequence $(\mathbf{{\Sigma}}^{pt}(\alpha_j,\beta_j))_{j\in\mathbb{N}}  \subset \mathcal{M}_{n+4}^{pt}(\Hh)$ converges by using these partially normalized representatives $(\Sigma_j,x_0)$. By the subsubsequence lemma, it suffices to prove that any subsequence of $((\Sigma_j,x_0))_{j\in\mathbb{N}}$ has a subsequence that converges to a pointed surface in the class ${\mathbf{\Sigma}}^{pt}(\alpha,\beta)$. Thus, fix any subsequence $\big((\Sigma_{j_k},x_0)\big)_{k\in\mathbb{N}}$ and we undertake this task. 

By \cite[Theorem 6.1(i)]{LTW24}, the sequence of pointed maximal surfaces $\big((\Sigma_{j_k}, x_0)\big)_{k\in\mathbb{N}}$ subconverges, and hence we can extract a limiting surface $(\Sigma_{\infty},x_0)$. 
Let $((\Sigma_{{j_k}_{\ell}},x_0))_{\ell\in\mathbb{N}}$ be the convergent subsequence. Now, for the original sequence, we know \emph{a priori} that $\partial_{\infty} \Sigma_{j} = \Lambda_{j}$ is a sequence of semi-positive polygons in $\Ein$, with $n+4$ vertices. By Proposition \ref{Prop:AsymptoticBoundaryCnty}, we conclude that $\Lambda = \partial_{\infty}(\Sigma_{\infty})$ is also a polygon, with at most $n+4$ vertices. In particular, the limiting surface $\Sigma_{\infty}$ is biholomorphic to the complex plane $\C$ by Corollary \ref{Cor:AllPolynomialMaximalSurfaces}.

Now, by the convergence of the subsubsequence and since every surface is conformal to $\C$, we choose conformal parametrizations 
$\sigma_{j_{k_{\ell}}}: \C \rightarrow \Sigma_{{j_{k_{\ell}}}}$ that converge smoothly on compact sets to a conformal parametrization $\sigma_{\infty}: \C \rightarrow \Sigma_{\infty}$. 
Then $Q_{4,j_{k_{\ell}}}\rightarrow Q_{4,\infty}$ smoothly over compact sets. 
Hence, $\deg(Q_{4,\infty})=n$ and the polygon $\Lambda_\infty$ has $n+4$ vertices. 
Now, again by choosing appropriate trivializations of the corresponding holomorphic normal line bundles, write $\alpha_{j_{k_{\ell}}}= A_{j_{k_{\ell}}}\dz^2\otimes \eta_{j_{k_{\ell}}}$, $\beta_{j_{k_{\ell}}}= B_{j_{k_{\ell}}}\dz^2\otimes \eta_{j_{k_{\ell}}}^{-1}$ 
and similarly for $\alpha_{\infty}$, $\beta_{\infty}$, with $A_{j_{k_{\ell}}}$, $B_{j_{k_{\ell}}}$, $A_\infty$ and $B_\infty$ being polynomials. 
Then $A_{j_{k_{\ell}}} \rightarrow A_{\infty}$ and $B_{j_{k_{\ell}}} \rightarrow B_{\infty}$ smoothly over compact sets, up to adjusting the trivializations $\eta_{j_{k_{\ell}}}$. Thus, by the original convergence 
of $\alpha_j\rightarrow \alpha$ and $\beta_{j} \rightarrow \beta$, we conclude $\beta=\beta_{\infty}$ and $\alpha =\alpha_{\infty}$. 
In other words, we have found $\boldsymbol{\delta}^o(\Sigma_{\infty})=[(\alpha,\beta)]$. 
By the injectivity of $\boldsymbol{\delta}$, 
we conclude that $\Sigma_{\infty} = \boldsymbol{\Sigma}(\alpha,\beta)$. This completes the proof. 
\end{proof}

We now conclude Theorem \ref{thm:A}. 
\begin{theorem}\label{thm:Polynomial-Maximal}
    The map $\boldsymbol{\Sigma}:\mathcal{M}_n(\mathcal{PD})\rightarrow\mathcal{M}_{n+4}(\H)$ is a homeomorphism.
\end{theorem}
\begin{proof}
    By Lemma \ref{Lem:DeltaInjective} and Lemma \ref{LemmaSigmaDeltaInverse}, the maps $\boldsymbol{\Sigma}$ and $\boldsymbol{\delta}$ are mutual inverses. By Lemma \ref{Lem:ContinuityPolynomialPolygon}, the map $\boldsymbol{\Sigma}$ is continuous and by Corollary \ref{CorollaryDeltaContinuous}, the map $\boldsymbol{\delta}$ is continuous.
\end{proof}

Let us note the properties of the asymptotic boundary map $\partial_{\infty}$. 

\begin{proposition}\label{Prop:AsymptoticBoundaryHomeo}
The map $\partial_{\infty}: \mathcal{M}_{n+4}(\H)\rightarrow \mathcal{M}_{n+4}(\Ein)$ is a homeomorphism. 
\end{proposition}

\begin{proof}
The map is continuous by Proposition \ref{Prop:AsymptoticBoundaryCnty}, bijective by \cite[Theorem A]{LTW24}, and has continuous inverse by \cite[Theorem 6.1(ii)]{LTW24}. 
\end{proof}

We finally deduce Theorem \ref{thm:B}.

\begin{theorem}[Polynomial $\leftrightarrow$ Polygon]\label{thm:PolynomialPolygon} $\Xi: \mathcal{M}_n({\mathcal{PD}})\rightarrow \mathcal{M}_{n+4}(\Ein)$ is a homeomorphism. 
\end{theorem}

\begin{proof}
The result follows from Theorem \ref{thm:Polynomial-Maximal} and Proposition \ref{Prop:AsymptoticBoundaryHomeo}.
\end{proof}

Given the classification of components in $\mathcal{M}_{n+4}(\H)$ via the turning invariant, we obtain the following relation between $\Theta$ and the polynomial differentials of a maximal surface.

\begin{corollary}\label{Cor:ThetaDegreeAB}
Let $\delta =[(A\dz^2,B\dz^2)] \in \mathcal{M}_n(\mathcal{PD}^o)$ and $\Sigma \subset \Hh$ be a maximal surface representing the class $\mathbf{\Sigma}^o(\delta)$ from Theorem \ref{thm:A} with asymptotic boundary $\Lambda \subset \Einh$. Then \[\Theta(\Lambda) =\deg(A)+2=n-\deg(B)+2\ .\]
\end{corollary}

\begin{proof}
This follows immediately from Corollary \ref{cor:SymmetricMonomials}, Corollary \ref{Cor:ConnectedComponentCountUp}, and the fact that the asymptotic boundary map 
$\partial_{\infty}: \mathcal{M}_{n+4}(\Hh) \rightarrow \mathcal{M}_{n+4}(\Einh)$ is a homeomorphism by Proposition \ref{Prop:AsymptoticBoundaryHomeo}. 
\end{proof}

There are two further corollaries of interest to record regarding special maximal surfaces. 

\begin{corollary}[Hitchin maximal surfaces]
Let $\Sigma \subset \H$ be a polygonal maximal surface with asymptotic boundary $\Lambda =\partial_{\infty}\Sigma$. Then $\Sigma$ lies in the Hitchin component if and only if the turning invariant
of either lift of $\Lambda$ to $\Einh$ is extremal. 
\end{corollary}

\begin{proof}
By definition of the Hitchin case, $\Sigma$ is maximal exactly when one of its lifts $\hat{\Sigma}$ admits a parametrization with $\alpha$ non-vanishing. Thus, the asymptotic boundary $\hat{\Lambda} = \partial_{\infty}(\hat{\Sigma})$ satisfies $\Theta(\hat{\Lambda})=2$ by Corollary \ref{Cor:ThetaDegreeAB}. 
\end{proof}

We now consider the `reducible' maximal surfaces in $\H$ that are trapped in a copy of $\mathbf{H}^{2,1}$. 

\begin{corollary}[Reducible maximal surfaces]\label{cor:AdS3}
Let $\Sigma \subset \H$ be a polygonal maximal surface with asymptotic boundary $\Lambda = \partial_{\infty}\Sigma$ a $2k$-gon in $\Ein$. Then the following are equivalent: 
\begin{enumerate}[label=(\roman*)]
    \item $\partial_{\infty}\Sigma$ is contained in a copy of $\mathbf{Ein}^{1,1}$, 
    \item For any conformal parametrization $\sigma:\C \rightarrow \Sigma$, we have $\mathrm{div}(\alpha)=\mathrm{div}(\beta)$. 
\end{enumerate}
\end{corollary}
 
\begin{proof}
(i)$\rightarrow$(ii). Let $\sigma$ be given and $\Lambda = \partial_{\infty}\Sigma$. Let $\hat{\Lambda}$ be a lift of $\partial_\infty\Sigma$ to $\Einh$. In any splitting $\Einh \simeq \S^1\times\S^2$, the projection $\Delta$ of $\hat\Lambda$ is contained in a great circle $C$. Then the isometry $g \in \mathsf{O}(2,3)$ defined as the identity on $\S^1$ and the reflection along the great circle $C$ on $\S^2$ fixes the semi-positive polygon $\hat{\Lambda}$. As a consequence, $g$ preserves $\Sigma$ by the uniqueness result of \cite{LTW24}. 
However, since $g$ reverses the time-orientation on $\R^{2,3}$, post-composition by $g$ will exchange the labels $\alpha$, $\beta$ for the parametrization $\sigma$. Hence, $\div(\alpha)= \div(\beta)$.

(ii)$\rightarrow$(i). In this case, one can choose dual holomorphic trivializations $\eta$, $\eta^{-1}$ of $\mathcal{N}$, $\mathcal{N}^{-1}$ such that $\alpha =A\dz^2\otimes \eta$
 and $\beta=B\dz^2\otimes \eta^{-1}$, where $A=B$. In this framing,
the Higgs bundle \eqref{SO(2,3)TrivializedNormal} associated to $\sigma$ reduces to an $\SO_0(2,2)$-Higgs bundle. Consequently, the resulting maximal surface $\Sigma= \Sigma(A\dz^2,B\dz^2)$, realized as in the description of Lemma \ref{Lem:PolynomialPolygon}, is contained in a copy of $\mathbf{H}^{2,1}$. 
\end{proof}

\bibliographystyle{alpha}
\bibliography{bib}

@phdthesis{Bar10,
	author = {Baraglia, David},
	school = {University of Oxford},
	title = {{$G_2$} {Geometry} and {Integrable Systems}},
	year = {2010}}

@article{DL24,
    author = {Dai, Song and Li, Qiongling}, 
    title = {Bounded differentials on the unit disk and the associated geometry},
    journal = {Trans. Amer. Math. Soc.},
    year = {2024},
    pages = {5445–5481},
    volume = {377},
    number = {8}
}

@article{LM19,
    author = {Loftin, John and McIntosh, Ian},
    title = {The moduli space of equivariant minimal surfaces in $\mathbb{RH}^3$ and $\mathbb{RH}^4$ via {Higgs} bundles},
    journal = {Geom. Ded.},
    year = {2019},
    volume = {201}, 
    pages = {325-351}
}

@book {Hei05,
    AUTHOR = {Heinonen, Juha},
     TITLE = {Lectures on {L}ipschitz analysis},
    SERIES = {Report. University of Jyv\"askyl\"a{} Department of
              Mathematics and Statistics},
    VOLUME = {100},
 PUBLISHER = {University of Jyv\"askyl\"a, Jyv\"askyl\"a},
      YEAR = {2005},
     PAGES = {ii+77},
      ISBN = {951-39-2318-5},
   MRCLASS = {49Q15 (26B35 30C35 46E35 49J52 52A01 54E40)},
  MRNUMBER = {2177410},
MRREVIEWER = {Jan\ Mal\'y},
}

@article{CT24,
	author = {Collier, Brian and Toulisse, J{\'e}r{\'e}my},
	journal = {Trans. Amer. Math. Soc.},
	title = {Holomorphic curves in the 6-pseudosphere and cyclic surfaces},
	year = {2024},
        pages = {6465-6514},
        issue = {377}
}

@article{KMT00,
    author = {Kapovich, Michael and Millson, John and Treloar, Thomas},
    title = {The symplectic geometry of polygons in hyperbolic space},
    journal = {Asian J. Math.},
    year = {2000},
    pages = {123-164}
}

@article{KM95,
    author = {Kapovich, Michael and Millson, John},
    title = {On the moduli space of polygons in the {Euclidean} plane},
    journal = {J. Differential Geom.},
    year = {1995},
    pages = {133-164},
    volume = {42}
}

@article{CDR03,
    author = {Connelly,Robert and Demaine, Erik D. and Rote, G\"unter},
    title = {{Straightening Polygonal Arcs and Convexifying Polygonal Cycles}},
    journal = {Disc. \& Comp. Geom. (DCG)},
    year = {2003},
    volume = {30},
    pages = {205-239}
}

@article{Col19,
	author = {Collier, Brian},
	journal = {SIGMA. Symmetry, Integrability and Geometry: Methods and Applications},
	title = {Studying {Deformations} of {Fuchsian} {Representations} with {Higgs} {Bundles}},
	volume = {15},
	year = {2019},}

@article{FN21,
	author = {Fredrickson, Laura and Neitzke, Andrew},
	journal = {Math. Proc. Cambridge Philos. Soc.},
	number = {3},
	pages = {623--656},
	title = {Moduli of wild {Higgs} bundles on $\mathbb{CP}^1$ with $\mathbb{C}^*$ actions},
	volume = {171},
	year = {2021}}

@misc{MT25,
      title={On the scalar curvature of complete maximal spacelike submanifolds in pseudo-hyperbolic spaces}, 
      author={Alex Moriani and Enrico Trebeschi},
      year={2025},
      eprint={2411.10352},
      archivePrefix={arXiv},
      primaryClass={math.DG},
      note={arXiv: 2411.10352}
}

@article{BB04,
    author = {Biquard, O. and Boalch, P.},
    title = {Wild non-abelian Hodge theory on curves},
    journal ={Compos. Math.},
    year = {2004}, 
    volume = {140},
    isse = {1},
    pages = {179-204}
}

@phdthesis{Col16,
    author = {Collier, Brian},
    title = {Finite order automorphisms of {Higgs} bundles: theory and applications},
    school = {University of Illinois at Urbain-Champaign},
    year = {2016}
}

@article{LM25,
    author = {Li, Qiongling and Mochizuki, Takuro},
     year = {2025}, 
    title = {Complete solutions of {Toda} equations and cyclic {Higgs} bundles over non-compact surfaces}, 
    journal = {Int. Math. Res. Not.}
}

@article{LM20b,
    AUTHOR = {Li, Qiongling and Mochizuki, Takuro},
     TITLE = {Isolated Singularities of {Toda} Equations and cyclic {Higgs} Bundles},
    VOLUME = {},
    journal = {arXiv:2010.06129},
    URL = {https://arxiv.org/abs/2010.06129},
      YEAR = {2020},
}

@article{Cheng,
  title = {Space-like Surfaces in an Anti-de {{Sitter}} Space},
  author = {Cheng, Qing Ming},
  year = 1994,
  journal = {Colloquium Mathematicum},
  volume = {66},
  number = {2},
  pages = {201--208},
  issn = {0010-1354,1730-6302},
  doi = {10.4064/cm-66-2-201-208},
  mrnumber = {1268064}
}

@book{PetersenRiemannianGeometry,
  title = {Riemannian {{Geometry}}},
  author = {Petersen, Peter},
  year = 2016,
  series = {Graduate {{Texts}} in {{Mathematics}}},
  volume = {171},
  publisher = {Springer International Publishing},
  address = {Cham},
  doi = {10.1007/978-3-319-26654-1},
  urldate = {2022-10-25},
  isbn = {978-3-319-26652-7 978-3-319-26654-1},
  langid = {english}
}

@article{KM99, 
title={On the Moduli Space of a Spherical Polygonal Linkage}, 
volume={42},
number={3},
journal={Canadian Mathematical Bulletin}, 
author={Kapovich, Michael and Millson, John J.},
year={1999}, 
pages={307–320}
}

@phdthesis{Gal,
    author = {Galitzer, A.J.},
    title = {On the moduli space of closed polygonal linkages on the {2-sphere}},
    school = {University of Maryland, College Park},
    year = {1997},
}

@misc{Ma26,
      title={Holonomy Asymptotics along Quartic Differential Rays}, 
      author={Weihan Ma},
      year={2026},
      eprint={2608.04729},
      archivePrefix={arXiv},
      primaryClass={math.DG},
      note={arXiv: 2608.04729}, 
}

@misc{Lem26,
      title={Volume of maximal representations into $\mathrm{SO}_0(2,3)$}, 
      author={Timothé Lemistre},
      year={2026},
      eprint={2605.04272},
      archivePrefix={arXiv},
      primaryClass={math.DG},
      note={arXiv:2605.04272}, 
}

@article{Mor25,
  title = {Polygonal Surfaces in Pseudo-Hyperbolic Spaces},
  author = {Moriani, Alex},
  year = 2025,
  journal = {Adv. Math.},
  volume = {480},
  pages = { 110484, 79},
  issn = {0001-8708,1090-2082},
  doi = {10.1016/j.aim.2025.110484},
}

@article{Li19,
	author = {Li, Qiongling},
	journal = {SIGMA. Symmetry, Integrability and Geometry: Methods and Applications},
	month = may,
	pages = {035},
	title = {An {Introduction} to {Higgs} {Bundles} via {Harmonic} {Maps}},
	volume = {15},
	year = {2019}}

@article{LTW24, 
author = {Labourie, Francois and Toulisse, Jeremy and Wolf, Michael}, 
title = {Plateau problems for maximal spacelike surfaces in pseudo-hyperbolic space},
Year = {2024}, 
Volume = {57},
Pages = {473-552},
journal = {Ann. Sci. \'Ec. Norm. Sup\'er.}}

@article{TW24, 
title = {Planar minimal surfaces with polynomial growth in the {$\mathrm{Sp}_4(\R)$}-symmetric space},
author = {Tamburelli, Andrea and Wolf, Michael}, 
journal = {Amer. J. Math.},
publisher = {Johns Hopkins University Press},
volume ={146}, 
number = {4},
year = {2024}, 
pages = {871-944}}

@misc{Eva24,
      title={Polynomial Almost-Complex Curves in $\hat{S}^{2,4}$}, 
      author={Parker Evans},
      year={2024},
      eprint={2208.14409},
      archivePrefix={arXiv},
      primaryClass={math.DG},
      note={arXiv:2208.14409}, 
}

@article{DW15,
 author = {Dumas, David and Wolf, Michael}, 
year = {2015},
title = {Polynomial cubic differentials and convex polygons in the projective plane},
journal = {Geom. Funct. Anal.},
pages ={1734-1798},
volume = {25}}

@article{LT23, 
    title = {Quasicircles and quasiperiodic surfaces in pseudo-hyperbolic spaces}, 
    author = {Labourie, Fran\c{c}ois and Toulisse, J\'er\'emy},
    year = {2023},
    pages = {81 - 168}, 
    journal = {Invent. Math.}, 
    volume = {233}, 
    issue = {1},
}

@article{HTTW95,
    title = {Harmonic maps from the complex plane into surfaces with nonpositive curvature}, 
    author = {Han, Zheng-Chao and Tam, Luen-Fai and Treibergs, Andrejs and  Wan, Tom},
    page = {85-114},
    journal = {Comm. Anal. Geom.}, 
    volume = {3}, 
    number = {1},
    year = {1995}, }

@article{ GS25, 
 author = {Gupta, Subhojoy and Sau, Gobinda},
 year = {2025}, 
 title = {On Harmonic Maps from the Complex Plane to Hyperbolic 3-Space},
 journal = {J. Geom. Anal.},
 volume = {35}
}

@article{CTT19,
	author = {Brian Collier and Nicolas Tholozan and J{\'e}r{\'e}my Toulisse},
	journal = {Duke Math. J.},
	number = {15},
	pages = {2873 -- 2949},
	title = {{The geometry of maximal representations of surface groups into $\mathrm{SO}_{0}(2,n)$}},
	volume = {168},
	year = {2019}}

@article{Got01,
title = {Components of spaces of representations and stable triples},
journal = {Topology},
volume = {40},
number = {4},
pages = {823-850},
year = {2001},
issn = {0040-9383},
author = {Peter B. Gothen},
}

@misc{LLX26,
      title={Harmonic Maps from Punctured Riemann Surfaces to the Hyperbolic Plane with Prescribed {Scherk} Asymptotics}, 
      author={Qiongling Li and Jinsong Liu and Yihui Xu},
      year={2026},
      eprint={2609.18297},
      archivePrefix={arXiv},
      primaryClass={math.DG},
      note = {arXiv:2609.18297}
}

@misc{Hua16,
      title={Harmonic maps of punctured surfaces to the hyperbolic plane}, 
      author={Andy C. Huang},
      year={2016},
      eprint={1605.07715},
      archivePrefix={arXiv},
      primaryClass={math.DG},
      note = {arXiv: 1605.07715}
}

@article{MV25,
  author = {{Mazzoli}, Filippo and {Viaggi}, Gabriele},
  title = {{\({\mathrm{SO}}_0(2,n+1)\)-maximal representations and hyperbolic surfaces}},
  journal = {Mem. Amer. Math. Soc.},
  volume = {313},
  year = {2025},
  number = {1587},
  pages = {v+119},
  issn = {0065-9266,1947-6221},
}

@misc{Via25,
      title={Geometric components of representation spaces via robust families of submanifolds}, 
      author={Gabriele Viaggi},
      year={2025},
      eprint={2508.14842},
      archivePrefix={arXiv},
      primaryClass={math.GT},
      note = {arXiv: 2508.14842}
}

@article{Tam19, 
    title = {Polynomial quadratic differentials on the complex plane and light-like polygons in the Einstein Universe},
    journal = {Adv. Math.}, 
    author = {Tamburelli, Andrea}, 
    Volume = {352}, 
    year = {2019}, 
    page = {483-515}, 
}

@misc{MNS26,
    author = {Maloni, Sara and Nguyen, Lam and Sagman, Nathaniel},
    title = {On harmonic maps with polynomial {Hopf} differential},
    year = {2026},
    note = {In preparation}
}

@article{SST23,
    author = {Seppi, Andrea and Smith, Graham and Toulisse, J\'er\'emy},
    title = {On complete maximal submanifolds in pseudohyperbolic space},
    year = {2023},
    journal = {J. Eur. Math. Soc. (to appear)}
}

\end{document}